\documentclass [11pt] {article}
\usepackage{amsfonts}
\usepackage{amssymb}
\usepackage{times}
\usepackage{color}
\usepackage{graphicx}

\newcommand{\Hom}{\mathrm{Hom}}

\newtheorem{lemma}{Lemma}[section]
\newtheorem{corollary}{Corollary}[section]
\newtheorem{theorem}{Theorem}

\newtheorem{remark}{Remark}[section]

\newtheorem{proposition}{Proposition}[section]
\newtheorem{definition}{Definition}[section]

\def\con{\star}

\def\<{\langle}
\def\>{\rangle}

\def\cl{{\rm cl}}

\def\to{\rightarrow}

\def\im{{\rm im}}

\def\bX{{\bf X}}

\def\bA{{\bf A}}

\def\bm{{\bf \mu}}

\begin{document}

\title{On higher order Fourier analysis}
\author{{\sc Bal\'azs Szegedy}}

\maketitle

\abstract{We develop a theory of higher order structures in compact abelian groups. In the frame of this theory we prove general inverse theorems and regularity lemmas for Gowers's uniformity norms. We put forward an algebraic interpretation of the notion ``higher order Fourier analysis'' in terms of continuous morphisms between structures called compact $k$-step nilspaces. As a byproduct of our results we obtain a new type of limit theory for functions on abelian groups in the spirit of the so-called graph limit theory. Our proofs are based on an exact (non-approximative) version of higher order Fourier analysis which appears on ultra product groups.}

\tableofcontents

\section{Introduction}

Higher order Fourier analysis is a notion which has many aspects and interpretations. The subject originates in a fundamental work by Gowers \cite{Gow},\cite{Gow2} in which he introduced a sequence of norms for functions on abelian groups and he used them to prove quantitative bounds for Szemer\'edi's theorem on arithmetic progressions \cite{Szem1} Since then many results were published towards a better understanding of the Gowers norms \cite{GrTao},\cite{GrTao2},\cite{GTZ},\cite{GowW},\cite{GowW2},\cite{GowW3},\cite{TZ},\cite{Sz1},\cite{Sz2},\cite{Sz3},\cite{Sz4},\cite{Sz5}
Common themes in all these works are the following four topics:

\medskip

\noindent{1.)~}{\it Inverse theorems for the Gowers norms.}

\noindent{2.)~}{\it Decompositions of functions into structured and random parts.}

\noindent{3.)~}{\it Counting structures in subsets and functions on abelian groups.}

\noindent{4.)~}{\it Connection to ergodic theory, nilmanifolds and nil sequences.}

\medskip

In the present paper we wish to contribute to all of these topics however we also put forward three other directions:

\medskip

\noindent{5.)~}{\it We develop an an algebraic interpretation of higher order Fourier analysis based on morphisms between structures that are generalizations of nilmanifolds.}

\noindent{6.)~}{\it We replace finite groups by arbitrary compact abelian groups.}

\noindent{7.)~}{\it We introduce limit objects for functions on abelian groups in the spirit of the graph limit theory.}

Important results on the fifth and sixth topics were also obtained by Host and Kra in the papers \cite{HKr2},\cite{HKr3}.
The paper \cite{HKr2} is the main motivation of \cite{NP} which is the corner stone of our approach.

\bigskip

\noindent{\bf Remark:}~~{\it Note that most the results in the present paper were obtained by the author in \cite{Sz1},\cite{Sz2},\cite{Sz3},\cite{Sz4}. However this paper together with \cite{NP} is a self contained account of the author's approach to higher order Fourier analysis. Many proofs are significantly different and more elementary than the discussion in the above four papers. 
The material of \cite{Sz5} is not covered by this paper and it will be a part of another sequence of papers in the topic.} 

\bigskip

To summarize the results in this paper we start with the definition of Gowers norms.
Let $f:A\rightarrow\mathbb{C}$ be a bounded measurable function on a compact abelian group $A$. Let $\Delta_t f$ be the function with $\Delta_tf(x)=f(x)\overline{f(x+t)}$.
With this notation
$$\|f\|_{U_k}=\Bigl(\int_{x,t_1,t_2,\dots,t_k\in A}\Delta_{t_1}\Delta_{t_2}\dots\Delta_{t_k}f(x)~d\mu^{k+1}\Bigr)^{2^{-k}}$$
where $\mu$ is the normalized Haar measure on $A$.
These norms satisfy the inequality $\|f\|_{U_k}\leq\|f\|_{U_{k+1}}$.
It is easy to verify that
\begin{equation}\label{u2norm}
\|f\|_{U_2}=\Bigl(\sum_{\chi\in\hat{A}}|\lambda_\chi|^4\Bigr)^{1/4}
\end{equation}
where $\lambda_\chi=(f,\chi)$ is the Fourier coefficient corresponding to the linear character $\chi$.
This formula explains the behaviour of the $U_2$ norm in terms of ordinary Fourier analysis.
However if $k\geq 3$, ordinary Fourier analysis does not seem to give a good understanding of the $U_k$ norm.

Small $U_2$ norm of a function $f$ with $|f|\leq 1$ is equivalent with the fact that $f$ is ``noise'' or ``quasi random'' from the ordinary Fourier analytic point of view. This means that all the Fourier coefficients have small absolute value. Such a noise however can have a higher order structure measured by one of the higher Gowers norms.
Isolating the structured part from the noise is a central topic in higher order Fourier analysis.   
In $k$-th order Fourier analysis a function $f$ is considered to be quasi random if $\|f\|_{U_{k+1}}$ is small.
As we increase $k$, this notion of noise becomes stronger and stronger and so more and more functions are considered to be structured. Our goal is to describe the structures that arise this way.

The prototype of a decomposition theorem into structured and quasi random parts is Szemer\'edi's famous regularity lemma for graphs \cite{Szem2}.
The regularity lemma together with an appropriate counting lemma is a fundamental tool in combinatorics.
It is natural to expect that a similar regularization corresponding to the $U_{k+1}$ norm is helpful in additive combinatorics.
One can state the graph regularity lemma as a decomposition theorem for functions of the form $f:V\times V\rightarrow\mathbb{C}$ with $\|f\|\leq 1$.
Roughly speaking it says that $f=f_s+f_e+r_r$ where $f_r$ has small cut norm, $f_e$ has small $L^1$ norm and $f_s$ is of bounded complexity. (All the previous norms are normalized to give $1$ for the constant $1$ function.) We say that $f_r$ has complexity $m$ if there is a partition of $V$ into $m$ almost equal parts such that $f_s(x,y)$ depends only on the partition sets containing $x$ and $y$. This can also be formulated in a more algebraic way. A complexity $m$ function on $V\times V$ is the composition of $\phi:V\times V\rightarrow [m]\times [m]$ ({\bf algebraic part}) with another function $f:[m]\times [m]\rightarrow\mathbb{C}$ ({\bf analytic part}) where $[m]$ is the set of first $m$ natural numbers and $\phi$ preserves the product structure in the sense that $\psi=g\times g$ for some map $g:V\rightarrow [m]$. 
In this language the requirement that the partition sets are of almost equal size translates to the condition that $\phi$ is close to be preserving the uniform measure. 

Based on this one can expect that there is a regularity lemma corresponding to the $U_{k+1}$ norm of a similar form.
This means that a bounded (measurable) function $f$ on a finite (or more generally on a compact) abelian group is decomposable as 
$f=f_s+f_e+f_r$ where $\|f_r\|_{U_{k+1}}$ is small, $\|f_e\|_1$ is small and $f_s$ can be obtained as the composition of $\phi:A\rightarrow N$ (algebraic part) and $g:N\rightarrow\mathbb{C}$ (analytic part) where $\phi$ is some kind of algebraic morphism preserving an appropriate structure on $A$.
The function $\phi$ would correspond to a regularity partition and $g$ would correspond to a function associating values with the partition sets.
The almost equality of the partition sets in Szemer\'edi's lemma should correspond to the fact that $\phi$ satisfies some approximative measure preserving property. 

We show that this optimistic picture is almost exactly true with some new additional features.
Quite interestingly, if $k>1$, to formulate the regularity lemma for the $U_{k+1}$ norm we need to introduce new structures called $k$-step nilspaces. It turns out that $k$-step nilspaces are forming a category and the morphisms are suitable for the purpose of regularization. 
Another interesting phenomenon is that geometry comes into the picture.
Topology and geometry does not play a direct role in stating the regularity lemma for graphs. (Note that a connection of Szemer\'edi's regularity lemma to topology was highlighted in \cite{LSz4}.) However in the abelian group case, even if we just regularize functions on finite abelian groups, compact geometric structures come up naturally as target spaces of the morphism $\phi$. To get a strong enough regularity lemma we will require that the values of the function $\phi:A\rightarrow N$ are so evenly distributed that we can basically say that $\phi$ (approximatly) reproduces the geometry of $N$ on the abelian group $A$. To guarantee that the composition $g\circ f$ respects this approximative geometry on $A$ we need to measure how much the function $g:N\rightarrow\mathbb{C}$ respects the geometry on $N$.
A possible way of doing it is to require that $g$ is continuous with bounded Lipschitz constant in some fixed metric on $N$. However the Lipschitz condition is not crucial in our approach. It can be replaced by almost any reasonable complexity notion. For example we can use an arbitrary ordering of an arbitrary countable $L^\infty$-dense set of continuous functions on $N$ and then we can require that $g$ is on this list with a bounded index.
  
\bigskip

To state our regularity lemma we will need the definition of nilspaces.
Nilspaces are common generalizations of abelian groups and nilmanifolds.
An abstract cube of dimension $n$ is the set $\{0,1\}^n$.
A cube of dimension $n$ in an abelian group $A$ is a function $f:\{0,1\}^n\rightarrow A$ which extends to an affine homomorphism (a homomorphism plus a translation) $f':\mathbb{Z}^n\rightarrow A$. Similarly, a morphism $\psi:\{0,1\}^n\rightarrow\{0,1\}^m$ between abstract cubes is a map which extends to an affine morphism from $\mathbb{Z}^n\rightarrow\mathbb{Z}^m$.

Roughly speaking, a nilspace is a structure in which cubes of every dimension are defined and they behave very similarly as cubes in abelian groups.
\begin{definition}[Nilspace axioms] A nilspace is a set $N$ and a collection $C^n(N)\subseteq N^{\{0,1\}^n}$ of functions (or cubes) of the form $f:\{0,1\}^n\rightarrow N$ such that the following axioms hold.
\begin{enumerate}
\item {\bf(Composition)} If $\psi:\{0,1\}^n\rightarrow\{0,1\}^m$ is a cube morphism and $f:\{0,1\}^m\rightarrow N$ is in $C^m(N)$ then the composition $\psi\circ f$ is in $C^n(N)$.
\item {\bf(Ergodictiry)} $C^1(N)=N^{\{0,1\}}$.
\item {\bf(Gluing)} If a map $f:\{0,1\}^n\setminus\{1^n\}\rightarrow N$ is in $C^{n-1}(N)$ restricted to each $n-1$ dimensional face containing $0^n$ then $f$ extends to the full cube as a map in $C^n(N)$.
\end{enumerate} 
\end{definition}

\bigskip

If $N$ is a nilspace and in the third axiom the extension is unique for $n=k+1$ then we say that $N$ is a $k$-step nilspace.
If a space $N$ satisfies the first axiom (but the last two are not required) then we say that $N$ is a {\bf cubespace}. 
A function $f:N_1\rightarrow N_2$ between two cubespaces is called a {\bf morphism} if $\phi\circ f$ is in $C^n(N_2)$ for every $n$ and function $\phi\in C^n(N_1)$.
The set of morphisms between $N_1$ and $N_2$ is denoted by $\Hom(N_1,N_2)$. With this notation $C^n(N)=\Hom(\{0,1\}^n,N)$.
We say that $N$ is a compact nilspace if $N$ has a compact, second countable, Hausdorff topology on it and $C^n(N)$ is a closed subset of $N^{\{0,1\}^n}$ for every $n$. 

The nilspace axiom system is a variant of the Host-Kra axiom system for parallelepiped structures \cite{HKr2}. In \cite{HKr2} the two step case is analyzed and it is proved that the structures are tied to two nilpotent groups. A systematic analysis of $k$-step nilspaces (with a special emphasis on the compact case) was carried out in \cite{NP}. It will be important that the notion of Haar measure can be generalized for compact nilspaces. It was proved in \cite{NP} that compact nilspaces are inverse limits of finite dimensional ones and the connected components of a finite dimensional compact nilspace are nilmanifolds with cubes defined through a given filtration on the nilpotent group. It is crucial that a $k$-step compact nilspace $N$ can be built up using $k$ compact abelian groups $A_1,A_2,\dots,A_k$ as structure groups in a $k$-fold iterated abelian group bundle. The nilspace $N$ is finite dimensional if and only if all the structure groups are finite dimensional or equivalently: the dual groups $\hat{A_1},\hat{A_2},\dots,\hat{A_n}$ are all finitely generated. 
It follows from the results in \cite{NP} that there are countably many finite dimensional $k$-step nilspaces up to isomorphism. An arbitrary ordering on them will be called a {\bf complexity notion}.

For every finite dimensional nilspace $N$ and natural number $n$ we fix a metrization of the weak convergence of probability measures on $C^n(N)$.
Let $M$ and $N$ be (at most) $k$-step compact nilspaces such that $N$ is finite dimensional. Let $\phi:M\rightarrow N$ be a continuous morphism and let us denote by $\phi_n:C^n(M)\rightarrow C^n(N)$ the map induced by $\phi$ using composition. The map $\phi$ is called $b$-balanced if the probability distribution of $\phi_n(x)$ for a random $x\in C^n(M)$ is at most $b$-far from the uniform distribution on $C^n(N)$ whenever $n\leq 1/b$.
Being well balanced expresses a very strong surjectivity property of morphisms which is for example useful in counting.

\begin{definition}[Nilspace-polynomials] Let $A$ be a compact abelian group. A function $f:A\rightarrow\mathbb{C}$ with $|f|\leq 1$ is called a $k$-degree, complexity $m$ and $b$-balanced nilspace-polynomial if
\begin{enumerate}
\item $f=\phi\circ g$ where $\phi:A\rightarrow N$ is a continuous morphism of $A$ into a finite dimensional compact nilspace $N$,
\item $N$ is of complexity at most $m$,
\item $\phi$ is $b$-balanced,
\item $g$ is continuous with Lipschitz constant $m$. 
\end{enumerate}
\end{definition}

Note, that (as it will turn out) a nilspace-polynomial on a cyclic group is polynomial nilsequence with an extra periodicity property.
Now we are ready to state the decomposition theorem.

\begin{theorem}[Regularization]\label{reglem} Let $k$ be a fixed number and $F:\mathbb{R}^+\times\mathbb{N}\rightarrow\mathbb{R^+}$ be an arbitrary function. Then for every $\epsilon>0$ there is a number $n=n(\epsilon,F)$ such that for every measurable function $f:A\rightarrow\mathbb{C}$ on a compact abelian group $A$ with $|f|\leq 1$ there is a decomposition $f=f_s+f_e+f_r$ and number $m\leq n$ such that the following conditions hold.
\begin{enumerate}
\item $f_s$ is a degree $k$, complexity $m$ and $F(\epsilon,m)$-balanced nilspace-polynomial,
\item $\|f_e\|_1\leq\epsilon$,
\item $\|f_r\|_{U_{k+1}}\leq F(\epsilon,m)$~,~$|f_r|\leq 1$ and $|(f_r,f_s)|~,~|(f_r,f_e)|\leq F(\epsilon,m)$.
\end{enumerate}
\end{theorem}

\begin{remark} The Gowers norms can also be defined for functions on $k$-step compact nilspaces. It makes sense to generalize our results from abelian groups to nilspaces. Almost all the proofs are essentially the same. This shows that the (algebraic part) of $k$-th order Fourier analysis deals with continuous functions between $k$-step nilspaces. 
\end{remark}

Note that various other conditions could be put on the list in theorem \ref{reglem}.
For example the proof shows that $f_s$ looks approximately like a projection of $f$ to a $\sigma$-algebra. This imposes strong restrictions on the value distribution of $f_s$ in terms of the value distribution of $f$. 

Theorem \ref{reglem} implies inverse theorems for the Gowers norms.
It says that if $\|f\|_{U_{k+1}}$ is separated from $0$ then it correlates with a bounded complexity $k$-degree nilspace polynomial $\phi$.
(We can also require the function $\phi$ to be arbitrary well balanced in terms of its complexity but we omit this from the statement to keep it simple.)

\begin{theorem}[General inverse theorem for $U_{k+1}$]\label{invthem} Let us fix a natural number $k$. For every $\epsilon>0$ there is a number $n$ such that if $\|f\|_{U_{k+1}}\geq\epsilon$ for some measurable function $f:A\rightarrow\mathbb{C}$ on a compact abelian group $A$ with $|f|\leq 1$ then $(f,g)\geq\epsilon^{2^k}/2$ for some nilspace polynomial $g$ of degree $k$ and complexity at most $n$.
\end{theorem}

Note that this inverse theorem is exact in the sense that if $f$ correlates with a bounded complexity nilspace polynomial then its Gowers norm is separated from $0$.

\bigskip

A strengthening of the decomposition theorem \ref{reglem} and inverse theorem \ref{invthem} deals with the situation when the abelian groups are from special families. 
For example we can restrict our attention to elementary abelian $p$-groups with a fixed prime $p$. Another interesting case is the set of cyclic groups or bounded rank abelian groups.
It also make sense to develop a theory for one particular infinite compact group like the circle group $\mathbb{R}/\mathbb{Z}$. (We will see in chapter \ref{chap:circle} that many features of higher order Fourier analysis become significantly simpler if we restrict it to the circle.) 
It turns out that in restricted families of groups we get restrictions on the structure groups of the nilspaces that we have to use in our decomposition theorem. 
To formulate these restrictions we need the next definition.

\begin{definition} Let $\mathfrak{A}$ be a family of compact abelian groups. We denote by $(\mathfrak{A})_k$ the set of finitely generated groups that arise as subgroups of $\hat{\bA}_k$ where $\bA$ is some ultra product of groups in $\mathfrak{A}$ and $\hat{\bA}_k$ is the $k$-th order dual group of $\bA$ in the sense of chapter \ref{chap:higherdual}. 
\end{definition}

The following statements about $(\mathfrak{A})_k$ follow from lemma \ref{duexp} and lemma \ref{charzero}.

\begin{enumerate}
\item {\bf (Bounded exponent and characteristic $p$)}~If $\mathfrak{A}$ is the set of finite groups of exponent $n$ then $(\mathfrak{A})_k=\mathfrak{A}$ for every $k$. In particular if $n=p$ prime then $\mathfrak{A}$ and $(\mathfrak{A})_k$ are just the collection of finite dimensional vector spaces over the field with $p$ elements.
\item {\bf (Bounded rank)}~If $\mathfrak{A}$ is the set of finite abelian groups of rank at most $d$ then $(\mathfrak{A})_1$ is the collection of finitely generated abelian groups whose torsion part has rank at most $d$. If $k\geq 2$ then $(\mathfrak{A})_k$ contains only free abelian groups. The case $d=1$ is the case of cyclic groups.
\item {\bf (Characteristic $0$)}~If $\mathfrak{A}$ is a family of finite abelian groups in which for every natural number $n$ there are only finitely many groups with order divisible by $n$ then $(\mathfrak{A})_k$ contains only free abelian groups for every $k$.
\item {\bf (Tori)}~If $\mathfrak{A}$ contains only tori $(\mathbb{R}/\mathbb{Z})^n$ then $(\mathfrak{A})_k$ contains only free abelian groups for every $k$.
\end{enumerate} 

\begin{definition} Let $\mathfrak{A}$ be a family of compact abelian groups. A $k$-step $\mathfrak{A}$-nilspace is a finite dimensional nilspace with structure groups $A_1,A_2,\dots,A_k$ such that $\hat{A_i}\in(\mathfrak{A})_i$ for every $1\leq i\leq k$.
A nilspace polynomial is called $\mathfrak{A}$-nilspace polynomial if the corresponding morphism goes into an $\mathfrak{A}$-nilspace.\end{definition}
Then we have the following.

\begin{theorem}[Regularization in special families]\label{restreg} Let $\mathfrak{A}$ be a set of compact abelian groups. Then Theorem \ref{reglem} restricted to functions on groups from $\mathfrak{A}$ is true with the stronger implication that the structured part $f_s$ is an $\mathfrak{A}$-nilspace polynomial.
\end{theorem} 

\begin{theorem}[Specialized inverse theorem for $U_{k+1}$]\label{restinv} Let $\mathfrak{A}$ be a set of compact abelian groups. Then for functions on groups in $\mathfrak{A}$ theorem \ref{invthem} holds with $\mathfrak{A}$-nilspace polynomials.
\end{theorem}

This theorem shows in particular that if $\mathfrak{A}$ is the set of abelian groups in which the order of every element divides a fixed number $n$ (called groups of exponent $n$) then $\mathfrak{A}$-nilspaces (used in the regularization) are finite and all the structure groups have exponent $n$. 

In the $0$ characteristic case $\mathfrak{A}$-nilspaces are $k$-step nilmanifold with a given filtration. This will help us to give a generalization of the Green-Tao-Ziegler theorem \cite{GTZ} for a multidimensional setting.
In the case of the circle group or more generally tori's, again we only get $k$-step nilmanifolds.

\bigskip
We highlight our results about counting and limit objects for function sequences.
Roughly speaking, counting deals with the density of given configurations in subsets or functions on compact abelian groups. We have two goals with counting. One is to show that our regularity lemma is well behaved with respect to counting and the second goal is
to show that function sequences in which the density of every fixed configuration converges have a nice limit object which is a measurable function on a nilspace. This fits well into the recently developed graph and hypergraph limit theories \cite{LSz1},\cite{BCLSV1},\cite{LSz3},\cite{LSz4},\cite{ESz}.
 
Counting in compact abelian groups has two different looking but equivalent interpretations. One is about evaluating certain integrals and the other is about the distribution of random samples from a function.
Let $f:A\rightarrow\mathbb{C}$ be a bounded function and the compact group $A$.
An integral of the form 
$$\int_{x,y,z\in A} f(x+y)f(x+z)f(y+z)~d\mu^3$$
can be interpreted as the triangle density in the weighted graph $M_{x,y}=f(x+y)$.
Based on this connection, evaluating such integrals can be called counting in $f$.
Note that one might be interested in more complicated integrals like this:
$\int f(x+y+z)^5\overline{f(x+y)}~d\mu^3$
where conjugations and various powers appear. 
It is clear that as long as the arguments are sums of different independent variables then all the above integrals can be obtained from knowing the seven dimensional distribution of
\begin{equation}\label{jointdis}
(f(x),f(y),f(z),f(x+y),f(x+z),f(y+z),f(x+y+z))\in\mathbb{C}^7
\end{equation}
where $x,y,z$ are randomly chosen elements form $A$ with respect to the Haar measure.
One can think of the above integrals as multi dimensional moments of the distribution in (\ref{jointdis}).
We will say that such a moment (or the integral itself) is simple if it does not contain higher powers. (We allow conjugation in simple moments.)
We will see that there is a slight, technical difference between dealing with simple moments and dealing with general moments.

Every moment can be represented as a colored (or weighted) hypergraph on the vertex set $\{1,2,\dots,n\}$ where $n$ is the number of variables and an edge $S\subseteq\{1,2,\dots,n\}$ represents the term $f(\sum_{i\in S}x_i)$ in the product.
The color of an edge tells the appropriate power and conjugation for the corresponding term.
The degree of a moment is the maximal size of an edge minus one in this hypergraph.
Let $\mathcal{M}$ denote the set of all simple moments and let $\mathcal{M}_k$ denote the collection of simple moments of degree at most $k$. 
We will denote by $D_n(f)$ the joint distribution of $\{f(\sum_{i\in S} x_i)\}_{S\subset[n]}$ where $[n]=\{1,2,\dots,n\}$.
It is a crucial fact that all the moments and the distributions $D_n$ can also be evaluated for functions on compact nilspaces with a distinguished element $0$. 
We call nilspaces with such an element ``rooted nilspaces''.
Let $N$ be a rooted nilspace. If we choose a random $n$-dimensional cube $c:\{0,1\}^n\rightarrow N$ in $C^n(N)$ with $f(0^n)=0$ then the joint distribution of the values $\{f(c(v))\}_{v\in\{0,1\}^n}$ gives the distribution $D_n(f)$. 

We continue with an interesting example.
Let $C=\mathbb{R}/\mathbb{Z}$ the circle group and let $\chi(x)=e^{x2\pi i}$ defined on $C$. The dual group of $C$ is the cyclic group generated by the linear character $\chi$.
Let us consider the function sequence $f_n=\chi+\chi^n$ for $n\in\mathbb{N}$. 
One can calculate that $\{f_n\}_{n=1}^\infty$ converges in the sense that for every  $k\in\mathbb{N}$ and $M\in\mathcal{M}_k$ the values $\{M(f_i)\}_{i=1}^\infty$ converege. In fact for every $n$ the sequence $\{D_n(f_i)\}_{i=1}^\infty$ is a convergent sequence of distributions. 
The natural question arises if there is a natural limit object for this function sequence. It turns out that there is no function on the circle which represents the limit however on the torus $C^2$ the function $f(x,y)=\chi(x)+\chi(y)$ has the property that $\lim_{i\to\infty}D_n(f_i)=D(f)$ holds for every $n$. It turns out that more complicated limit objects can arise. For example there are function sequences on the circle (or on finite abelian groups) which converge to functions on the Heisenberg nilmanifold. 
Our next theorems provide limit objects for convergent function sequences.

\begin{theorem}[Limit object I.]\label{simplim} Assume that $\{f_i\}_{i=1}^\infty$ is a sequence of uniformly bounded measurable functions on the compact abelian groups $\{A_i\}_{i=1}^\infty$. Then if $\lim_{i\to\infty} M(f_i)$ exists for every $M\in\mathcal{M}$ then there is a measurable function (limit object) $g:N\rightarrow\mathbb{C}$ on a compact rooted nilspace $N$ such that $M(g)=\lim_{i\to\infty}M(f_i)$ for $M\in\mathcal{M}$.
\end{theorem}

\begin{corollary}[Limit object II.]\label{simplimcor} Let $k$ be a fixed natural number. Assume that $\{f_i\}_{i=1}^\infty$ is a sequence of uniformly bounded measurable functions on the compact abelian groups $\{A_i\}_{i=1}^\infty$. Then if $\lim_{i\to\infty} M(f_i)$ exists for every $M\in\mathcal{M}_k$ then there is a measurable function (limit object) $g:N\rightarrow\mathbb{C}$ on a compact $k$-step rooted nilspace $N$ such that $M(g)=\lim_{i\to\infty}M(f_i)$ for $M\in\mathcal{M}_k$.
\end{corollary}

Let $\mathcal{P}_r$ denote the space of Borel probability distributions supported on the set $\{x:|x|\leq r\}$ in $\mathbb{C}$.

\begin{theorem}[Limit object III.]\label{genlim} Assume that $\{f_i\}_{i=1}^\infty$ is a sequence of functions with $|f_i|\leq r$ on the compact abelian groups $\{A_i\}_{i=1}^\infty$. Then if $\lim_{i\to\infty} D_k(f_i)$ exists for every $k\in\mathbb{N}$ then there is a measurable function (limit object) $g:N\rightarrow\mathcal{P}_r$ on a compact rooted nilspace $N$ such that $D_k(g)=\lim_{i\to\infty}D_k(f_i)$ for $k\in\mathbb{N}$.
\end{theorem}
Let us observe that theorem \ref{genlim} implies the other two.

We devote the last part of the introduction to our main method and the simple to state theorem \ref{ultreg} on ultra product groups which implies almost everything in this paper.
Let $\{A_i\}_{i=1}^\infty$ be a sequence of compact abelian groups and let $\bA$ be their ultra product.
Our strategy is to develop a theory for the Gowers norms on $\bA$ and then by indirect arguments we translate it back to compact groups.
First of all note that $U_{k+1}$ is only a semi norm on $\bA$. We prove in this paper that there is a unique maximal $\sigma$-algebra $\mathcal{F}_k$ on $\bA$ such that $U_{k+1}$ is a norm on $L^\infty(\mathcal{F}_k)$ and $L^\infty(\mathcal{F}_k)$ is orthogonal to every function whose $U_{k+1}$ norm is zero.
It follows that every function $f\in L^\infty(\bA)$ has a unique decomposition as $f_s+f_r$ where $\|f_r\|_{U_{k+1}}=0$ and $f_s$ is measurable in $\mathcal{F}_k$. 
This shows that on the ultra product $\bA$ it is simple to separate the structured part of $f$ from the random part. 

The question remains how to describe the structured part in a meaningful way.
It turns out that to understand this we need to go beyond measure theory and use topology. 
Note that the reason for this is not that the groups $A_i$ are already topological. Even if $\{A_i\}_{i=1}^\infty$ is a sequence of finite groups, topology will come into the picture in the same way.  

We will make use of the fact that $\bA$ has a natural $\sigma$-topology on it. A $\sigma$-topology is a weakening of ordinary topology where only countable unions of open sets are required to be open.
The structure of $\mathcal{F}_1$, which is tied to ordinary Fourier analysis, sheds light on how topology comes into the picture. It turns out that $\mathcal{F}_1$ can be characterized as the smallest $\sigma$-algebra in which all the continuous surjective homomorphisms $\phi:\bA\rightarrow G$ are measurable where $G$ is a compact abelian group.
In other words the ordinary topological space $G$ appears as a factor of the $\sigma$-topology on $\bA$.
The next theorem explains how nilspaces enter the whole topic:

\begin{theorem}[Characterization of $\mathcal{F}_{k+1}$.] The $\sigma$-algebra $\mathcal{F}_k$ is the smallest $\sigma$-algebra in which all continuous morphisms $\phi:\bA\rightarrow N$ are measurable where $N$ is a compact $k$-step nilspace.
\end{theorem}

Another, stronger formulation of the previous theorem says that every separable $\sigma$-algebra in $\mathcal{F}_k$ is measurable in a  $k$-step compact, Hausdorff nilspace factor of $\bA$.
We will see later that we can also require a very strong measure preserving property for the nilspace factors $\phi:\bA\rightarrow N$. This will also be crucial in the proofs.
As a corollary we have a very simple regularity lemma on the ultra product group $\bA$.

\begin{theorem}[Ultra product regularity lemma]\label{ultreg} Let us fix a natural number $k$. Let $f\in L^\infty(\bA)$ be a function. Then there is a unique (orthogonal) decomposition $f=f_s+f_r$ such that $\|f_r\|_{U_{k+1}}=0$ and $f_s$ is measurable in a $k$-step compact nilspace factor of $\bA$.
\end{theorem}

\subsection{Nilmanifolds as nilspaces}\label{nilasnil}

In this chapter we outline the connection between nilmanifolds and nilspaces (for this topic see also \cite{NP}).
Let $F$ be a nilpotent Lie-group and let $F=F_0\geq F_1\geq F_2\geq...\geq F_k=\{1\}$ be a filtration in $F$ i.e. $[F,F_i]\leq F_{i+1}$ holds for every $0\leq i\leq k-1$. (Note that the existence of such a filtration implies that $F$ is at most $k$-nilpotent.)

Let us define the following cubic structure on $F$ which depends on the filtration $\{F_i\}_{i=1}^k$.
A map $f:\{0,1\}^n\rightarrow F$ is in $C^n(F)$ if it can be obtained from the constant $1$ map in a finite process where in each step we choose a natural number $1\leq i\leq k$ and an element $x\in N_i$ and then we multiply the value of $f$ on an $i+1$ co-dimensional face of $\{0,1\}^n$ by $x$. 
It is not hard to show that $F$ together with this cubic structure is a $k$-step nilspace.

An alternative way of defining the same cubic structure gives the sets $C^n(N)$ directly through an equation system in $F^{2^n}$. For $n\in\mathbb{N}$ let $g_n:\{0,1\}^n\rightarrow\{1,2,\dots,2^n\}$ be the ordering such that $g_1(0)=1,g(1)=2$ and if $n>1$ then $g_n((v,0))=g_{n-1}(v)$ and $g_n((v,1))=2^n+1-g_{n-1}(v)$ where $v\in\{0,1\}^{n-1}$.

\begin{definition} Let $G$ be a group and $f:\{0,1\}^n\rightarrow G$ be some function. We say that $f$ satisfies the Gray code property if $$\prod_{i=1}^{2^n}f(g_n^{-1}(i))^{(-1)^i}=1.$$
\end{definition}

Let $C^n(F)$ be the collection of functions $f:\{0,1\}^n\rightarrow F$ such that if $d\leq k+1$ and $g:\{0,1\}^d\rightarrow \{0,1\}^n$ is a morphism then the map $f\circ g$ satisfies the Gray code property modulo $G_{d-1}$.  
Note that it is enough to check the condition for morphisms $g:\{0,1\}^d\rightarrow\{0,1\}^n$ that are injective and the image is a $d$ dimensional face of $\{0,1\}^n$. 
This definition of $C^n(F)$ shows that $C^n(F)\subset F^{\{0,1\}^n}$ is a closed set. 

It will be crucial to describe morphisms from abelian groups (as nilspaces) into these nilspaces.
We will use the next two definitions by Leibman \cite{Lei},\cite{Lei2}.

\begin{definition}[Polynomial map between groups] A map $\phi$ of a group $G$ to a group $F$ is said to be polynomial of degree $k$ if it trivializes after $k+1$ consecutive applications of the operator $D_h,~h\in G$ defined by $D_h\phi(g)=\phi(g)^{-1}\phi(gh).$
\end{definition}

\begin{definition} Let $F$ be a $k$-nilpotent group with filtration $\mathcal{V}=\{F_i\}_{i=0}^k$ with $F=F_0$, $F_{i+1}\subseteq F_i$, $F_k=\{1\}$ and $[F_i,F]\subseteq F_{i+1}$ if $i<k$. A map $\phi:G\rightarrow F$ is a $\mathcal{V}$-polynomial if $\phi$ modulo $F_i$ is a polynomial of degree $i$.
\end{definition}

It is proved in \cite{Lei2} that $\mathcal{V}$ polynomials are closed under multiplication. 
Let $A$ be an abelian group. The second definition of the nilspace structure on $F$ shows that nilspace morphisms from $A$ to $F$ are exactly the $\mathcal{V}$ polynomials. 

\medskip

Let $\Gamma\leq F$ be a discrete co-compact subgroup in $F$. 
We denote by $M$ the (left) coset space $\{g\Gamma\}_{g\in F}$. Manifolds of the form $M$ are called nil-manifolds. 
Let $\pi:F\rightarrow M$ denote the projection $\pi(g)=g\Gamma$. We define the cubic structure on $M$ by $C^n(M)=\{\pi\circ c|c\in C^n(F)\}$. A simple calculation shows that $N$ together with this cubic structure is a $k$-step nilspace. However to guarantee that $C^n(M)\subset M^{2^n}$ is a closed set we also need that $F_i\cap\Gamma$ is co-compact in $F_i$ for every $0\leq i\leq k$. If this holds we will say that $\Gamma$ is co-compact in $\{F_i\}_{i=0}^k$.
Note that the structure groups of $M$ are the abelian groups $A_i=F_{i-1}\Gamma/F_i\Gamma$. 
The next theorem is one of the main results in \cite{NP}.

\begin{theorem}\label{finitedimnil} Let $M$ be a compact $k$-step nilspace such that the structure groups $\{A_i\}_{i=1}^k$ are all finite dimensional tori. Then there is a nilpotent Lie-group $F$ with filtration $\mathcal{V}=\{F_i\}_{i=0}^k$ and a discrete subgrup $\Gamma\leq F$ which is co-compact in $\mathcal{V}$ such that $M$ is (topologically) isomorphic to the nilspace corresponding to $(\mathcal{V},\Gamma)$.
\end{theorem}

\subsection{A multidimensional generalization of the Green-Tao-Ziegler theorem.}\label{cyclic}

Our goal in this chapter is to relate nilspace polynomials on cyclic groups to Leibman type polynomials.
As a consequence we will obtain a new proof of the inverse theorem by Green, Tao and Ziegler for cyclic groups.
We will use the notation from chapter \ref{nilasnil}

\begin{lemma}\label{polyab} Let $A$ be an abelian group. Then the set of at most degree $k$ polynomials form $\mathbb{Z}^n$ to $A$ is generated by the functions of the form
\begin{equation}\label{abpoly}
f(x_1,x_2,\dots,x_n)=a\prod_{i=1}^n{{x_i}\choose{n_i}}
\end{equation}
where $a\in A$ and $\sum_{i=1}^n n_i\leq k$.
(We use additive notation here)
\end{lemma}

\begin{proof} We go by induction on $k$. The case $k=0$ is trivial. Assume that it is true for $k-1$. Let $g_1,g_2,\dots,g_n$ be the generators of $\mathbb{Z}^n$. If $\phi:\mathbb{Z}^n\rightarrow A$ is a polynomial map then $$\omega(y_1,y_2,\dots,y_k)=D_{y_1}D_{y_2}\dots D_{y_k}\phi$$ is a symmetric $k$-linear form on $\mathbb{Z}^n$. We claim that there is map $\phi'$ which is generated by the functions in (\ref{abpoly}) and whose $k$-linear form is equal to $\omega$. Let $f:\otimes^k(\mathbb{Z}^n)\rightarrow A$ be a homomorphism representing $\omega$. Then $f$ is generated by homomorphisms $h$ such that $h(g_{j_1}\otimes g_{j_2}\otimes\dots g_{j_k})=a$ for some indices $j_1,j_2,\dots,j_k$ (and any ordering of them) and take $0$ on any other tensor products of generators. It is enough to represent such an $h$ by a function of the form (\ref{abpoly}). 
It is easy to see that if $n_i$ is the multiplicity of $i$ among the indices $\{j_r\}_{r=1}^k$ then (\ref{abpoly}) gives a polynomial whose multi linear form is represented by $h$. 

The difference $\phi-\phi'$ has a trivial $k$-linear form which shows that it is a $k-1$ dimensional polynomial and then we use induction the generate $\phi-\phi'$.
\end{proof}

\begin{lemma}\label{felemelo} Let $\phi:\mathbb{Z}^n\rightarrow M$ be a morphism. Then there is a lift $\psi:\mathbb{Z}^n\rightarrow F$ such that $\psi$ is a $\mathcal{V}$-polynomial and $\psi$ composed with the projection $F\rightarrow M$ is equal to $\phi$.
\end{lemma}

\begin{proof} Using induction of $j$ we show the statement for maps whose image is in $F_{k-j}H$.
If $j=0$ then $\phi$ is a constant map and then the statement is trivial. Assume that we have the statement for $j-1$ and assume that the image of $\phi$ is in $F_{k-j}H$.
The cube preserving property of $\phi$ shows that $\phi$ composed with the factor map $F_{k-j}H\rightarrow F_{k-j}H/F_{k-j+1}H=A_{k-j}$ is a degree $k-j+1$ polynomial map $\phi_2$ of $\mathbb{Z}^n$ into the abelian group $A_{k-j}$. 

We have from lemma \ref{polyab} that using multiplicative notation
\begin{equation}\label{polfor}
\phi_2(x_1,x_2,\dots,x_n)=\prod_{t=1}^m a_t^{f_t(x_1,x_2,\dots,x_n)}
\end{equation}
where $a_t\in A_{k-j}$ and $f_t$ is an integer valued polynomial of degree at most $k-j+1$ for every $t$. Let us choose elements $b_1,b_2,\dots,b_m$ in $F_{k-j}$ such that their images in $A_{k-j}$ are $a_1,a_2,\dots,a_m$.
Let us define the function $\alpha:\mathbb{Z}\rightarrow F$ given by the formula (\ref{polfor}) when $a_t$ is replaced by $b_t$. The map $\alpha$ is a $\mathcal{V}$-polynomial.

Since $\phi$ maps to the left cosets of 
$H$ it makes sens to multiply $\phi$ by $\alpha^{-1}$ from the left. It is easy to see that the new map $\gamma=\alpha^{-1}\phi$ is a morphism of $\mathbb{Z}$ to $F_{k-j+1}H$ and thus by induction it can be lifted to a $\mathcal{V}$ polynomial $\delta$. Then we have that $\alpha\delta$ is a lift of $\phi$ to a polynomial map. 
\end{proof}

\begin{corollary} If $A$ is a finite abelian group and $f:A\rightarrow M$ is a morphism then for every homomorphism $\beta:\mathbb{Z}^n\rightarrow A$ there is a degree $k$ polynomial map $\phi:\mathbb{Z}^n\rightarrow F$ such that $\phi$ composed with the factor map $F\rightarrow M$ is the same as $\beta$ composed with $f$.
\end{corollary}

\begin{definition}[$d$-dimensional polynomial nilsequence] Assume that $F$ is a connected $k$-nilpotent Lie group with filtration $\mathcal{V}$ and $\Gamma$ is a co-compact subgroup of $F$. Assume that $M$ is the left coset space of $\Gamma$ in $F$. Then a map $h:\mathbb{Z}^d\rightarrow\mathbb{C}$ is called a $d$-dimensional polynomial nilsequence (corresponding to $M$) if there is a polynomial map $\phi:\mathbb{Z}^d\rightarrow F$ of degree $k$ and a continuous Lipschitz function $g:M\rightarrow\mathbb{C}$ such that $h$ is the composition of $\phi$, the projection $F\rightarrow M$ and $g$. 
The complexity of such a nilsequence is measured by the maximum of $c$ and the complexity of $N$.
\end{definition}

Let $\mathfrak{A}$ be a $0$-characteristic family of abelian groups. Then all the structure groups of $\mathfrak{A}$-nilspaces are tori. From theorem \ref{finitedimnil} and theorem \ref{restinv} we obtain the following consequence.
\begin{theorem}[polynomial nilsequence inverse theorem]\label{PNIT} Let $\mathfrak{A}$ be a $0$ characteristic family of finite abelian groups. Let us fix a natural number $k$. For every $\epsilon>0$ there is a number $n$ such that if $\|f\|_{U_{k+1}}\geq\epsilon$ for some measurable $f:A\rightarrow\mathbb{C}$ with $|f|\leq 1$ on $A\in\mathfrak{A}$ with $d$-generators $a_1,a_2,\dots,a_d$ then $(f,g)\geq\epsilon^{2^k}/2$ such that $g(n_1a_1+n_2a_2+\dots+n_da_d)=h(n_1,n_2,\dots,n_d)$ for some $d$-dimensional polynomial nilsequence $h$ of complexity at most $n$.
\end{theorem}

Note that the above theorem implies an interesting periodicity since the defining equation of $g$ is true for every $d$-tuple $n_1,n_2,\dots,n_d$ of integers. 
Using theorem \ref{PNIT} we obtain the Green-Tao-Ziegler inverse theorems for functions $f:[N]\rightarrow\mathbb{C}$ with $|f|\leq 1$. Their point of view is that if we put the interval $[N]$ into a large enough cyclic group (say of size $m>N2^{k+1}$) then the normalized version of $\|f\|_{U_{k+1}}$ does not depend on the choice of $m$. The proper normalization is to divide with the $U_{k+1}$-norm of the characteristic function on $1_{[N]}$. 

To use theorem \ref{PNIT} in this situation we need to make sure that $m$ is not too big and that it has only large prime divisors. This can be done by choosing a prime between $N2^{k+1}$ and $N2^{k+2}$.
Then we can apply theorem \ref{PNIT} for the family of cyclic groups of prime order which is clearly a $0$ characteristic family.
What we directly get is that $f$ correlates with a bounded complexity polynomial nil-sequence of degree $k$.
This seems to be weaker then the Green-Tao-Ziegler theorem because they obtain the correlation with a linear nil-sequence. However in the appendix of \cite{GTZ} it is pointed out that the two versions are equivalent.

Form theorem \ref{PNIT} we can also obtain a $d$-dimensional inverse theorem for functions of the form $f:[N]^d\rightarrow\mathbb{C}$ with $|f|\leq 1$. Here we use the family of $d$-th direct powers of cyclic groups with prime order. 

\begin{theorem}[Multi dimensional inverse theorem] Let us fix two natural numbers $d,k>0$. Then for every $\epsilon>0$ there is a number $n$ such that for every function $f:[N]^d\rightarrow\mathbb{C}$ with $\|f\|_{U_{k+1}}\geq\epsilon$ there is a $d$-dimensional polynomial nil-sequence $h:\mathbb{Z}^d\rightarrow\mathbb{C}$ of complexity at most $n$ and degree $k$ such that $(f,h)\geq \epsilon^{2^k}/2$. 
\end{theorem}

Note that $(f,h)$ is the scalar product normalized as as $(f,h)=N^{-d}\sum_{v\in [N]^d}(f(v)\overline{h(v)})$.

\subsection{An example involving the Heisenberg group}\label{heis}

In this chapter we discuss an example which highlights a difference between the nilseqence approach used in \cite{GTZ} and the nilspace-polynomial approach used in the present paper. 

Let $e(x)=e^{x2\pi i}$. For an integer $1<t<m$ we introduce the function $f:\mathbb{Z}_m\rightarrow\mathbb{C}$ defined by $f(k)=\lambda^{k^2}$ where $\lambda=e(t/m^2)$ and $k=0,1,2,\dots,m-1$. Note that this function does not ``wrap around'' nicely like a more simple quadratic function of the form $k\mapsto\epsilon^{k^2}$ where $\epsilon$ is an $m$-th root of unity. This means that to define $f$ we need to choose explicit integers to represent the residue classes modulo $m$.
On the other hand it can be seen that $\|f\|_{U_3}$ is uniformly separated from $0$ so it has some quadratic structure.

In the nilsequence approach this function is not essentially different from the case of $k\mapsto\epsilon^{k^2}$.
However in our approach we are more sensitive about the periodicity issue since we want to establish $f$ through a very rigid algebraic morphism which uses the full group structure of $\mathbb{Z}_m$.
We will show that the quadratic structure of $f$ is tied to a nilspace morphism $\phi$ which maps $\mathbb{Z}_m$ into the Heisenberg nilmanifold.

The Heisenberg group $H$ is the group of three by three upper uni-triangular matrices with real entries.
Let $\Gamma\subset H$ be the set of integer matrices in $H$. It can be seen that $\Gamma$ is a co-compact subgroup. The left coset space $N=\{g\Gamma|g\in H\}$ of $\Gamma$ in $H$ is the Heisenberg nilmanifold.
Let $M\in H$ be the following matrix:
\[ \left( \begin{array}{ccc}
1 & 2t/m & t/m^2\\
0 & 1 & 1/m \\
0 & 0 & 1 \end{array} \right)\]
then

\[M^k= \left( \begin{array}{ccc}
1 & 2kt/m & k^2t/m^2\\
0 & 1 & k/m \\
0 & 0 & 1 \end{array} \right)\]
In particular $M^m$ is an integer matrix.
This implies that the map $\tau:k\rightarrow M^k\Gamma$ defines a periodic morphism from $\mathbb{Z}$ to $N$.
Since the period length is $m$, it defines a morphism $\phi:\mathbb{Z}_m\rightarrow N$.

Let $D$ be the set of elements in $H$ in which all entries are between $0$ and $1$. The set $D$ is a fundamental domain for $\Gamma$. We can define a function on $N$ by representing it on the fundamental domain.
Let $g:D\rightarrow\mathbb{C}$ be the function $A\rightarrow e(A_{1,3})$ where $A_{1,3}$ is the upper-right corner of the matrix $A$. 

We compute $g(\tau(k))=g(M^k\Gamma)$ by multiplying $M^k$ back into the fundamental domain $D$.
Since $g(M^k\Gamma)$ is periodic we can assume that $0\leq k<m$.
Let us multiply $M^k$ from the right by

\[ \left( \begin{array}{ccc}
1 & -\lfloor 2kt/m\rfloor & -\lfloor k^2t/m^2\rfloor\\
0 & 1 & 0\\
0 & 0 & 1 \end{array} \right)\]

We get 

\[ \left( \begin{array}{ccc}
1 & \{ 2kt/m\} & \{k^2t/m^2\}\\
0 & 1 & k/m\\
0 & 0 & 1 \end{array} \right)\in D\]

So the value of $g$ on $\tau(k)$ is $e(\{k^2t/m^2\})=e(k^2t/m^2)$.

\subsection{Higher order Fourier analysis on the cirlce}\label{chap:circle}

In this chapter we sketch a consequence of our results when specialized to the circle grouop $C=\mathbb{R}/\mathbb{Z}$.
Since the circle falls in to the $0$-characteristic case, theorem \ref{finitedimnil} shows that higher order Fourier analysis on the circle deals with continuous morphisms from $C$ to nilspaces that arise from nilmanifolds. We show that such morphisms arise from one parameter subgroups in nilpotent Lie-groups which periodically intersect a co-compact subgroup. It follows that Gowers norms have rather aesthetical inverse theorems on the circle.

Let us use the notation from chapter \ref{nilasnil}. We denote by $M$ the nilspace on $F/\Gamma$ corresponding to the filtration $\mathcal{V}$. Recall that $\pi:F\rightarrow M$ is the natural projection. We prove the following theorem. 

\begin{theorem}\label{circhom} If $\phi:C\rightarrow M$ is a continuous nilspace morphism with $\phi(0)=\pi(1)$. Then there is a group homomorphism $f:\mathbb{R}\rightarrow F$ such that $f(1)\in\Gamma$ and $\phi(x\mathbb{Z})=f(x)\Gamma$.
\end{theorem}

Note that theorem \ref{circhom} implies that contintinuous morphisms from $C$ to $M$ (which are normalized in the way that $\phi(0)=\pi(1)$) are in a one to one correspondence with the elments of $\Gamma$. If $g\in\Gamma$ then basic Lie-group theory shows that there is a unique one parameter subgroup $f:\mathbb{R}\rightarrow F$ with $f(1)=g$. 
A surprising consequence of theorem \ref{circhom} is that filtrations and polynomial maps become irrelevant for the description of morphisms of the circle. They can be characterized through classical group homomorphisms.

\begin{lemma}\label{lin} Let $f:C\rightarrow C$ be a continuous polynomial map. Then $f$ is linear. (Consequently we obtain that any continuous polynomial map $f:C\rightarrow C^n$ is linear.)
\end{lemma}

\begin{proof} In this proof we will think of $C$ as the complex unit circle. Assume by contradiction that $f$ is not linear. Then by repeatedly applying operators $D_h$ to $f$ we can get a non-linear quadratic function.
This means that it is enough to get a contradiction if $f$ is quadratic.
In this case $D_h(f)$ is a linear map that depends continuously on $h\in C$ in the $L_2$ norm. On the other hand  $D_h(f)$ is a linear character times a complex number from the unit circle. The characters are orthogonal to each other and so the character corresponding to $D_h(f)$ has to be the same for every $h$. This is only possible if it is the trivial character but then $f$ is linear.
\end{proof}

\medskip

\noindent{\it Proof of theorem \ref{circhom}}~~We go by induction on $k$. If $k=0$ then the statement is trivial.
Assume that we have the statement for $k-1\geq 0$. By factoring out with the central subgroup $F_{k-1}\cap\Gamma$ we can assume that $F_{k-1}\cap\Gamma$ is trivial. Thus $F_{k-1}$ is a torus. Let $M_{k-1}$ denote the nilspace $(N/F_{k-1})/\Gamma$ and let $\varrho:M\rightarrow M_{k-1}$ denote the natural projection. Elements of $M_{k-1}$ are orbits of the action of $F_{k-1}$ on $M$. Since $\varrho$ is a continuous nilspace morphism we have that $\varrho\circ\phi$ is a continuous morphism. By our induction hypothesis there is a one parameter subgroup $f_1:\mathbb{R}\rightarrow F/F_{k-1}$ representing $\varrho\circ\phi$. Assume that $f_1(1)=gF_{k-1}$ for some $g\in\Gamma$. Let $f_2:\mathbb{R}\rightarrow F$ denote the unique one parameter subgroup with $f_2(1)=g$. We have that $f_2(x)F_{k-1}=f_1(x)$ holds for every $x\in\mathbb{R}$. We have that $\pi(f_2(x))$ is in the same $F_{k-1}$ orbit as $\phi(x)$. Let $h:\mathbb{R}\rightarrow F_{k-1}$ denote the unique function with $\phi(x)=h(x)\pi(f_2(x))$. It is easy to see from basic nilspace theory (see theorem \ref{bundec}) that $h$ is a polynomial map so by lemma \ref{lin} it is a homomorphism. From $\pi(f_2(1))=\pi(1)$ and $\phi(1)=\pi(1)$ we have that $h(1)=0$. Since $F_{k-1}$ is in the center of $F$ we have that $f(x)=\pi(x)f_2(x)$ is a homomorphism which satisfies the required conditions.

\medskip 
 
We are ready to describe functions $f_s:C\rightarrow\mathbb{C}$ that are "structured" with respect to the $U_{k+1}$ norm.
Let $F$ be a $k$-nipontent Lie-group with a co-compact subgroup $\Gamma$ and let $h:F/\Gamma\rightarrow\mathbb{C}$ be a continuous function with $|h|\leq 1$. Let $g\in\Gamma$ be a fixed element and let $\tau:\mathbb{R}\rightarrow F$ be the unique one parameter subgroup with $\tau(1)=g$. Then we obtain a continuous function $f_s=h\circ\pi\circ\tau$ on $C=\mathbb{R}/\mathbb{Z}$. 
Note that $f_s$ can be regarded as a continuous periodic nil-sequence. 
Roughly speaking, if the complexity of $F/\Gamma$ and the Lipschitz constant of $h$ are both bounded then $f_s$ is a structured function in $k$-th order Fourier analysis. Theorem \ref{restinv} shows that if $\|f\|_{U_{k+1}}$ is separated from $0$ then $f$ correlates with such a structured function $f_s$. 

Using this we also get an ``interval'' version of the Green-Tao-Ziegler theorem for the $U_{k+1}([0,1])$ norm.
Functions on the interval $[0,1]$ can be represented in a large enough Cyclic group say $\mathbb{R}/2^{k+1}\mathbb{Z}$.
We obtain that if $f:[0,1]\rightarrow\mathbb{C}$ is a measurable function with $|f|\leq 1$ and $\|f\|_{U_{k+1}}$ is separated from $0$ then $f$ correlates with a continuous bounded complexity nilsequence.

\section{Compact abelian groups, Gowers norms and nilspaces}

\subsection{$\sigma$-algebras of probability spaces}

Let $(\Omega,\mathcal{A},\mu)$ be a probability space. We will use the standard notation for normed spaces obtained from $\Omega$. For $L^p(\Omega,\mathcal{A},\mu)$ we use the short hand notation $L^p(\Omega)$ or $L^p(\mathcal{A})$. Since we never consider two different probability measures on the same $\sigma$-algebra the meaning of these abbreviations is always clear from the context. We denote by $L^p_u$ the unit balls in these normed spaces. In this paper we only use the values $p=2,\infty$. 

We will often work with sub $\sigma$-algebras of $\mathcal{A}$ in a fixed probability space $(\Omega,\mathcal{A},\mu)$. For two sub $\sigma$-algebras $\mathcal{B},\mathcal{C}$ in $\mathcal{A}$ we denote by $\mathcal{B}\vee\mathcal{C}$ the $\sigma$-algebra generated by $\mathcal{B}$ and $\mathcal{C}$. The expression $\mathcal{B}\wedge\mathcal{C}$  denotes the intersection of $\mathcal{B}$ and $\mathcal{C}$. According to our definition a set $S$ is in $\mathcal{B}\wedge\mathcal{C}$ if there are measurable sets $B\in\mathcal{B}$ and $C\in\mathcal{C}$ such that $\bm(S\triangle A)=\bm(S\triangle B)=0$. 
If $\mathcal{B}$ is a sub $\sigma$-algebra in $\mathcal{A}$ and $f:\Omega_2\rightarrow \Omega$ is a measure preserving map for some probability space $\Omega_2$ then we denote by $\mathcal{B}\circ f$ the $\sigma$-algebra $\{f^{-1}(S)|S\in\mathcal{B}\}$.
Two sub $\sigma$-algebras $\mathcal{B}$ and $\mathcal{C}$ in $\mathcal{A}$ are called {\bf conditionally independent} if $\mathbb{E}(\mathbb{E}(f|\mathcal{C})|\mathcal{B})=\mathbb{E}(\mathbb{E}(f|\mathcal{B})|\mathcal{C})=\mathbb{E}(f|\mathcal{B}\wedge\mathcal{C})$ holds for an arbitrary bounded measurable function $f$. To prove that $\mathcal{B}$ and $\mathcal{C}$ are conditionally independent it is enough to check that $\mathbb{E}(f|\mathcal{C})=0$ whenever $f$ is measurable in $\mathcal{B}$ and $\mathbb{E}(f|\mathcal{B}\wedge\mathcal{C})=0$.

\begin{definition} Let $\{\mathcal{B}_i\}_{i=1}^n$ be a collection of sub $\sigma$-algebras in $\mathcal{A}$. Then we denote by $\mathcal{R}(\{\mathcal{B}_i\}_{i=1}^n)$ the set of functions of the form $f=\prod_{i=1}^nf_i$  where $f_i\in L^\infty_u(\mathcal{B}_i)$.
\end{definition}

We will use the following classical fact from measure theory.

\begin{lemma}\label{siggen} Let $\{\mathcal{B}_i\}_{i=1}^n$ be a collection of sub $\sigma$-algebras in $\mathcal{A}$ and let $\mathcal{B}=\bigvee_{i=1}^n\mathcal{B}_i$ be the $\sigma$-algebra generated by them.  Then every function in $L^2(\mathcal{B})$ can be approximated with an arbitrary precision in $L^2$ by a finite linear combination of functions in $\mathcal{R}=\mathcal{R}(\{\mathcal{B}_i\}_{i=1}^\infty)$.
\end{lemma}

A $\sigma$-algebra $\mathcal{B}\subseteq\mathcal{A}$ is {\bf separable} if there is a countable subset $S\subset\mathcal{B}$ such that for every $\epsilon>0$ and $H\in\mathcal{B}$ there is a set $T\in S$ such that $\mu(H\triangle T)\leq\epsilon$. Every $\sigma$-algebra which is generated by countable many sets is separable.
We will use the following basic fact.

\begin{lemma}\label{sepsiggen} Let $\{\mathcal{B}_i\}_{i=1}^n$ be a collection of sub $\sigma$-algebras in $\mathcal{A}$ and let $\mathcal{C}\subseteq\bigvee_{i=1}^n\mathcal{B}_i$ be a separable $\sigma$-algebra. Then there are separable $\sigma$-algebras $\mathcal{B}'_i\subseteq\mathcal{B}$ for $1\leq i\leq n$ such that $\mathcal{C}\subseteq\bigvee_{i=1}^n\mathcal{B}'_i$.
\end{lemma}

We will need the following lemma on conditional independence.

\begin{lemma}\label{siggen2} Let $\mathcal{B}$ and $\mathcal{C}$ be two conditionally independent $\sigma$-algebras and let $\mathcal{B}_1$ be a sub $\sigma$-algebra of $\mathcal{B}$.
Then $(\mathcal{C}\vee\mathcal{B}_1)\wedge\mathcal{B}=(\mathcal{C}\wedge\mathcal{B})\vee\mathcal{B}_1$.
\end{lemma}

\begin{proof}
It is trivial that $(\mathcal{C}\vee\mathcal{B}_1)\wedge\mathcal{B}\supseteq(\mathcal{C}\wedge\mathcal{B})\vee\mathcal{B}_1$.
To see the other containment let $f\in L^\infty((\mathcal{C}\vee\mathcal{B}_1)\wedge\mathcal{B})$.
Using that $f\in L^\infty(\mathcal{C}\vee\mathcal{B}_1)$ we have by lemma \ref{siggen} that for an arbitrary small $\epsilon>0$ there is an approximation of $f$ in $L_2$ of the form $f'=\sum_{i=1}^n c_ib_i$
where $c_i\in L^\infty(\mathcal{C})$ and $b_i\in L^\infty(\mathcal{B}_1)$. Since $\|f-f'\|_2\leq\epsilon$ and $f\in L^\infty(\mathcal{B})$ we have that $$\epsilon\geq\|\mathbb{E}(f|\mathcal{B})-\mathbb{E}(f'|\mathcal{B})\|_2=\|f-\sum_{i=1}^n \mathbb{E}(c_i|\mathcal{B})b_i\|_2.$$
Using conditional independence we have that $E(c_i|\mathcal{B})$ is measurable in $\mathcal{B}\wedge\mathcal{C}$ for every $i$ and the whole sum is measurable in $(\mathcal{C}\wedge\mathcal{B})\vee\mathcal{B}_1$.
Using it for every $\epsilon$ the proof is complete.
\end{proof}

\subsection{Couplings of probability spaces}

Let $I$ be a finite index set and $\mathcal{U}=\{(\Omega_i,\mathfrak{S}_i,\mu_i)\}_{i\in I}$ be a system of probability spaces. If all the probability spaces in $\mathcal{U}$ are separable the we say that $\mathcal{U}$ is separable.
A {\bf coupling } of $\mathcal{U}$ is a probability space $(\Omega,\mathfrak{S},\mu)$ together with measure preserving transformations $\{\psi_i:\Omega\rightarrow\Omega_i\}_{i\in I}.$ This means that for every $i\in I$ and set $S\in\mathfrak{S}_i$ we have $\mu(\psi_i^{-1}(S))=\mu_i(S)$.

\begin{definition}\label{coupeq} Let $\{\psi_i:\Omega\rightarrow\Omega_i\}_{i\in I}$ and $\{\psi_i^*:\Omega^*\rightarrow\Omega_i\}_{i\in I}$ be two couplings of $\mathcal{U}$ on the spaces $(\Omega,\mathfrak{S},\mu)$ and $(\Omega^*,\mathfrak{S}^*,\mu^*)$.
We say that these two couplings are {\bf equivalent} if for every system of sets $\{S_i\in\mathfrak{S}_i\}_{i\in I}$ we have
\begin{equation}\label{cupconst}
\mu\Bigl(\bigcap_{i\in I}\psi_i^{-1}(S_i)\Bigr)=\mu^*\Bigl(\bigcap_{i\in I}{\psi_i^*}^{-1}(S_i)\Bigr).
\end{equation}
\end{definition}

\begin{remark} It is not hard to show that in the previous definition the equivalence of the two couplings imply that if $F$ is an arbitrary $m$ variable set formula (using intersection, union and complement) , $\{M_j\in\mathfrak{S}_{a_j}\}_{j=1}^m$ is a system a events, $M$ is the value of $F$ on $\{\psi_{a_i}^{-1}(M_i)\}_{i=1}^m$ and $M^*$ is the value of $F$ on $\{{\psi^*_{a_i}}^{-1}(M_i)\}_{i=1}^m$ then $\mu(M)=\mu^*(M^*)$.
In other words the two couplings are equivalent if they can't be distinguished using probabilities of events formulated with events in $\mathcal{U}$ in a fixed way.
\end{remark}

\begin{definition} Let $\prod_{i\in I}\mathfrak{S}_i$ denote the set of vectors $(S_i)_{i\in I}$ where $S_i\in\mathfrak{S}_i$. Let
\begin{equation}\label{couprep}
f:\prod_{i\in I}\mathfrak{S}_i\rightarrow [0,1]
\end{equation}
be a function. We say that $f$ represents the (equivalence class) of a coupling $\Psi=\{\psi_i:\Omega\rightarrow\Omega_i\}_{i\in I}$ if for $S=(S_i)_{i\in I}$ the value of $f(S)$ is equal to the left hand side of (\ref{cupconst}). 
\end{definition}

The next lemma gives a simple answers to the following question: {\it For which functions $f$ is there a coupling $\Psi$ such that $f$ represents $\Psi$?}

\begin{lemma}\label{couprepl} A function $f$ of the form (\ref{couprep}) represents some coupling of $\mathcal{U}$ if and only if the following two conditions hold.
\begin{enumerate}
\item For every $j\in I$ and $A\in\mathfrak{S}_j$ if $S=(S_i)_{i\in I}$ denotes the vector with $S_i=\Omega_i$ for $i\neq j$ and $S_j=A$  then $f(S)=\mu_j(A)$.
\item $f$ is additive in every coordinate. This means that if in $S=(S_i)_{i\in I}$ the set $S_j$ is the disjoint union of $A$ and $B$ then $f(S)=f(T)+S(U)$ where $T$ (resp. $U$) is obtained from $S$ by replacing $S_j$ with $A$ (resp. $B$).
\end{enumerate}  
\end{lemma} 

Let ${\rm coup}(\mathcal{U})$ denote the set of equivalence classes of the possible couplings of the system $\mathcal{U}$.
Lemma \ref{couprepl} says that elements of ${\rm coup}(\mathcal{U})$ are in a one to one correspondence with functions of the form (\ref{couprep}) which satisfy the two algebraic conditions of the lemma. 
We list a few basic concepts related to couplings.

\bigskip

\noindent{\bf Topology:}~We say that $\{\Psi_i\}_{i=1}^\infty$ is a convergent sequence in ${\rm coup}(\mathcal{U})$ if the representing functions $\{f_i\}_{i=1}^\infty$ converge for every fixed element $S\in\prod_{i\in I}\mathfrak{S}_i$. It is clear by lemma \ref{couprepl} that the pointwise limit of $\{f_i\}_{i=1}^\infty$ also represents a coupling. If the coupling is separable the we obtain a compact Hausdorff topological structure on ${\rm coup}(\mathcal{U})$. 

\bigskip 

\noindent{\bf Convexity:}~Let $\nu$ be any Borel probability measure on ${\rm coup}(\mathcal{U})$. Then the function $f$ defined by
$$f(S)=\int_{C\in {\rm coup}(\mathcal{U})} f_C(S)~d\nu$$ represents a coupling where $f_C$ is the function representing $C$. This shows that ${\rm coup}(\mathcal{U})$ is a convex set in the topological sense.

\bigskip

\noindent{\bf Complete dependence:}~We say that the coupling $\{\psi_i:\Omega\rightarrow\Omega_i\}_{i\in I}$ is completely dependent if for every $j\in I$ we have that $\mathfrak{S}_j\circ\psi_j=\vee_{i\neq j}\mathfrak{S}_i\circ\psi_i$. It is easy to see that this property depends only on the equivalence class of the coupling.

\bigskip

\noindent{\bf self-couplings and their sub-couplings:~}~Throughout this paper we will mostly study couplings when the probability spaces in $\mathcal{U}$ are all identical to a fixed space $(X,\mathcal{A},\nu)$. A coupling of such a system $\mathcal{U}$ is a system of measure preserving maps $\{\psi_i:\Omega\rightarrow X\}_{i\in I}$. Let us denote by ${\rm coup}(X,I)$ the set of self-couplings of $I$ copies of $X$. If $\phi:J\rightarrow I$ is a map then it induces a continuous map $\hat{\phi}:{\rm coup}(X,I)\rightarrow{\rm coup}(X,J)$ by $\hat{\phi}(\{\psi_i\}_{i\in I})=\{\psi_{\phi(j)}\}_{j\in J}$. If $\phi$ is injective and
 $C\in{\rm coup}(X,I)$ then we call $\hat{\phi}(C)$ the sub-couplig of $C$ corresponding to $\phi$.

\bigskip

\noindent{\bf The multi linear form $\xi$:}~Let $G=\{g_i:\Omega_i\rightarrow\mathbb{C}\}_{i\in I}$ be a system of bounded measurable functions. Let $C=\{\psi_i:\Omega\rightarrow\Omega_i\}_{i\in I}$ be a coupling of $\mathcal{U}$. 
Then $\xi$ is defined by
\begin{equation}\label{cupconst2}
\xi(C,G):=\mathbb{E}_{x\in\Omega}\Bigl(\prod_{i\in I} g_i(\psi_i(x))\Bigr).
\end{equation}
Notice that if $g_i$ is the characteristic function of a set $S_i\in\mathfrak{S}_i$ for every $i$ then the value in (\ref{cupconst2}) is equal to the value (on the left hand side) in (\ref{cupconst}) . This shows that $\xi$ determines the equivalence class of the coupling $C$.

\bigskip

\noindent{\bf Factor coupling and independence over a factor:}~Let $\mathfrak{S}_i'\subset\mathfrak{S}_i$ be a sub $\sigma$-algebras for $i\in I$. The coupling $\{\psi_i:\Omega\rightarrow\Omega_i\}_{i\in I}$ can be restricted to a coupling of the probability spaces $\{\Omega_i,\mathfrak{S}'_i,\mu_i\}_{i\in I}$. This will be called a factor coupling of the original one. We will say that the original coupling is independet over the factor given by $\{\mathfrak{S}_i'\}_{i\in I}$ if for every $j\in I$ and system of bounded measurable functions $G=\{g_i:\Omega_i\rightarrow\mathbb{C}\}_{i\in I}$ with $\mathbb{E}(g_j|\mathfrak{S}_j')=0$ we have $\xi(C,G)=0$. Equivalently, if $G=\{g_i:\Omega_i\rightarrow\mathbb{C}\}_{i\in I}$ is a system of bounded measurable functions and $G'=\{\mathbb{E}(g_i|\mathfrak{S}_i')\}_{i\in I}$ then $\xi(C,G)=\xi(C,G')$. This shows that if a coupling $C$ is independent over a certain factor $C'$, then the multi linear form $\xi(C,G)$ and thus $C$ itself is uniquely determined by $C'$.

\bigskip

The basic properties of $\xi$ are summarized in the next lemma.

\begin{lemma}\label{xiprop} The function $\xi$ satisfies the following properties.
\begin{enumerate}
\item For every $i\in I$ we have $\xi(C,G)\leq\|g_i\|_1\prod_{j\neq i}\|g_j\|_\infty$,
\item $\xi(G,C)$ is linear in each component $g_i$,
\item If $F=\{f_i:\Omega_i\rightarrow\mathbb{C}\}_{i\in I}$ is a system of measurable functions then
\begin{equation}\label{smallchange}
|\xi(G,C)-\xi(F,C)|\leq \Bigl(\sum_{i\in I}\|f_i-g_i\|_1\Bigr)\prod_{i=1}^n\max(\|f_i\|_\infty,\|g_i\|_\infty).
\end{equation}
\item for a fixed $G$ the value $\xi(C,G)$ depends only on the equivalence class of the coupling $C$,
\item for a fixed $G$ the function $C\rightarrow\xi(C,G)$ is continuous on ${\rm coup}(\mathcal{U})$.
\item If for every $i$ the space $\Omega_i$ is a compact, Hausdorff space and $\mathfrak{S}_i$ is the Borel $\sigma$-algebra then the functions of the form $C\rightarrow\xi(C,G)$ where $G$ is a system of continuous functions generate the topology on ${\rm coup}(\mathcal{U})$.
\end{enumerate}
\end{lemma}

\begin{proof} The first and second properties are trivial from the definition.
The third property follows from the first two by replacing each $g_i$ by $f_i$ in $n$ consecutive steps. 
To see the fourth property observe that by the second property, if in $F=\{f_i:\Omega_i\rightarrow\mathbb{C}\}_{i\in I}$ every function is a step function then the function $\xi(F,C)$ is a linear combination of numbers of the form (\ref{cupconst}) and so $\xi(F,C)$ depends only on the equivalence class of $C$. Furthermore the function $\tau:C\rightarrow\xi(F,C)$ is a linear combination of continuous functions and so it is continuous. In the general case we can approximate every $g_i$ be a step function $f_i$ with $\|f_i\|_\infty\leq\|g_i\|_\infty$ such that $\|f_i-g_i\|_1\leq\epsilon$. Let $C'$ be an equivalent coupling with $C$. Then by (\ref{smallchange}) we have that both $|\xi(G,C)-\xi(F,C)|$ and $|\xi(G,C')-\xi(F,C')|$ are at most $|I|\epsilon\prod_{i\in I}\|g_i\|_\infty$. Since this is true for every $\epsilon>0$ and $\xi(F,C)=\xi(F,C')$ we have that $\xi(G,C)=\xi(G,C')$. The same argument shows that the function $\psi:C\rightarrow\xi(G,C)$ can be arbitrarily well approxiamted in $L^\infty$ by a continuous function $\tau$ and so $\psi$ is continuous. 
The last property follows from inequality (\ref{smallchange}) and from the fact that every $\{0,1\}$ valued measurable function on $\Omega_i$ can be approximated with an arbitrary precision in $L^1$ by a continuous function of absolute value at most $1$.
\end{proof}


\subsection{Abelian groups, cubes and cubic couplings}\label{cubes}

Let $A$ be an abelian group. If $A$ is compact, Hausdorff and second countable then we use the short hand notion {\bf compact} for $A$. Compact abelian groups admit a unique, shift invariant probability measure called {\bf Haar measure}.
An affine version of $A$ is a set $X$ such that $A$ acts transitively and freely (fix point free) on $X$. This means that for every pair $x,y\in X$ there is a unique element $a\in A$ such that $x^a=y$. We can interpret $a$ as $y-x$, however $y+x$ does not have a natural interpretation. Affine abelian groups are very similar to abelian groups. If we fix an element in $x\in X$ then there is a natural bijection between $X$ and $A$ given by $a\longleftrightarrow x^a$. An {\bf affine homomorphism} from an abelian group $A$ to another abelian group $B$ is the composition of an ordinary homomorphism with a shift on $B$. Affine homomorphisms can be naturally defined between affine abelian groups making them a category. 

\medskip

An {\bf abstract cube} of dimension $n$ is a set of the form $\{0,1\}^n$ (or more generally $\{0,1\}^I$ where $I$ is a set of size $n$). We denote by $0$ the all $0$ vector in $\{0,1\}^n$. A function $\phi:\{0,1\}^a\rightarrow\{0,1\}^b$ is called a cube morphism if it extends to an affine homomorphism $\phi':\mathbb{Z}^a\rightarrow\mathbb{Z}^b$. 
Cube morphisms have a combinatorial description. They are maps $\phi:\{0,1\}^a\rightarrow\{0,1\}^b$ such that each coordinate function of $\phi(x_1,x_2,\dots,x_a)$ is one of $1$, $0$, $x_i$ and $1-x_i$ for some $1\leq i\leq a$. 
An $n$-dimensional cube (or more precisely the morphism of a cube) in an abelian group $A$ is a map $c:\{0,1\}^n\rightarrow A$ which extends to an affine homomorphism $c':\mathbb{Z}^n\rightarrow A$.
Cubes can also be described through the formula
$$c(e_1,e_2,\dots,e_n)=x+\sum_{i=1}^n t_ie_i$$
where $x,t_1,t_2,\dots,t_n$ are elements in $A$. Finally 
a map $c:\{0,1\}^n\rightarrow A$ is a cube if and only if for every morphism $\phi:\{0,1\}^2\rightarrow\{0,1\}^n$ we have that $$c(\phi(0,0))-c(\phi(1,0))-c(\phi(0,1))+c(\phi(1,1))=0.$$

Let $Q=\{0,1\}^n$. We denote by $\hom(Q,A)$ the set of morphisms of $Q$ into $A$. With respect to pont wise addition on $Q$ the set $\hom(Q,A)$ is an abelian group. Since every $n$-dimensional cube is uniquely determined by $x,t_1,t_2,\dots,t_n$ in the above formula we have that $\hom(Q,A)$ is isomorphic to the direct power $A^{n+1}$.  
If $S$ is a subset in $Q$ and $f:S\rightarrow A$ is an arbitrary function then we denote by $\hom_f(Q,A)$ the set of maps $c\in\hom(Q,A)$ such that restriction of $c$ to $S$ is $f$.
Note that $\hom_f(Q,A)$ may be empty if $f$ does not extend to a morphism of the full cube. 
If $\hom_f(Q,A)$ is not empty then it is a coset of the group $\hom_g(Q,A)$ where $g$ is the identically $0$ function on $S$.
In other words $\hom_f(Q,A)$ is an affine version of $\hom_g(Q,A)$.
If $A$ is compact then $\hom_g(Q,A)$ is also compact and so using its Haar measure we get a unique $\hom_g(Q,A)$ invariant probability space structure on $\hom_f(Q,A)$.

When it doesn't lead to confusion we will use the short hand notation $C^n(A)$ for $\hom(Q,A)$ and $C^n_f(A)$ for $\Hom_f(Q,A)$. If $S=\{0\}\subset Q$ and $f:S\rightarrow A$ is given by $f(0)=x$ then we use the short hand notation $C_x^n(A)$ for $\hom_f(Q,A)$.   
It is clear that $C_0^n(A)$ is a subgroup of $C^n(A)$ which is isomorphic to $A^n$. If $x\neq 0$ then $C_x^n(A)$ is an affine version of $C_0^n(A)$. 

It is a crucial idea in this paper to consider $C^n(A)$ (resp. $C^n_x(A)$) as a coupling of $2^n$ (resp. $2^n-1$) copies of $A$. For every $v\in\{0,1\}^n$ we define the map $\psi_v:C^n(A)\rightarrow A$ by $\psi_v(c)=c(v)$. These maps are surjective homomorphisms between compact abelian groups and so they are all measure preserving. Ovserve that if $v\neq 0$ then the restriction of $\psi_v$ to $C^n_x$ is a measure preserving map from $C^n_x$ to $A$.
The system of maps $\Psi^n=\{\psi_v\}_{v\in\{0,1\}^n}$ on $C^n(A)$ is a self coupling $A$ with index set $\{0,1\}^n$. Let $K_n=\{0,1\}^n\setminus\{0\}$ and let $\Psi^n_x$ be the restriction of the function system $\{\psi_v\}_{v\in K_n}$ to the probability space $C^n_x(A)$. Then $\Psi^n_x$ is a self coupling of $A$ with index set $K_n$.
It is very important to note that $\Psi^n_x$ is not a sub-coupling of $\Psi^n$ and it will depend on the choice of $x$.
Let $G=\{g_v\}_{v\in\{0,1\}^n}$ and $F=\{f_v\}_{v\in K_n}$ be systems of bounded measurable functions on $A$. Then the values of both $\xi(G,\Psi^n)$ and $\xi(F,\Psi^n_x)$ are crucial in this paper. 
The following formulas follow directly from the definitions.
\begin{equation}\label{ginner} \xi(G,\Psi^n)=\mathbb{E}_{x,t_1,t_2,\dots,t_n}\prod_{v\in\{0,1\}^n}g_v(x+\sum_{i=1}^n v_it_i),
\end{equation}
\begin{equation}\label{corner}
\xi(F,\Psi_x^n)=\mathbb{E}_{t_1,t_2,\dots,t_n}\prod_{v\in K_n}f_v(x+\sum_{i=1}^n v_it_i).
\end{equation} 

It turns out that certain calculations work out a bit nicer if we put conjugations on the terms in (\ref{ginner}) and (\ref{corner}) whose indices $v$ have an odd number of $1$'s. This motivates the next definitions.
Let us use the convention that if $t_v$ is a complex valued term which depends on an element $v\in\{0,1\}^n$ then $t_v^\con$ is $t_v$ if $v$ has an even number of $1$'s and is the conjugate of $t_v$ if $v$ has an odd number of $1$'s.

\begin{definition} Let $G=\{g_v\}_{v\in\{0,1\}^n}$ be $\{f_v\}_{v\in K_n}$ be function systems and $G^\con=\{g_v^\con\}_{v\in\{0,1\}^n}$ and $F^\con=\{f_v^\con\}_{v\in K_n}$ be their conjugated versions.
Then the {\bf Gowers inner product} $(G)$ of $G$ is defined by
$$(G)=\xi(G^\con,\Psi^n),$$
and the {\bf corner convolution} $[F]$ of $F$ is defined by
$$[F](x)=\xi(F^\con,\Psi^n_x).$$
By abusing the notation we will also define the convolution $[G]$. Let $G'$ be the function system obtained from $G$ by ignoring $g_0$. Then $[G]:=[G']$. We introduce the notations $$(G)^\times=\prod_{v\in\{0,1\}^n}g_v^\con\circ\psi_v~~~~~{\it and}~~~~~[F]^\times=\prod_{v\in K_n}f_v^\con\circ\psi_v.$$
If the function system $[F]$ (resp. $[G]$) is constant such that each member is equal to the same function $f$ (resp. $g$) then we use the short hand notations $$(f)_n=(F)~~,~~[g]_n=[G]~~,~~(f)_n^\times=(F)^\times~~,~~[g]_n^\times=[G]^\times.$$

\end{definition}

Convolutions of the form $[F]$ in the above definition will be also called $n$-th order convolutions if we need to emphasize the value $n$. 
Let us observe that with the above notation we have the following equations.
\begin{equation}\label{coincon}
(G)=([G],\overline{g_0})=([G]^\times,\overline{g_0}\circ\psi_0),
\end{equation}

$$\mathbb{E}((G)^\times)=(G),$$

\begin{equation}\label{rankcon1}
\mathbb{E}([F]^\times|\psi_0)=[F]\circ\psi_0,
\end{equation}

\begin{equation}\label{rankcon2}
\mathbb{E}_{y\in C^k_0(A)}[F]^\times(z+y)=[F](\psi_0(z)).
\end{equation}
Note that (\ref{rankcon1}) and (\ref{rankcon2}) are the same equations written in a different form.
An easy way of seeing (\ref{rankcon1}) and (\ref{rankcon2}) is to write the elements of $C^k(A)$ as vectors $z=(x,t_1,t_2,\dots,t_k)$ as described in chapter \ref{cubes}. In this coordinate system $C^k_0(A)$ is the set of vectors of the form $(0,t_1,t_2,\dots,t_k)$. Then $\psi_0(z)=x$ and (\ref{corner}) shows that the value $[G](x)$ is the average of $g$ on the coset of $C^k_0(A)$ containing $z$.

\begin{remark}\label{convat} Let $w\in\{0,1\}^n$, $K=\{0,1\}^n\setminus\{w\}$. Let $F=\{f_v\}_{v\in K}$ be a function system. One can define the convolution $[F]$ in a similar way as above since our setup does not distinguish the $0$ vector in $\{0,1\}^n$. Let $\alpha:\{0,1\}^n\rightarrow\{0,1\}^n$ be an automorphism with $\alpha(0)=w$. Then define $[F]$ as the convolution of the function system $\{f_{\alpha(v)}\}_{v\in K_n}$. It is clear that it does not depend on the choice of $\alpha$.
\end{remark}

\begin{lemma} Let $A$ be a compact abelian group and $n$ be a natural number. Then the map $\gamma:x\rightarrow\Psi^n_x$ is a continuous map from $A$ to ${\rm coup}(\mathcal{U}_0)$.
\end{lemma}

\begin{proof} According to the last point in lemma \ref{xiprop} the topology on ${\rm coup}(\mathcal{U}_0)$ is generated by functions of the form $C\rightarrow \xi(F,C)$ where $F=\{f_v\}_{v\in K_n}$ is a system of continuous functions. This means that it is enough to check the continuity of the composition of $\gamma$ with such functions. This composition is the function $h:x\rightarrow\xi(F,\Psi^n_x)$. The formula (\ref{corner}) shows that the continuity of the functions $f_v$ imply the continuity of $h$.
\end{proof}

The previous lemma and the fourth point in lemma \ref{xiprop} imply the following corollary.

\begin{corollary} The function $[F]$ is continuous for an arbitrary system $F=\{f_v\}_{v\in K_n}$ of bounded measurable functions.
\end{corollary}

\subsection{Sub-couplings of cubic couplings}

Let $S\subset\{0,1\}^n$ be an arbitrary subset and $f:S\rightarrow A$ be an arbitrary function into a compact abelian group $A$. 
Let $\phi:\{0,1\}^m\rightarrow\{0,1\}^n$ be a morphism and let $g$ denote the function $f\circ\phi$ on the set $\phi^{-1}(S)$. Then we denote by $\hat{\phi}:C^n_f(A)\rightarrow C^m_g(A)$ the map defined by $\hat{\phi}(c)=c\circ\phi$. The natural question arises: {\it Under what conditions is the map $\hat{\phi}$ measure preserving?}
First of all notice that $\hat{\phi}$ is a continuous morphism between affine compact abelian groups and thus $\hat{\phi}$ is measure preserving if and only if it is surjective. 
The next lemma connects surjectivity of $\hat{\phi}$ with equivalence of couplings.

\begin{lemma}\label{cupis} Assume that for every $v\in\{0,1\}^n\setminus S$ the map $\psi_v:C^n_f(A)\rightarrow A$ is surjective and that $\hat{\phi}$ is surjective. Then the coupling $\{\psi_v\}_{v\in H}$ on $C^m_g(A)$ is equivalent with the coupling $\{\psi_{\phi(v)}\}_{v\in H}$ on $C^n_f(A)$ where $H=\{0,1\}^m\setminus\phi^{-1}(S)$.
\end{lemma}

\begin{proof} The statement is clear from the facts that $\psi_{\phi(v)}=\psi_v\circ\hat{\phi}$ for every $v$ in $H$ and that $\hat{\phi}$ is measure preserving. 
\end{proof}

\begin{lemma}\label{cupis1} Let $0\leq k\leq n$ be integers and $S\subset\{0,1\}^n$ be the $k$-dimensional face in $\{0,1\}^n$ which consists of all vectors with $0$ in the last $n-k$ coordinates. Let $\tau:\{0,1\}^n\rightarrow\{0,1\}^{n-k}$ be the projection to the last $n-k$ coordinates. Assume that $f:S\rightarrow A$ is in $C^k(A)$ and $\phi:\{0,1\}^m\rightarrow\{0,1\}^n$ is a morphism such that $\tau\circ\phi$ is injective. Then $\hat{\phi}$ is surjective.
\end{lemma}

\begin{proof} Let $\tau'$ be the projection of $\{0,1\}^n$ to the first $k$ coordinates and let $z$ be the identically zero function on $S$. Since $C^n_f(A)=f\circ\tau'+C^n_z(A)$ we can assume without loss of generality that $f$ is identically $0$. It is enough to prove that $\hat{\phi}\circ\hat{\tau}:C^{n-k}_0(A)\rightarrow C_g^m(A)$ is surjective. This reduces the problem to the case when $k=0$. In this case $S$ is the $0$ vector and $\phi$ is an injective morphism.

By using appropriate automorphisms of $\{0,1\}^m$ and coordinate permutation of $\{0,1\}^n$ we can assume without loss of generality  (using the injectivity of $\phi$) that $\phi(v_1,v_2,\dots,v_m)_i=v_i$ if $1\leq i\leq m$ (this can be obtained from the combinatorial description of morphisms). 
We distinguish between two cases. In the first case $0$ is not in the range of $\phi$ and in the second case $\phi(0)=0$. 

In the first case, since the image of $\phi$ does not contain the zero vector we can assume that either $\phi(v)_{m+1}=1$ or $\phi(v)_{m+1}=1-v_1$. Let $\tau_2:\{0,1\}^n\rightarrow\{0,1\}^{m+1}$ be the projection to the
 first $m+1$ coordinates.  
 It is enough to show that $\hat{\phi}\circ\hat{\tau_2}$ is surjective. Using our parametrization, if $\phi(v)_{m+1}=1$ then for  $t=(0,t_1,t_2,\dots,t_{m+1})$ in $C^{m+1}_0(A)$ we have that $\hat{\phi}(\hat{\tau}_2(t))=(t_{m+1},t_{m+1}+t_1,\dots,t_{m+1}+t_m)$ and if $\phi(v)_{m+1}=1-v_1$ the $\hat{\phi}(\hat{\tau}_2(t))=(t_{m+1},t_1,t_{m+1}+t_2,\dots,t_{m+1}+t_m)$. Both are surjective.
 
In the second case we denote by $\tau_2:\{0,1\}^n\rightarrow\{0,1\}^m$ the projection to the first $m$ coordinates. The map $\hat{\tau_2}:C^m_0(A)\rightarrow C^n_0(A)$ composed with $\hat{\phi}:C^n_0(A)\rightarrow C_0^m(A)$ is obviously bijective which completes the proof.
\end{proof}

\subsection{Gowers norms and corner convolutions}\label{chap:gc}

Let $A$ be a compact abelian group.
If $f:A\rightarrow\mathbb{C}$ then we define the function $\Delta_t f$ by $(\Delta_t f)(x)=f(x)\overline{f(x+t)}$.
The Gowers norm $\|f\|_{U_n}$ is defined by
\begin{equation}\label{gowersnorm}
\|f\|_{U_n}^{2^n}=\mathbb{E}_{x,t_1,t_2,\dots,t_n}\Delta_{t_1,t_2,\dots,t_n}f(x)=(f)_n
\end{equation}
for $f\in L^\infty(A)$.

The so-called Gowers-Cauchy-Schwartz inequality says that if $F=\{f_v\}_{v\in\{0,1\}^n}$ is a system of bounded measurable functions then 
\begin{equation}\label{GCS}
|(F)|\leq\prod_{v\in\{0,1\}^n}\|f_v\|_{U_n}.
\end{equation}

\medskip

We continue with a basic trick which makes calculations with $(F)$ and $[F]$ easier.
Let $i\in [n]$ and let $Q\subset\{0,1\}^n$ be the set of vectors with $0$ in the $i$-th coordinate. Let $w\in\{0,1\}^n$ be the vector with $1$ at the $i$-th coordinate and $0$ everywhere else. For $t\in A$ we introduce $\delta_{i,t} F$ as the function system $\{f_v(x)\overline{f_{v+w}(x+t})\}_{v\in Q}$. If $F$ is a function system parametrized by $K_n$ then we define $\delta_{i,t} F$ by the previous formula such that $Q$ is repleced by $Q\setminus\{0\}$.
Then we have the following two equations 
\begin{equation}\label{dimred}
(F)=\mathbb{E}_t((\delta_{i,t} F))~~~~{\rm and}~~~~[F](x)=\mathbb{E}_t(\overline{f}_w(x+t)\delta_{i,t}[F](x)).
\end{equation}
The equations in (\ref{dimred}) are useful because they reduce the dimension $n$ in the calculations and thus they can be used in proofs with inductions on $n$. The next lemma is an example for this.  

\medskip

\begin{lemma}\label{cornineq} Let $F=\{f_v\}_{v\in K_n}$ be a system of bounded measurable functions on $A$. Then for every $j\in[n]$ we have that
$$|[F](x)|\leq\prod_{v\in K_n,v_j=0}\|f_v\|_\infty\prod_{v\in K_n,v_j=1}\|f_v\|_{U_n}.$$
\end{lemma}

\begin{proof} If $n=1$ then the statement is true with equality. If $n>1$ then by induction we assume that it is true for $n-1$. Without loss of generality (using symmetry) we can assume that $j\neq n$. 
We have that $\delta_{n,t} F=\{g_v^t\}_{v\in K_{n-1}}$ where $g_v^t$ is the function $y\mapsto\overline{f}_{(v,1)}(y+t)f_{(v,0)}(y)$. Let $w=(0,0,\dots,0,1)\in\{0,1\}^n$.
By (\ref{dimred}), induction, Cauchy-Schwartz inequality and using the fact that $\|g_v^t\|_\infty\leq\|f_{(v,0)}\|_\infty\|f_{(v,1)}\|_\infty$ we get
$$|[F](x)|\leq\|f_w\|_2\mathbb{E}_t^{1/2}\Bigl(\prod_{v\in K_{n-1},v_j=0}\|g_v^t\|^2_\infty\prod_{v\in K_{n-1},v_j=1}\|g_v^t\|^2_{U_{n-1}}\Bigr)\leq$$

$$\|f_w\|_\infty\prod_{v\in K_{n-1},v_j=0}\|f_{(v,0)}\|_\infty\|f_{(v,1)}\|_\infty\mathbb{E}_t^{1/2}\Bigl(\prod_{v\in K_{n-1},v_j=1}\|g_v^t\|^2_{U_{n-1}}\Bigr)\leq$$
$$\prod_{v\in K_n,v_j=0}\|f_v\|_\infty\prod_{v\in K_{n-1},v_j=1}\Bigl(\mathbb{E}_t(\|g_v^t\|_{U_{n-1}}^{2^{n-1}})\Bigr)^{2^{1-n}}$$

Let $v\in K_{n-1}$ and let $H=\{h_z\}_{z\in \{0,1\}^n}$ be the function system defined by $h_z=f_{(v,0)}$ if $z_n=0$ and $h_z=f_{(v,1)}$ if $z_n=1$. Then by (\ref{GCS}), (\ref{dimred}) we get
$$\mathbb{E}_t(\|g_v^t\|_{U_{n-1}}^{2^{n-1}})=\mathbb{E}_t((\delta_{n,t}H))=(H)\leq\|f_{(v,0)}\|_{U_n}^{2^{n-1}}\|f_{(v,1)}\|_{U_n}^{2^{n-1}}$$
which completes the proof.
\end{proof}

\medskip

\begin{lemma} If $k\geq 1$ then $\|f\|_{U_k}\leq\|f\|_{2^{k-1}}$.
\end{lemma}

\begin{proof} We prove the statement by induction. If $k=1$ then $\|f\|_{U_1}=|\mathbb{E}(f)|\leq\|f\|_1$.
Let $f_t$ denote the function with $f_t(x)=f(x+t)$. We have by induction that
$$\|f\|_{U_{k+1}}^{2^{k+1}}=\mathbb{E}_{t}(\|f\overline{f_t}\|_{U_k}^{2^k})\leq\mathbb{E}_t(\|f\overline{f_t}\|_{2^{k-1}}^{2^k})\leq\mathbb{E}_{t,x}(|f(x)|^{2^k}|f(x+t)|^{2^k})=\|f\|_{2^k}^{2^{k+1}}.$$
\end{proof}

\begin{corollary}\label{l2becs} If $k\geq 2$ and $|f|\leq 1$ then $(f,f)=\|f\|_2^2\geq \|f\|_{U_k}^{2^{k-1}}$.
\end{corollary}

\begin{proof} If $|f|\leq 1$ then $\|f\|_{U_k}^{2^{k-1}}\leq\|f\|_{2^{k-1}}^{2^{k-1}}\leq\|f\|_2^2$.
\end{proof}

\subsection{Low rank approximation and products of convolutions}

\begin{lemma}\label{shiftap} Let $A$ be a compact abelian group, $B<A$ a compact subgroup, $\epsilon>0$ and $f:A\rightarrow\mathbb{C}$ be a measurable function with $\|f\|_\infty\leq 1$. Let furthermore $f_B(z)=\mathbb{E}_{y\in B}f(z+y)$. Then there are elements $a_1,a_2,\dots,a_n$ in $B$ with $n\leq 1+4/\epsilon^2$ such that the function $g(z)=\frac{1}{n}\sum_{i=1}^n f(z+a_n)$ satisfies $\|f_B-g\|_2\leq\epsilon$.
\end{lemma}

\begin{proof} Let $n$ be an integer with $4/\epsilon^2\leq n\leq 1+4/\epsilon^2$. Let $h(z,a_1,a_2,\dots,a_n)=\frac{1}{n}\sum_{i=1}^nf(z+a_n)$ and $r(z,a_1,a_2,\dots,a_n)=f_B(z)$ be functions defined on $A\times B^n$. For a fixed $z\in A$ let $Y_z$ denote the random variable $f(z+y)-f_B(z)$ where $y$ is chosen randomly from $B$. If $z\in A$ is fixed then the value of $h-r$ for a randomly chosen element $(a_1,a_2,\dots,a_n)$ in $B^n$ has the same distribution as the average of $n$ independent copies of $Y_z$ and so on this probability space ${\rm Var}(h-r)={\rm Var}(Y_z)/n\leq 4/n$.
By taking the average of this for every $z$ we get that
$\|h-r\|_2^2\leq 4/n$. Consequently there is a fixed vector $(a_1,a_2,\dots,a_n)$ in $B^n$ such that
the function $\mathbb{E}_z(h(z,a_1,a_2,\dots,a_n)-r(z,a_1,a_2,\dots,a_n))^2\leq 4/n\leq\epsilon^2$ which finishes the proof.
\end{proof}

\medskip

\begin{definition}\label{rankone}  A {\bf rank one} function on $C^k(A)$ is a function of the form $[G]^\times$ where $G=\{g_v\}_{v\in K_k}$ is a function system in $L^\infty_u(A)$.
\end{definition}

Note that rank one functions are shift invariant on $C^k(A)$ and are closed under point wise multiplication.
The next lemma says that convolutions have low rank approximations when lifted to the space $C^k(A)$ with $\psi_0$. 

\begin{lemma}[Low rank approximation]\label{lowrank} Let $F=\{f_v\}_{v\in K_k}$ be a system of functions in $L^\infty_u(A)$ and $\epsilon>0$. Then there is a function $g$ on $C^k(A)$ which is the average of at most $1+4/\epsilon^2$ rank one functions and $\|g-[F]\circ\psi_0\|_2\leq\epsilon$ 
\end{lemma}

\begin{proof} Using (\ref{rankcon2}) and lemma \ref{shiftap} we obtain that there exist elements $y_1,y_2,\dots,y_n$ in $C^k_0(A)$ with $n\leq 1+4/\epsilon^2$ such that the function $g(z)=\frac{1}{n}\sum_{i=1}^n [F]^\times(z+y_i)$ satisfies $\|[F]\circ\psi_0-g\|_2\leq\epsilon$. Since the functions $z\rightarrow [F]^\times(z+y_i)$ are all rank one functions the proof is complete. 
\end{proof}

\begin{remark}\label{aptrans} It will be important that in lemma \ref{lowrank} the rank one functions occurring in the approximation use only shifted versions of functions from the system $F$.
\end{remark}

\begin{lemma}[product of convolutions]\label{prodconv} Let $F=\{f_v\}_{v\in K_k}$ and $G=\{g_v\}_{v\in K_k}$ be two systems in $L^\infty_u(A)$ and let $\epsilon>0$. Then there are function systems $H^i=\{h_v^i\}_{v\in K_k}$ in $L^\infty_u(A)$ for $i=1,2,\dots,n$ with $n\leq (1+64/\epsilon^2)^2$ such that 
$$\|~[F][G]-\frac{1}{n}\sum_{i=1}^n[H^i]~\|_2\leq\epsilon.$$
\end{lemma}

\begin{proof} We use lemma \ref{lowrank} for both $F$ and $G$ with $\epsilon/4$. This way we obtain approximations $f$ and $g$ for $[F]\circ\psi_0$ and $[G]\circ\psi_0$ with $L^2$ error $\epsilon/4$ such that both $f$ and $g$ are the averages of at most $1+64/\epsilon^2$ functions of rank one. In particular $\|f\|_\infty,\|g\|_\infty\leq 1$.  
Let us write $[F]\circ\psi_0=f+e_F$ and $[G]\circ\psi_0=g+e_G$ where $\|e_F\|_2,\|e_G\|_2\leq\epsilon/4$ and $\|e_F\|_\infty,\|e_G\|_\infty\leq 2$. Now we have 
$$([F][G])\circ\psi_0=([F]\circ\psi_0)([G]\circ\psi_0)=fg+e_Fg+e_Gf+e_Fe_G.$$ By
$\|e_Gf\|_2,\|e_Fg\|_2\leq\epsilon/4$ and $\|e_Fe_G\|_2\leq\epsilon/2$ we get that $$\|([F][G])\circ\psi_0-fg\|_2\leq\epsilon.$$ The function $fg$ is the average of $n\leq (1+64/\epsilon^2)^2$ functions of rank one. Let us denote the corresponding function systems by $H^1,H^2,\dots,H^n$. By (\ref{rankcon1}) it follows  that $$\mathbb{E}(fg|\psi_0)=\Bigl(\frac{1}{n}\sum_{i=1}^n[H^i]\Bigr)\circ\psi_0.$$ The function $([F][G])\circ\psi_0$ is already measurable in the $\sigma$-algebra generated by $\psi_0$ and so conditional expectation with respect to this $\sigma$ algebra leaves it invariant. Since conditional expectation is a contraction on $L^2$ and $\psi_0$ is measure preserving the proof is complete.
\end{proof}

\subsection{Higher degree cubes}

\begin{definition} Let $A$ be an Abelian group. A map $c:\{0,1\}^n\rightarrow A$ is a degree-$k$ cube if it extends to a  degree-$k$ polynomial map $f:\mathbb{Z}^n\rightarrow B$. 
\end{definition}

It can be seen that a function$\{0,1\}^n\rightarrow A$ is a degree-$k$ cube if and only if for every $k+1$ dimensional face $S\subset\{0,1\}^n$ we have $\sum_{v\in S}(-1)^{h(v)}c(v)=0$ where $h(v)=\sum_{i=1}^nv_i$.
If $k\leq-1$ then we define a degree-$k$ cube as the constant $0$ function on $\{0,1\}^n$.
For an integer $k\in\mathbb{Z}$ and abelian group $A$ we introduce the cubespace $\mathcal{D}_k(A)$ in which $C^n(\mathcal{D}_k(A))$ is the collection of degree-$k$ cubes of dimension $n$. It is easy to seet that if $k\geq 1$ then $\mathcal{D}_k(A)$ is a $k$-step nilspace. We regard $\mathcal{D}_k(A)$ as a degree-$k$ version of of $A$. In particular $\mathcal{D}_1(A)$ is the group $A$ with the usual cubic structure.

\begin{lemma}\label{dualker1} Let $n\in\mathbb{N},k\in\mathbb{Z}$. Let $A$ be a compact abelian group and let $f\in C^n(\mathcal{D}_{n-k-1}(\hat{A}))$. Then for every $c\in C^n(\mathcal{D}_k(A))$ we have that $\prod_{v\in\{0,1\}^n} \chi_v^\con(c(v))=1$ where $\chi_v=f(v)$.
\end{lemma}

\begin{proof} We go by induction on $n$. If $n=0$ then the statement is trivial.  
Assume that $n>0$ and that the statement is true for $n-1$. Then by induction we have the product in the lemma is equal to $$\prod_{v\in Q}\Bigl(\chi_v\overline{\chi_{v+w}}(c(v))\Bigr)^\con\prod_{v\in Q}\Bigl(\chi_{v+w}(c(v)-c(v+w))\Bigr)^\con$$
where $Q=\{(v,0)|v\in\{0,1\}^{n-1}\}$~,~$w=(0,0,\dots,0,1)\in\{0,1\}^n$.
The first product satisfies the conditions with $n-1,k$ and the second one with $n-1,k-1$. We have by induction that both products are $1$.
\end{proof}

\begin{lemma}\label{dualker2} Let $n\in\mathbb{N},k\in\mathbb{Z}$. Let $A$ be a compact abelian group and let $f\in\hom(K_n,\mathcal{D}_{n-k-1}(\hat{A}))$. Then for every $c\in C^n_0(\mathcal{D}_k(A))$ we have that $\prod_{v\in\{0,1\}^n}\chi_v^\con(c(v))=1$ where $\chi_v=f(v)$.
\end{lemma}

\begin{proof} Let us extend $f$ to the full cube $\{0,1\}^n$ such that $\overline{f}(0)=\prod_{v\in K_n}\chi_v^\con$. 
By lemma \ref{dualker1} it is enough to prove that this extension is in $C^n(\mathcal{D}_{n-k-1}(\hat{A}))$. If $k<0$ or $k>n$ then it is trivial. Assume that $0\leq k\leq n$ and
let $S$ be a face in $Q=\{0,1\}^n$ of dimension $n-k$. If $S\subset K_n$ then we have by our assumption that $\prod_{v\in S}f^\con(v)=1_A$. If $0\in S$ then $\prod_{v\in Q\setminus S}f^\con(v)=1_A$ since $Q\setminus S$ is a disjoint union of faces parallel to $S$. Then by $\prod_{v\in Q}f^\con(c)=1_A$ the proof is complete.
\end{proof}

\medskip

\subsection{Compact nilspaces}

The nilspace axioms were given in the introduction. 
In this chapter we review some of the results from \cite{NP} which we use in this paper.

Let $A$ be an abelian group and $X$ be an arbitrary set. An $A$ bundle over $X$ is a set $B$ together with a free action of $A$ such that the orbits of $A$ are parametrized by the elements of $X$. This means that there is a projection map $\pi:B\rightarrow X$ such that every fibre is an $A$-orbit. 
The action of $a\in A$ on $x\in B$ is denoted by $x+a$. Note that if $x,y\in B$ are in the same $A$ orbit then it makes sense to talk about the difference $x-y$ which is the unique element $a\in A$ with $y+a=x$. In other words the $A$ orbits can be regarded as affine copies of $A$.

A $k$-fold abelian bundle $X_k$ is a structure which is obtained from a one element set $X_0$ in $k$-steps in a way that in the $i$-th step we produce $X_i$ as an $A_i$ bundle over $X_{i-1}$.
The groups $A_i$ are the structure groups of the $k$-fold bundle. We call the spaces $X_i$ the $i$-th factors. 
If all the structure groups $A_i$ and spaces $X_i$ are compact and the actions are continuous then the $k$-fold bundle admits a Borel probability measure which is built up from the Haar measures of the structure groups in a recursive way. Let $\pi_i$ denote the the projection from $X_k$ to $X_i$.
Assume that the measure $\mu_{k-1}$ is already defined on $X_{k-1}$ and $\mu_{k-1}^*$ denotes the measure defined by $\mu^*_{k-1}(\pi_{k-1}^{-1}(S))=\mu_{k-1}(S)$ for Borel sets in $X_{k-1}$. Then the measure $\mu_k$ is the unique measure on $X_k$ with the property $$\int_{X_k} f ~d\mu_k=\int_{x\in X_k}\int_{a\in A_k} f(x+a)~d\nu_k~d\mu^*_{k-1}$$ where $\nu_k$ is the Haar measure on $A_k$ and $f$ is a bounded Borel function on $X_k$.

\begin{definition} Let $N_k$ be a $k$-fold abelian bundle with factors $\{N_i\}_{i=1}^k$ and structure groups $\{A_i\}_{i=1}^k$. Let $\pi_i$ denote the projection of $N_k$ to $N_i$.
Assume that $N_k$ admits a cubespace structure with cube sets $\{C^n(N_k)\}_{n=1}^\infty$.
We say that $N_k$ is a {\bf $k$-degree bundle} if it satisfies the following conditions
\begin{enumerate}
\item $N_{k-1}$ is a $k-1$ degree bundle.
\item For every $n\in\mathbb{N}$ the set $C^n(N_k)$ is a $C^n(\mathcal{D}_k(A_k))$-bundle with the pointwise action over $C^n(N_{k-1})$. The projection of the bundle is given by the composition with $\pi_{k-1}$.
\end{enumerate} 
\end{definition}

The next theorem form \cite{NP} says that $k$-degree bundles are the same as $k$-step nilspaces.

\begin{theorem}\label{bundec} Every $k$-degree bundle is a $k$-step nilspace and every $k$-step nilspace arises as a $k$-degree bundle. 
\end{theorem}

It will be important that if $N$ is a $k$-step compact nilspace then since the set $C^n(N)$ admits a $k$-fold bundle structure it has a natural probability measure on it.
Furthermore we have that the maps $\psi_v:C^n\rightarrow N$ defined by $\psi_v(c)=c(v)$ are all measure preserving and thus we can define the expressions $[F],(F),[F]^\times,(F)^\times$ similarly as in case of abelian groups.
We have for example that $[F]$ is continuous for every bounded measurable function system on $N$. 
We define $\|f\|_{U_n}$ by $\|f\|_{U_n}^{2^n}=(f)_n$ which is equal to $\mathbb{E}((f)_n^\times)$.
It turns out that if $n\leq k$ then $U_n$ is a semi-norm on $L^\infty(N)$ and it is a norm if $n\geq k+1$.

\begin{definition} Let $N$ be a $k$-step compact nilspace and $\chi\in\hat{A_k}$ be a linar character of the $k$-th structure group $A_k$. We denote by $W(\chi,N)$ the Hilbert space of functions $f\in L^2(N)$ such that $f(x+a)=f(x)\chi(a)$ holds for every $x\in N$ and $a\in A_k$.
\end{definition}

The next lemma is a direct consequence of theorem \ref{bundec}.

\begin{lemma}\label{nilspcharderiv} Assume that $N$ is a $k$-step compact nilspace, $\chi\in\hat{A}_k$ and $f\in W(\chi,N)$ is bounded. Then the function $(f)_{k+1}^\times$ (resp. $[f]_{k+1}^\times$ restricted to $C^{k+1}_x(N)$ for some $x\in N$) is the composition of a Borel function on $C^{k+1}(N_{k-1})$ (resp. $C^{k+1}_{\pi_{k-1}(x)}(N_{k-1})$~) and the projection $C^{k+1}(N)\rightarrow C^{k+1}(N_{k-1})$ induced by $\pi_{k-1}$.
Furthermore we have that $[f]_{k+1}\in V(\overline{\chi},N)$.
\end{lemma}

Another important fact about the spaces $W(\chi,N)$ is the following.

\begin{lemma}[Fourier decomposition on nilspaces]\label{nilfourdec} Let $N$ be a compact $k$-step nilspace. Then $$L^2(N)=\bigoplus_{\chi\in\hat{A_k}}W(\chi,N)$$ where the direct summands are orthogonal to each other. If $f:N\rightarrow\mathbb{C}$ is a bounded Borel measurable function then there is a unique decomposition
$f=\sum_{\chi\in\hat{A_k}}f_\chi$ into bounded functions $f_\chi\in W(\chi,N)$
converging in $L^2$. 
\end{lemma}

\begin{proof} It is clear that the functions defined by $f_\chi(x)=\mathbb{E}_{b\in A_k}f(x+b)\overline{\chi(b)}$ satisfy the above equality. 
\end{proof}

\begin{lemma}\label{nilspchar} Let $N$ be a $k$-step compact nilspace and $\chi\in\hat{A_k}$. Then there is a function $\phi\in W(\chi,N)$ such that $|\phi(x)|=1$ holds for every $x\in N$. Furthermore every function $f\in W(\chi,N)$ can be written as a product of $\phi$ and $h\circ\pi_{k-1}$ where $h$ is an $L^2$ function on the $k-1$ step factor of $N$.
\end{lemma}

\begin{proof} Let $f:N_{k-1}\rightarrow N_k$ be a function which chooses a Borel representative system for the fibres of $\pi_{k-1}$. Then we define $\phi(x)=\chi(x-f(\pi_{k-1}(x)))$. It is clear that $\phi$ satisfies the required condition.
\end{proof}

\begin{lemma}\label{gownilproj} Assume that $k\geq i\geq n-1\geq 0$ and $f\in L^\infty(N)$. Then $\|f\|_{U_n}=\|\mathbb{E}(f|\pi_i)\|_{U_n}$.
\end{lemma}

\begin{proof} We prove the statement by induction on $k-i$. If $k=i$ then there is nothing to prove. Assume that the statement is true for $i+1\leq k$. By $\mathbb{E}(f|\pi_i)=\mathbb{E}(\mathbb{E}(f|\pi_{i+1})\pi_i)$ we can assume that $f$ is measurable in the factor $\pi_{i+1}$. By abusing the notation we can assume that $f$ is defined on $N_{i+1}$ and we do the calculation inside $N_{i+1}$. We have that $C^{i+1}(N_{i+1})$ is a $C:=C^{i+1}(\mathcal{D}_{i+1}(A_{i+1}))$ bundle over $C^{i+1}(N_i)$. Recall that by definition the set $C$ is equal to the set of all functions $\{0,1\}^{i+1}\rightarrow A_{i+1}$. We have that $$\|f\|_{U_{i+1}}=\mathbb{E}((f)_{i+1}^\times)=\mathbb{E}(\mathbb{E}((f)_{i+1}^\times|\pi_i))=\mathbb{E}(\mathbb{E}_{c\in C}((f)_{i+1}^\times)^c)=\mathbb{E}((\mathbb{E}_{a\in A_{i+1}}f^a)_{i+1}^\times).$$ Since the right hand side is equal to $\|\mathbb{E}(f|\pi_i)\|_{U_{i+1}}$ the proof is complete.
\end{proof}

\medskip

Morphisms between compact nilspaces were defined in the introduction. We will also need a stronger notion of morphism which was defined in \cite{NP}.

\begin{definition}[Fibre surjective morphism] Let $N$ and $M$ be two $k$-step nilspaces. A morphism $\phi:N\rightarrow M$ is called fibre surjective if for every $0\leq i\leq k$ and $x\in N_i$ we have that $\phi(\pi_i^{-1}(x))=\pi_i^{-1}(y)$ for some element $y\in M_i$.  
\end{definition}

Fibre surjective morphisms have the useful property (see \cite{NP}) that they are measure preserving in the very strong sense that the induced maps from $C^n_x(N)$ to $C^n_{\phi(x)}(M)$ are all measure preserving for arbitrary $x\in N$ and $n\in\mathbb{N}$. A crucial result from \cite{NP} says the following.

\begin{theorem}\label{inverselimit} If $N$ is a compact $k$-step nilspace then it is the inverse limit of finite dimensional $k$-step nilspaces where the maps used in the inverse system are all fibre surjective.
\end{theorem} 

Note that a nilspace is called finite dimensional if all the structure groups $\{A_i\}_{i=1}^k$ are finite dimensional i.e. compact Abelian Lie-groups. This means that the dual groups $\{\hat{A_i}\}_{i=1}^k$ are all finitely generated. In particular finite nilspaces are $0$ dimensional.

\section{Ultra product groups and their factors}

\subsection{Ultra product spaces}

Let $\omega$ be a non principal ultra filter on the natural numbers. 
Let $\{X_i,\mu_i\}_{i=1}^\infty$ be pairs where $X_i$ is a compact Hausdorff space and $\mu_i$ is a probability measure on the Borel sets of $X_i$.
We denote by $\bX$ the ultra product space $\prod_{\omega}X_i$.
The space $\bX$ has the following structures on it.

\medskip

\noindent{\bf Strongly open sets:} ~We call a subset of $\bX$ strongly open if it is the ultra product of open sets $\{S_i\subset X_i\}_{i=1}^\infty$.

\medskip

\noindent{\bf Open sets:}~We say that $S\subset \bX$ is open if it is a countable union of strongly open sets. Open sets on $\bX$ form a $\sigma$-topology. This is similar to a topology but it has the weaker axiom that only countable unions of open sets are required to be open. It can be proved that $\bX$ with this $\sigma$-topology is countably compact. This means that if $\bX$ is covered by countably many open sets then there is a finite sub-system which covers $\bX$.

\medskip

\noindent{\bf Borel sets:} A subset of $\bX$ is called Borel if it is in the $\sigma$-algebra generated by strongly open sets. We denote by $\mathcal{A}(\bX)$ the $\sigma$ algebra of Borel sets.

\medskip

\noindent{\bf Ultra limit measure:}  If $S\subseteq \bX$ is a strongly open set of the form $S=\prod_\omega S_i$ then we define $\mu(S)$ as $\lim_\omega\mu_i(S_i)$. It is well known that $\mu$ extends as a probability measure to the $\sigma$-algebra of Borel sets on $\bX$. 

\medskip

\noindent{\bf Ultra limit functions:} Let $T$ be a compact Hausdorff topological space. Let $\{f_i:X_i\rightarrow T\}_{i=1}^\infty$ be a sequence of Borel measurable functions. We define $f=\lim_\omega f_i$ as the function on $\bX$ whose value on the equivalence class of $\{x_i\in X_i\}_{i=1}^\infty$ is $\lim_\omega f_i(x_i)$. Such functions will be called ultra limit functions. It is easy to see that ultra limit functions can always be modified on a $0$ measure set that they becomes measurable in the Borel $\sigma$-algebra on $\bX$. This means that ultra limit functions are automatically measurable in the completion of the Borel $\sigma$-algebra.

\medskip

\noindent{\bf Measurable functions:} It is an important fact (see \cite{ESz}) that every bounded measurable function on $\bX$ is almost everywhere equal to some ultra limit function $f=\lim_\omega f_i$. 

\medskip

\noindent{\bf Continuity:} A function $f:\bX\rightarrow T$ from $\bX$ to a topological space $T$ is called continuous if $f^{-1}(U)$ is open in $\bX$ for every open set in $T$. If $T$ is a compact Hausdorff topological space then $f$ is continuous if and only if it is the ultra limit of continuous functions $f_i:X_i\rightarrow T$. Furthermore the image of $\bX$ in a compact Hausdorff space $T$ under a continuous map is compact. 

\medskip

The fact that an ultra product space has only a $\sigma$-topology and not a topology might be upsetting for the first look. However one can look at $\bX$ as a space which is glued together form many ``ordinary'' topological spaces. These topological spaces appear as quotients of $\bX$.

\begin{definition} A {\bf topological factor} of $\bX$ is a surjective continuous map $\bX\rightarrow T$ to a countably based, Hausdorff topological space $T$.
\end{definition}

Note that the compactness of $T$ follows automatically from the fact that $\bX$ is countably compact. One can equivalently define topological factors through equivalence relations $\sim$ on $\bX$ such that the collection of open sets on $\bX$ that are unions of equivalence classes form a countably based, Hausdorff topological space.  


\subsection{Ultra product groups and the correspondence principle}

Let $\{A_i\}_{i=1}^\infty$ be a sequence of compact abelian groups. Let $(\bA,\mathcal{A},\bm)$ be the triple where $\bA=\prod_\omega A_i$, $\mathcal{A}$ is the Borel $\sigma$ algebra on $\bA$ and $\bm$ is the ultra limit of the normalized Haar measures on $A_i$.
To avoid the situation when $\bA$ is finite we will assume that for every $n\in\mathbb{N}$ the set $\{i:|A_i|>n\}$ is in the ultra filter $\omega$. 
Note that $\bA$ is an abelian group which is similar to an ordinary compact group in the sense that it admits the shift invariant probability measure $\bm$ defined on the Borel $\sigma$-algebra $\mathcal{A}$. 

\medskip

\noindent{\bf Associated structures:}~In this paper we will often work with algebraic structures associated with $\bA$. The most typical one is the cube space $C^k(\bA)$. In all of our cases there is a commutativity between taking ultra products and taking the associated structure. For example there is a natural bijection between $C^k(\bA)$ and $\prod_\omega C^k(A_i)$.

\medskip

\noindent{\bf Topology and $\sigma$-algebra:}~An important source of differences between compact groups and their ultra products is the different behavior of their Borel $\sigma$-algebras. As an abstract set, the ultra product $\prod_\omega A_i\times A_i$ is in a natural bijection with $\prod_\omega A_i\times\prod_\omega A_i=\bA\times\bA$.
However the Borel $\sigma$-algebra on $\prod_\omega A_i\times A_i$ is bigger than the product $\mathcal{A}\otimes\mathcal{A}$.  (The same thing is true for the $\sigma$-topology defined on them.)

\medskip

\noindent{\bf Shift invariant $\sigma$-algebras:}~A sub $\sigma$-algebra $\mathcal{B}\subset\mathcal{A}$ is called {\bf shift invariant} if $S\in\mathcal{B}$ implies that $S+t\in\mathcal{B}$ holds for every $t\in\bA$. It is clear that $\mathcal{A}$ and the measure $\bm$ are both shift invariant.

\subsection{Higher order Fourier $\sigma$-algebras}

Let $f=\lim_\omega f_i$ be a bounded measurable function on the ultra product group $\bA$. 
The general correspondence principle for ultra product groups yields that $\|f\|_{U_k}=\lim_\omega\|f_i\|_{U_k}$ holds for every natural number $k$.
Since the ultra limit of positive numbers is non-negative it follows that $U_k$ is a semi norm on $L^\infty(\bA)$. 
We show in this chapter that for every $k\in\mathbb{N}$ there is a unique largest $\sigma$-algebra $\mathcal{F}_k\subset\mathcal{A}$ such that $U_{k+1}$ is a norm on $L^\infty(\mathcal{F}_k)$ and $L^\infty(\mathcal{F}_k)$ is orthogonal to every function with zero $U_{k+1}$ norm. Roughly speaking, a function is measurable in $\mathcal{F}_k$ if it is purely ``structured'' in $k$-th order Fourier analysis. 
One of many advantages of the ultra product framework is that there is clear distinction between $k$-th order noise and $k$-th order structure on $\bA$. Functions that have zero $U_{k+1}$ norm are considered to be ``pure noise'' in the $k$-th order setting.

\begin{definition} Let $\mathcal{B}\subseteq\mathcal{A}$ be a $\sigma$-algebra. Then we denote by $[\mathcal{B}]_k$,$[\mathcal{B}]_k^\times$ and $(\mathcal{B})_k^\times$ the $\sigma$-algebra generated by the functions $[F],[F]^\times$ and $(F)^\times$ where $F$  runs through all the function systems $\{f_v\}_{v\in\{0,1\}^k}$ in $L^\infty(\mathcal{B})$. 
\end{definition}

We have that $$[\mathcal{B}]^\times_k=\bigvee_{v\in K_k}\mathcal{B}\circ\psi_v~~~{\rm and}~~~(\mathcal{B})^\times_k=\bigvee_{v\in\{0,1\}^k}\mathcal{B}\circ\psi_v.$$

\begin{lemma}\label{convcontlift} Let $\mathcal{B}\subset\mathcal{A}$~,~$k\in\mathbb{N}$ and let $\mathcal{C}\subset\mathcal{A}$ be the unique $\sigma$-algebra with $\mathcal{C}\circ\psi_0=\mathcal{A}\circ\psi_0\wedge[\mathcal{B}]_k^\times$. Then $\mathcal{C}\subset[\mathcal{B}]_k$.
\end{lemma}

\begin{proof} Let $f\in L^\infty(\mathcal{C})$. We have by lemma \ref{siggen} that for every $\epsilon>0$ the function $f\circ\psi_0$ has an epproximation of the form $g=\sum_{i=1}^t[F^i]^\times$ with $L^2$ error at most $\epsilon$ where each $F^i$ is a function systme $\{f^i_v\}_{v\in K_k}$ in $L^\infty(\mathcal{B})$.  Then $g'=\sum_{i=1}^t[F^i]$ satisfies $g'\circ\psi_0=\mathbb{E}(g|\psi_0)$ and thus by $\epsilon\geq\|g-f\circ\psi_0\|_2\geq\|\mathbb{E}(g|\psi_0)-f\circ\psi_0\|_2=\|g'-f\|_2$ we have that $f$ has an arbitrary precise approximation in $L^2([\mathcal{B}]_k)$.
\end{proof}

\begin{definition}  Let $\mathcal{F}_k$ be the $\sigma$-algebra $[\mathcal{A}]_{k+1}$. We say that $\mathcal{F}_k$ is the {\bf $k$-th order Fourier $\sigma$-algebra}. 
\end{definition}

\begin{theorem}[Properties of $\mathcal{F}_k$]\label{propfk} Let $k$ be a natural number. Then
\begin{enumerate}
\item On the space $C^{k+1}(\bA)$ we have
$\mathcal{F}_k\circ\psi_0=\mathcal{A}\circ\psi_0\wedge [\mathcal{A}]_{k+1}^\times,$
\item $\|f\|_{U_{k+1}}=0$ holds if and only if $\mathbb{E}(f|\mathcal{F}_k)=0$,
\item $U_{k+1}$ is a norm on $L^\infty(\mathcal{F}_k)$,
\item $\mathcal{F}_k$ is shift invariant,
\item On the space $C^{k+1}(\bA)$ we have
$\mathcal{F}_k\circ\psi_0=\mathcal{A}\circ\psi_0\wedge [\mathcal{F}_k]_{k+1}^\times,$
\item $\mathcal{F}_0$ is the trivial $\sigma$-algebra and $\mathcal{F}_0\subset\mathcal{F}_1\subset\mathcal{F}_2\subset\dots$ is an increasing chain.
\end{enumerate}
\end{theorem}

\begin{proof} 

Let $\mathcal{B}$ denote the $\sigma$-algebra whose pre-image under $\psi_0$ is equal to $\mathcal{A}\circ\psi_0\wedge[\mathcal{A}]_{k+1}^\times$. The first statement in the theorem says that $\mathcal{B}=\mathcal{F}_k$. By lemma \ref{convcontlift} we have that $\mathcal{B}\subseteq\mathcal{F}_k$.
For the statement $\mathcal{F}_k\subseteq\mathcal{B}$ it is enough to prove that $[F]\circ\psi_0$ is measurable in $[\mathcal{A}]_{k+1}^\times$ for every function system $F=\{f_v\}_{v\in K_{k+1}}$ in $L^\infty_u(\mathcal{A})$. This follows immediately from lemma \ref{lowrank}.

Let $R$ be the set of rank one functions on $C^{k+1}(\bA)$. We claim that $\mathbb{E}(f|\mathcal{F}_k)=0$ holds if and only if $f$ is orthogonal to every function in $R$. 
One direction is trivial since $L^\infty(\mathcal{F}_k)$ contains $R$. Assume that $f$ is orthogonal to $R$. Let $G=\{g_v\}_{v\in K_{k+1}}$ be a function system in $L^\infty_u(\bA)$ and let $g=[G]^\times$ be the corresponding rank one function. We have by (\ref{coincon}) that $0=(f,[G])=(f\circ \psi_0,g)$. This shows that $f\circ\psi_0$ is orthogonal to every rank one function and thus $\mathbb{E}(f\circ\psi_0|\mathcal{G})=0$. Using $\mathcal{F}_k\circ\psi_0\subseteq\mathcal{G}$ from the first part of the theorem we obtain $$0=\mathbb{E}(f\circ\psi_0|\mathcal{F}_k\circ\psi_0)=\mathbb{E}(f|\mathcal{F}_k)\circ\psi_0$$ showing that $\mathbb{E}(f|\mathcal{F}_k)=0$.

By (\ref{GCS}) and (\ref{coincon}) we have that $\|f\|_{U_{k+1}}=0$ if and only if $f$ is orthogonal to $R$. This proves the second statement. The third statement follows form the fact that if $f\in L^\infty(\mathcal{F}_k)$ then $f=\mathbb{E}(f|\mathcal{F}_k)$.

Statement four follows from the fact that convolution of shifts of functions (with a fix element $t\in\bA$) is the shift of the convolution. Thus the generating system of $\mathcal{F}_k$ is shift invariant.
 
For the fifth statement let $\mathcal{B}'\subset\mathcal{A}$ be the $\sigma$-algebra such that $\mathcal{B}'\circ\psi_0=\mathcal{A}\circ\psi_0\wedge[\mathcal{F}_k]_{k+1}^\times$. Our goal is to show that $\mathcal{F}_k=\mathcal{B}'$. It is clear from the first statement of the theorem that $\mathcal{B}'\subseteq\mathcal{F}_k$. To show the other inclusion it is enough to show that every convolution in $R$ is contained in $\mathcal{B}'$. Let $F=\{f_v\}_{v\in K_{k+1}}$ be a function system in $L^\infty_u(\bA)$. Let $G=\{g_v:=\mathbb{E}(f_v|\mathcal{F}_k)\}_{v\in K_{k+1}}$. By the second part of the theorem we get $\|f_v-g_v\|_{U_{k+1}}=0$ for every $v\in K_{k+1}$. Using the fact that convolution is linear in the components and lemma \ref{cornineq} we conclude that in a process, where we step by step replace the terms $f_v$ by $g_v$ in the system $F$, the convolution doesn't change. It follows that $[G]=[F]$. Now remark \ref{aptrans} and lemma \ref{lowrank} show that $[F]\circ\psi_0=[G]\circ\psi_0$ can be approximated by rank one functions using only translates of the functions $g_v$. Since the functions $g_v$ are all measurable in $\mathcal{F}_k$ and $\mathcal{F}_k$ is shift invariant we obtain that the function $[F]\circ\psi_0$ is measurable in $[\mathcal{F}_k]_{k+1}^\times$. Thus we get that $[F]$ is measurable in $\mathcal{B}'$. 

The last statement follows from the fact that every convolution of order $k$ is also a convolution of order $k+1$. this can be seen by using constant one functions in a function system on $K_{k+1}$ outside of a $k$ dimensional face containing the $0$ vector.
\end{proof}

\subsection{Identities for convolutions}

\begin{lemma}\label{fkconvfk} Let $n,k$ be natural numbers and let $F=\{f_v\}_{v\in K_n}$ be a system of functions in $L^\infty(\mathcal{F}_k)$. Then the function $[F]$ is measurable in $\mathcal{F}_k$.
\end{lemma}

\begin{proof} If $n\leq k+1$ then the claim is clear by the definition of $\mathcal{F}_k$. Assume that $n>k+1$. By theorem \ref{propfk} the statement is equivalent with the fact that $[F]$ is orthogonal to any function $g\in L^\infty(\bA)$ with $\|g\|_{U_{k+1}}=0$. Let $g_v=f_v^\con$ for $v\in K_n$. Then $[F](x)$ can be calculated by the formula (\ref{corner}). It is clear that for every fixed $t_{k+2},t_{k+3},\dots,t_n$ in $\bA$ the expected value of the expression in (\ref{corner}) according to $t_1,t_2,\dots,t_{k+2}$ is a $k+1$-th order convolution of functions in $\mathcal{F}_k$ and thus it is orthogonal to $g$. Then the non-standard version of Fubini's theorem \cite{ESz} finishes the proof.
\end{proof}

\begin{lemma}\label{szorzat} If $f\in L^\infty(\mathcal{F}_k)$ and $g\in L^\infty(\bA)$ satisfies $\|g\|_{U_{k+1}}=0$ the $\|fg\|_{U_{k+1}}=0$.
\end{lemma}

\begin{proof} Using theorem \ref{propfk} we have  $\mathbb{E}(fg|\mathcal{F}_k)=f\mathbb{E}(g|\mathcal{F}_k)=0$ and thus $\|fg\|_{U_{k+1}}=0$.
\end{proof}

\begin{lemma}\label{szomszed} Let $F=\{f_v\}_{v\in \{0,1\}^{k+1}}$ be a function system in $L^\infty(\mathcal{A})$ such that for some pair of two neighbouring vertices $w_1,w_2\in \{0,1\}^{k+1}$ we have $f_{w_1}\in L^\infty(\mathcal{F}_{k-1})$ and $\|f_{w_2}\|_{U_k}=0$. then $(F)=0$.
\end{lemma}

\begin{proof} Assume that $w_1$ and $w_2$ differ at the $i$-th coordinate. Then by lemma \ref{szorzat} we have that for every $t\in\bA$ the system $\delta_{i,t} F$ has a function (obtained from $f_{w_1}$ and $f_{w_2}$) with zero $U_k$-norm. Then by (\ref{GCS}) we have that $(\delta_{i,t}(F))=0$. and thus (\ref{dimred}) finishes the proof. 
\end{proof}

The next two lemmas are useful consequences of lemma \ref{szomszed}. 

\begin{lemma}\label{twofunct} Let $f,g$ be $L_\infty(\mathcal{A})$ functions such that $f=f_1+f_2$ where $f_1=\mathbb{E}(f|\mathcal{F}_{k-1})$ and $g=g_1+g_2$ where $g_1=\mathbb{E}(g|\mathcal{F}_{k-1})$. Then
\begin{equation}\label{twofun} \mathbb{E}_t\Bigl(\|f(x)\overline{g(x+t)}\|^{2^k}_{U_k}\Bigr)=\mathbb{E}_t\Bigl(\|f_1(x)\overline{g_1(x+t)}\|^{2^k}_{U_k}\Bigl)+\mathbb{E}_t\Bigr(\|f_2(x)\overline{g_2(x+t)}\|^{2^k}_{U_k}\Bigr).
\end{equation}
\end{lemma}

\begin{proof} Let $H=\{h_v\}_{v\in\{0,1\}^{k+1}}$ be the function system such that $h_w=f$ if the last coordinate of $w$ is zero and $h_w=g$ if the last coordinate is one. It is clear by (\ref{dimred}) that the left hand side of (\ref{twofun}) is equal to $(H)$. By the linearity of $(H)$ in each coordinate we can decompose it into $2^{2^{k+1}}$ terms using $f=f_1+f_2$ and $g=g_1+g_2$. Then lemma \ref{szomszed} shows that only the two terms representing the right hand side of (\ref{twofun}) are not necessarily zero.
\end{proof}

\begin{lemma}\label{vetites1} Let $F=\{f_v\}_{v\in K_{k+1}}$ be a function system in $L^\infty(\bA)$. Furthermore let $G=\{g_v\}_{v\in K_{k+1}}$ with $g_v=\mathbb{E}(f_v|\mathcal{F}_{k-1})$. Then $[G]=\mathbb{E}([F]~|~\mathcal{F}_{k-1})$.
\end{lemma}

\begin{proof} Since $[G]$ is measurable in $\mathcal{F}_{k-1}$ by lemma \ref{fkconvfk}, it is enough to prove that for an arbitrary $h\in L^\infty(\mathcal{F}_{k-1})$ we have $([G],h)=([F],h)$. Let us write $f_v=g_v+r_v$ for every $v\in K_{k+1}$ and notice that $\|r_v\|_{U_{k-1}}=0$. Let $F'=\{f'_v\}_{v\in\{0,1\}^{k+1}}$ be the function system with $f'_v=f_v$ if $v\neq 0$ and $f'_0=\overline{h}$. Then by (\ref{coincon}) we have that $([F],h)=(F')$.
The linearity of the Gowers inner product implies that we can decompose $(F')$ into $2^{2^{k+1}-1}$ terms according to the decompositions $f_v=g_v+r_v$. By lemma \ref{szomszed} we have that the only non zero term is the one where we use $g_v$ at every place. This term is equal to $([G],h)$.
\end{proof}

\medskip

In the rest of the chapter we study conditions that force convolutions of the form $[F]$ to be $0$. For example lemma \ref{cornineq} implies that if any function in the function system $F=\{f_v\}_{v\in K_n}$ has zero $U_n$-norm then $[F]$ is zero. We will need some notation.

For an element $v\in\{0,1\}^n$ we introduce the height $h(v)$ of $v$ as the coordinate sum of $v$. For $v,w\in\{0,1\}^n$ we say that $v\leq w$ if $w_i=0$ implies $v_i=0$ for every $i\in [n]$. In other words $v\leq w$ if ${\rm supp}(v)\subseteq{\rm supp}(w)$. A simplicial set $S\in\{0,1\}^n$ is a set such that $w\in S$ and $v\leq w$ implies $v\in S$. For $v\in S$ the degree $d(v)$ is defined as $\max\{h(w)|w\in S,v\leq w\}$. A maximal element $v$ in $S$ is an element with $h(v)=d(v)$.
A maximal face of $S$ is a set of the form $\{v|v\leq w\}$ where $w\in S$ is maximal.
The hight of $S$ is the maximum of the heights of its elements.
For a number $i\in [n]$ let us define the projection $p_i:\{0,1\}^n\rightarrow\{0,1\}^{n-1}$ given by deleting the $i$-th coordinate. It is clear that the image of a simplicial set $S$ under the projection $p_i$ is again simplicial.  

\begin{lemma}\label{simpzero} Let $S\subset\{0,1\}^n$ be a simplicial set and let $s,u\in S$ be two distinct elements. Let $K=\{0,1\}^n\setminus\{u\}$ and $k=d(s)$. Let $F=\{f_v\}_{v\in K}$ be a function system in $L^\infty(\bA)$ such that $\|f_s\|_{U_k}=0$ and $f_v=1_{\bA}$ if $v\in K\setminus S$.
Then the convolution $[F]$ (taken at $u$) is identically $0$. 
\end{lemma}

\begin{proof}
For the definition of $[F]$ in this case see remark \ref{convat}. If $h(s)=n$ or $h(u)=n$ then $S=\{0,1\}^n$ and thus $k=n$. In this this case lemma \ref{cornineq} shows that $[F]=0$. We prove the statement by induction on $n-h(s)$. The case $n-h(s)=0$ is now proved. We can assume that $h(s),h(u)<n$. 
Using the fact that neither of $s$ and $u$ is the all $1$ vector we get that there is a coordinate $r\in [n]$ such that $s_r=0$ and the vector $s'$ obtained from $s$ by changing the $r$-th coordinate to $1$ satisfies $s'\neq u$. Let us decompose $f_{s'}$ as $f_{s'}=g_1+g_2$ where $g_1=\mathbb{E}(f_{s'}|\mathcal{F}_{k-1})$and $\|g_2\|_{U_k}=0$. Similarly we introduce two function systems $F_1,F_2$ where $F_i$ is obtained from $F$ by replacing $f_{s'}$ by $g_i$.
By linearity of convolution we get that $[F]=[F_1]+[F_2]$. If $s'\notin S$ then $f_{s'}=1_{\bA}$~,~$g_2=0$ and so $[F_2]=0$. If $s'\in S$ then $d(s')=d(s)=k$ and so by induction $[F_2]=0$. It remains to show that $[F_1]=0$.
We use our induction step for the function system $\delta_{r,t} F_1$. Notice that by lemma \ref{szorzat} the function $f_s(x)\overline{f_{s'}(x+t)}$ of coordinate $p_r(s)$ in $\delta_{r,t} F_1$ has zero $U_k$ norm. It is clear that in the complement of $p_r(S)$ every function in $\delta_{r,t} F_1$ is $1_{\bA}$. Since $h(p_r(s))=h(s)$ and $d(p_r(s))\leq d(s)=k$ we have by induction that $[\delta_{r,t} F_1]=0$. By (\ref{dimred}) we obtain that $[F_1]=0$.
\end{proof}

\begin{corollary}\label{simpzerocor} Let $S\subset\{0,1\}^n$ be a simplicial set of hight at most $k$, and let $\{f_v\}_{v\in\{0,1\}^n}$ be a function system in $L^\infty(\bA)$ such that $f_v=1_{\bA}$ if $v\in\{0,1\}^n\setminus S$. Let $G=\{g_v:=\mathbb{E}(f_v|\mathcal{F}_{k-1})\}_{v\in\{0,1\}^n}$. Then $(F)=(G)$. 
\end{corollary}

\begin{proof} By the multi linearity of $(F)$ it is enough to prove that if $\|f_v\|_{U_k}=0$ holds for some $v\in S$ then $(F)=0$. (Then the statement follows by decomposing each $f_v$ as $g_v+(f_v-g_v)$ where $\|f_v-g_v\|_{U_k}=0$.)
Let $u\neq v$ be some element in $S$. We have by (\ref{coincon}) that $(F)$ is the scalar product of the convolution of $F$ taken at $u$ with $\overline{f_u}$. Then lemma \ref{simpzero} finishes the proof. 
\end{proof}

\subsection{Higher order dual groups and Fourier decompositions}

Fourier analysis on a compact abelian group relies on the fact that $L^2(A)$ is the orthogonal sum of one dimensional shift invariant subspaces. This decomposition is unique and the one dimensional subspaces are forming an abelian group $\hat{A}$ under point wise multiplication. The group $\hat{A}$ is a discrete ablelian group called the {\bf dual group} of $A$. Each one dimensional shift invariant subspace is generated by a continuous homomorphism form $A$ to the unit circle in the complex plane. Consequently $\hat{A}$ is also the group of linear characters under pointwise multiplication. 
The decomposition 
\begin{equation}\label{mdecomp1}
L^2(A)=\bigoplus_{\chi\in\hat{A}}W_\chi
\end{equation}
give rise to the Fourier decomposition
\begin{equation}\label{fdecomp1}
f=\sum_{\chi\in\hat{A}}f_\chi
\end{equation}
converging in $L_2$ where $f_\chi$ is the projection of $f$ to $W_\chi$. 

\begin{remark} Let $X$ be an affine version of $A$ (see chapter \ref{cubes}). Then then the same decomposition as (\ref{mdecomp1}) holds for $L^2(X)$ where the one dimensional subspaces are again indexed by $\hat{A}$. This shows that the Fourier decomposition (\ref{fdecomp1}) can be uniquely defined on $X$ despite of the fact that linear characters are not uniquely defined on $X$ (they depend on a constant multiplicative factor).
\end{remark}

In this chapter we study similar decompositions in $L^2(\bA)$ for an ultra product group $\bA$. We will see that a new interesting phenomenon emerges in the ultra product setting which is a crucial part of our approach to higher order Fourier analysis.
Let $\hat{\bA}$ denote the set of one dimensional shift invariant subspaces of $L^2(\bA)$. The surprising fact is that $L^2(\bA)$ is not generated by the spaces in $\hat{\bA}$ and thus ordinary Fourier analysis is not enough to treat an arbitrary measurable function on $\bA$. In fact it turns out that the space spanned by the spaces in $\hat{\bA}$ is exactly $L^2(\mathcal{F}_1(\bA))$ where $\mathcal{F}_1$ is the first order Fourier $\sigma$-algebra on $\bA$.
This means that the use of ordinary Fourier analysis is restricted to functions that are measurable in $\mathcal{F}_1$.
We will need higher order generalizations of $\hat{\bA}$ to define the analogy of (\ref{mdecomp1}) and (\ref{fdecomp1}) for functions that are measurable in $\mathcal{F}_k$. 

\begin{definition} A {\bf module of order $k$} is a closed subspace $W\subset L^2(\bA)$ such that if $f\in L^\infty(\mathcal{F}_{k-1}(\bA))$ then $fW\subseteq W$ (using pointwise multiplication). For every set of elements $\{\phi_i\}_{i\in I}$ in $L^2(\bA)$ there is a unique smallest module $W$of order $k$ containing all of them. We say that $W$ is generated by the system $\{\phi_i\}_{i\in I}$. The {\bf rank} of $W$ is the smallest cardinality of a generating system.
\end{definition}

Note that modules of order one are just linear subspaces of $L^2(\bA)$. In this case rank is equal to the dimension of the subspace. Using the above definition we arrive to our main definition.

\begin{definition} A {\bf $k$-th order character} of $\bA$ is a function $\phi:\bA\rightarrow\mathbb{C}$ of absolute value one such that $\Delta_t\phi$ is measurable in $\mathcal{F}_{k-1}$ for every $t\in\bA$. The {\bf $k$-th order dual group} $\hat{\bA}_k$ of $\bA$ is the set of $k$-th order rank one modules generated by $k$-th order characters.
\end{definition}

Note that the definition implies that every element of $\hat{\bA}_k$ is a shift invariant rank one module of order $k$. We will see later that $\hat{\bA}_k$ could be equivalently defined as the set of shift invariant rank one modules of order $k$.
To justify the name ``higher order dual group'' we need to give a group structure to it. It is basically the point wise multiplication but we need to define it carefully. The product of two $L^2$ functions is not necessary in $L^2$. For this reason we define the product of $W_1,W_2\in\hat{\bA}_k$ as the $L_2$ closure of the set of products $f_1f_2$ where $f_1$ and $f_2$ are bounded functions from $W_1$ and $W_2$.
It is clear that $\hat{\bA}_k$ becomes an abelian group with this multiplication where the inverse of an element $W$ is obtained by conjugating the elements in $W$. 
Note that $\hat{\bA}_k$ is isomorphic to the group of $k$-th order characters factored out by the group of functions in $L^\infty(\mathcal{F}_{k-1}(\bA))$ of absolute value one.

\begin{lemma}\label{erosort} Let $\mathcal{B}\subseteq\mathcal{A}(\bA)$ be a shift invariant $\sigma$-algebra. Let $\phi:\bA\rightarrow\mathbb{C}$ be a function with $|\phi|=1$ such that $\Delta_t\phi$ is measurable in $\mathcal{B}$ for every $t\in\bA$. Then either $\mathbb{E}(\phi|\mathcal{B})$ is constant $0$ or $\phi\in L^\infty(\mathcal{B})$. In particular if $\phi$ is a $k$-th order character which is not in the trivial module then $\mathbb{E}(\phi|\mathcal{F}_{k-1})=0$.
\end{lemma}

\begin{proof} 
Let $f=\mathbb{E}(\phi|\mathcal{B})$. Then for every fixed $t$ we have $$f(x+t)=\mathbb{E}(\phi(x+t)|\mathcal{B})=\mathbb{E}(\phi\overline{\Delta_t\phi}|\mathcal{B})=f\overline{\Delta_t\phi}$$
and thus $f\overline\phi$ is translation invariant. We obtain that $f=c\phi$ for some constant $c$. Using that $f$ is the projection of $\phi$ to $L^2(\mathcal{B})$ we get that either $f=0$ or $f=\phi$. 
\end{proof}

\begin{lemma}\label{kisk} Every $k$-th order character (and thus every module in $\hat{\bA}_k$) is in $L^2(\mathcal{F}_k)$.
\end{lemma}

\begin{proof} Let $\phi$ be a $k$-th order character. We have by (\ref{dimred}) that $$\mathbb{E}_t(\|\Delta_t\phi\|_{U_k}^{2^k})=\|\phi\|_{U_{k+1}}^{2^{k+1}}.$$
Using that $\Delta_t\phi$ is measurable in $\mathcal{F}_{k-1}$ we get by theorem \ref{propfk} that $\|\Delta_t\phi\|_{U_k}>0$ holds for every $t$ and thus by the above formula $\|\phi\|_{U_{k+1}}>0$. This means by theorem \ref{propfk} that $\mathbb{E}(\phi|\mathcal{F}_k)\neq 0$. Since $\phi$ is also a $k+1$-th order character we obtain by lemma \ref{erosort} that $\phi$ has to be in the trivial module and so $\phi$ is measurable in $\mathcal{F}_k$. 
\end{proof}

\begin{lemma}\label{charort} Every two distinct modules $W_1,W_2$ in $\hat{\bA}_k$ are orthogonal to each other.
\end{lemma}

\begin{proof} Let $g_1=f_1\phi_1\in W_1$ and $g_2=f_2\phi_2\in W_2$ be two elements where $\phi_1,\phi_2$ are $k$-th order characters and $f_1,f_2\in L^\infty(\mathcal{F}_{k-1}(\bA))$. We have that $$(g_1,g_2)=\mathbb{E}(\mathbb{E}(g_1\overline{g_2}|\mathcal{F}_{k-1}))=\mathbb{E}(f_1f_2\mathbb{E}(\phi_1\overline{\phi_2}|\mathcal{F}_{k-1})).$$
By lemma \ref{erosort} the right hand side is $0$. Since such elements $g_1$ and $g_2$ are $L^2$ dense in $W_1$ and $W_2$ which completes the proof.
\end{proof}

An important consequence of our main result, theorem \ref{main} is the following.

\begin{theorem}[Higher order Fourier decomposition]\label{hofdecomp} For every $1\leq k\in\mathbb{N}$ we have that
$$L^2(\mathcal{F}_k(\bA))=\bigoplus_{W\in\hat{\bA}_k}W$$
and so every function $f\in L^2(\mathcal{F}_k(\bA))$ has a unique decomposition
$$f=\sum_{W\in\hat{\bA}_k}f_W$$ converging in $L^2$ where $f_W$ is the projection of $f$ to the modul $W$.
\end{theorem}

In the rest of this chapter we focus on the measure theoretic properties of higher order characters.
The simplest examples for $k$-th order characters are functions $\phi:\bA\rightarrow\mathcal{C}$ such that
$$\Delta_{t_1,t_2,\dots,t_{k+1}}\phi(x)=1$$
for every $t_1,t_2,\dots,t_{k+1},x$ in $\bA$ or equivalently: $(\phi)_{k+1}^\times=1$.
Such functions could be be called {\bf pure} characters. 
Unfortunately for $k>1$ there are groups $\bA$ on which not every modul in $\hat{\bA}_k$ can be represented by a pure character.
This justifies the next definition.

\begin{definition}[Locally pure characters] Let $\mathcal{B}\subseteq\mathcal{A}(\bA)$ be any $\sigma$ algebra.
We denote by $[\mathcal{B},k]^*$ the set of functions $\phi:\bA\rightarrow\mathbb{C}$ of absolute value $1$ such that $(\phi)^\times_{k+1}$ is measurable in $(\mathcal{B})_{k+1}^\times$.  
\end{definition}

In case $\mathcal{B}$ is the trivial $\sigma$-algebra then $[\mathcal{B},k]^*$ is just the set of pure characters with $(\phi)_{k+1}^\times=1$.

\begin{lemma}\label{pureprop} Let $\mathcal{B}\subseteq\mathcal{A}(\bA)$ be a $\sigma$-algebra. Then: \begin{enumerate}
\item $[\mathcal{B},k]^*$ is an Abelian group with respect to point wise multiplication.
\item $[\mathcal{B},0]^*$ is the set of $\mathcal{B}$ measurable functions $f:\bA\rightarrow \mathbb{C}$ of absolute value $1$.
\item $[\mathcal{B},k]^*\subseteq [\mathcal{B},k+1]^*$
\item If $\mathcal{B}$ is shift invariant and $\phi\in[\mathcal{B},k]^*$ then $\Delta_t\phi\in[\mathcal{B},k-1]^*$ for every $t\in\bA$
\end{enumerate}
\end{lemma}

\begin{proof} The first three properties are trivial. We show the last statement. Let $f=f(x,t_1,t_2,\dots,t_{k+1})$ be a bounded function on $C^{k+1}(\bA)$ using the parametrization $C^{k+1}\simeq\bA^{k+2}$ introduced in chapter \ref{cubes} such that $f$ is measurable in $(\mathcal{B})_{k+1}^\times$.  We claim that for almost every fixed value of $t_k$ the restriction of $f$ defined on $C^k(\bA)$ is measurable in $(\mathcal{B})_k^\times$. 
Using the shift invariance of $\mathcal{B}$ the statement is obviously true for functions in $\mathcal{R}(\{\mathcal{B}\circ\psi_v\}_{v\in\{0,1\}^{k+1}})$ (see notation in lemma \ref{siggen}). Then lemma \ref{siggen} shows the general case. 

By applying the claim for $(\phi)_{k+1}^\times$ we get that for almost every $t\in\bA$ the function $(\Delta_t\phi)_k^\times$ is measurable in $(\mathcal{B})_k^\times$ and thus for such $t$'s we have that $\Delta_t\phi\in [\mathcal{B},k-1]^*$. Now let $t\in\bA$ be an arbitrary fixed element. By the previous result both $\Delta_{t'}\phi$ and $\Delta_{t-t'}\phi$ are in $[\mathcal{B},k-1]^*$ for almost every $t'$ and by fixing one such element we obtain that $\Delta_t\phi(x)=\Delta_{t'}\phi(x)\Delta_{t-t'}\phi(x+t')$ is in $[\mathcal{B},k-1]^*$.
\end{proof}


\begin{lemma}\label{charmes} Every element in $[\mathcal{B},k]^*$ is measurable in the $\sigma$-algebra  $\mathcal{F}_k\vee\mathcal{B}$.
\end{lemma}

\begin{proof} Assume that $f\in[\mathcal{B},k]^*$. Let $\mathcal{C}=[\mathcal{A}]_{k+1}^\times$.
Let $g_1=[f]_{k+1}^\times$ and $g_2=f\circ\psi_0/g_1$ on $C^{k+1}(\bA)$.
We have that $g_1$ is measurable in $\mathcal{C}\vee\mathcal{B}\circ\psi_0$ and $g_2$ is measurable in $\mathcal{C}$. This means that $f\circ\psi_0=g_1g_2$ is measurable in $\mathcal{C}\vee\mathcal{B}\circ\psi_0$. On the other hand $f\circ\psi_0$ is measurable in $\mathcal{A}\circ\psi_0$ and thus it is measurable in
$(\mathcal{C}\vee\mathcal{B}\circ\psi_0)\wedge\mathcal{A}\circ\psi_0.$

We claim that $\mathcal{A}\circ\psi_0$ is conditionally independent from $\mathcal{C}$. By theorem \ref{propfk} we have that $\mathcal{A}\circ\psi_0\wedge\mathcal{C}=\mathcal{F}_k\circ\psi_0$. Assume that $h\in L^\infty(\bA)$ is orthogonal to $L^2(\mathcal{F}_k)$. Then $\|h\|_{U_{k+1}}=0$ and thus by (\ref{GCS}) we have that $h\circ\psi_0$ is orthogonal to every element in $\mathcal{R}(\{\mathcal{A}\circ\psi_v\}_{v\in\{0,1\}^{k+1}})$. Lemma \ref{siggen} shows that $h\circ\psi_0$ is orthogonal to $L^2(\mathcal{C})$ which is needed for conditional independence.

Now lemma \ref{siggen2} implies that $f\circ\psi_0$ is measurable in
$(\mathcal{C}\wedge\mathcal{A}\circ\psi_0)\vee\mathcal{B}\circ\psi_0$ which by theorem \ref{propfk} is equal to $\mathcal{F}_k\circ\psi_0\vee\mathcal{B}\circ\psi_0$.

\end{proof}


\begin{lemma}\label{charsep} If $\phi$ is a $k$-th order character then $\phi\in[\mathcal{F}_{k-1},k]^*$.
\end{lemma}

\begin{proof} It is clear that for every $t\in C^{k+1}(\bA)$ we have that $\Delta_t (\phi)_{k+1}^\times$ is measurable in $\mathcal{B}=(\mathcal{F}_{k-1})_{k+1}^\times$. Using that $\mathcal{B}$ is shift invariant and lemma \ref{erosort} we obtain that either $(\phi)_{k+1}^\times$ is measurable in $\mathcal{B}$ or $E((\phi)_{k+1}^\times|\mathcal{B})=0$.
However the second possibility is impossible since the integral of $(\phi)_{k+1}^\times$ is equal to the $k+1$-th Gowers norm of $\phi$ which is positive because $\phi\in\mathcal{F}_k$.
\end{proof}

\begin{lemma}\label{charchar} A function $\phi:\bA\rightarrow\mathbb{C}$ is a $k$-th order character if and only if $\phi\in[\mathcal{F}_{k-1},k]^*$. 
\end{lemma}

\begin{proof} Lemma \ref{charsep} shows one implication. To see the other implication assume that $\phi\in[\mathcal{F}_{k-1},k]^*$. Then if $t\in\bA$ is arbitrary we have by lemma \ref{pureprop} that $\Delta_t\phi\in [\mathcal{F}_{k-1},k-1]^*$. By lemma \ref{charmes} we obtain that $\Delta_t\phi$ is measurable in $\mathcal{F}_{k-1}$ showing that $\phi$ is a $k$-th order character.
\end{proof}

\begin{lemma}\label{charconv} Let $n,k\in\mathbb{N}^+$. Let $\beta:K_n\rightarrow\hat{\bA}_k$ be a map such that $\beta\notin\hom(K_n,\mathcal{D}_{n-k-1}(\hat{\bA}_k))$ and let $F=\{f_v\}_{v\in K_n}$ be a system of bounded functions on $\bA$ such that $f_v$ is in $\beta(v)$. Then $[F]$ is identically $0$.
\end{lemma} 

\begin{proof} We go by induction on $n$. If $n=1$ then lemma \ref{erosort} shows the claim. Assume that the statement holds for $n-1\geq 1$.
Let $S\subset K_n$ be a face such that $\sum_{v\in S}\beta(v)\neq 0$ and assume that $i\in [n]$ is a direction parallel to $S$. Then $\delta_{i,t} F$ satisfies the condition for $n-1$ and thus by induction we have that $[\delta_{i,t} F]$ is identically $0$. The statement for $n$ follows from (\ref{dimred}). 
\end{proof}

\subsection{Topological nilspace factors of ultra product groups}

\bigskip

\begin{definition} A continuous surjective function $\gamma:\bA\rightarrow T$ into a compact Hausdorff space $T$ is a {\bf $k$-step nilspace factor} of $\bA$ if the cubespace structure on $T$ (obtained by composing cubes in $\bA$ with $f$) is a $k$-step nilspace. Equivalently, a $k$-step nilspace factor of $\bA$ is given by a continuous nilspace morphism $\gamma:\bA\rightarrow N$ into a compact $k$-step nilspace such that the induced maps $\gamma^n:C^n(\bA)\rightarrow C^n(N)$ are surjective for every $n$. 
\end{definition}

Note that nilspace factors of ultra product groups are automatically compact nilspaces.
Throughout this chapter $N$ is a $k$-step nilspace with structure groups $A_1,A_2,\dots,A_k$ and $i$-step factors $\pi_i:N\rightarrow N_i$ where $0\leq i\leq k$.
We will need the following lemma.

\begin{lemma}\label{homok} Let $\gamma:\bA\rightarrow N$ be a $k$-step nilspace factor of $\bA$. Then $\gamma$ is measurable in $\mathcal{F}_k(\bA)$ and there are homomorphism $\tau_k:A_i\rightarrow\hat{\bA}_i$ for $1\leq i\leq k$ such that if $\chi\in\hat{A_i}$ and $f\in W(\chi,N_i)$ then $f\circ\pi_i\circ\gamma$ is in the module $\tau_i(\chi)$. 
\end{lemma}

\begin{proof} By lemma \ref{kisk} and lemma \ref{nilfourdec} the calim that $\gamma$ is measurable in $\mathcal{F}_k$ follows from the existence of the map $\tau_k$. By induction on $k$ we can assume that the maps $\{\tau_i\}_{i=1}^{k-1}$ exist as required by the lemma. To see the existence of $\tau_k$ let $\chi\in\hat{A}_k$ be an arbitrary character and let $\phi\in V(\chi,N)$ be a function of absolute value $1$ (guaranteed by lemma \ref{nilspchar}). We have that $(\phi)_{k+1}^\times\circ\gamma^{k+1}=(\phi\circ\gamma)_{k+1}^\times$.
From our induction hypothesis, lemma \ref{nilspcharderiv}, lemma \ref{charchar} we obtain that $\phi\circ\gamma$ is a $k$-th order character that we denote by $\tau_k(\bA)$. Then lemma \ref{nilspchar} and our induction assumption will guarantee that $\tau_k(\chi)$ depends only on $\chi$ and thus $\tau_k:\hat{A}_k\rightarrow\hat{\bA}_k$ is well defined.
The fact that $\tau_k$ is a homomorphism is clear from the definitions.
\end{proof}

We introduce the following four properties for nilspace factors.

\medskip

\noindent{\bf 1.)~measure preserving:}~We say that the nilspace factor $\xi:\bA\rightarrow N$ is measure preserving if all the maps $\gamma^n:C^n(\bA)\rightarrow C^n(N)$ are measure preserving.

\medskip

\noindent{\bf 2.)~Rooted measure preserving:}~We say that $\gamma$ is rooted measure preserving if for every $a\in\bA$ and natural number $n\in\mathbb{N}$ the map $\gamma^n_a:C^n_a(\bA)\rightarrow C^n_{\gamma(a)}(N)$ induced by $\gamma$ is measure preserving.

\medskip

\noindent{\bf 3.)~Character preserving:}~The homomorphisms $\tau_i:\hat{A}_k\rightarrow\hat{\bA}_k$ defined in lemma \ref{homok} are all injective.

\medskip

\noindent{\bf 4.)~Factor consistent:}~We say that $\gamma:\bA\rightarrow N$ is factor consistent if for every $1\leq i\leq k$ and bounded measurable function $f:N\rightarrow\mathbb{C}$ the function $\mathbb{E}(f\circ\gamma|\mathcal{F}_i(\bA))$ is measurable in the $\sigma$-algebra of $\pi_i\circ\gamma$.

\medskip

\begin{theorem}\label{charpres} Let $\gamma:\bA\rightarrow N$ be a $k$-step nilspace factor. Then the following statements are equivalent.
\begin{enumerate}
\item $\gamma$ is character preserving,
\item $\gamma$ is rooted measure preserving,
\item $\gamma$ is measure preserving,
\item $\gamma$ is factor consistent.
\end{enumerate}
\end{theorem} 

\begin{definition} A nilspace factor $\gamma$ satisfying the equivalent conditions in theorem \ref{charpres} is called a {\bf strong} nilspace factor. 
\end{definition}

\bigskip

\noindent{\it Proof of theorem \ref{charpres}:} We prove the statement by induction on $k$. If $k$ is $0$ then everything is trivial. Assume that the statement is true for $k-1$. This means that the factor $\gamma_{k-1}$ satisfies all the conditions simultaneously.

\noindent~$(1)\Rightarrow (2)$:~~
Let $x\in\bA$ be an arbitrary element and let $y=\gamma(x)$.
Let $\nu$ denote the probability distribution on $C_y^n(N)$ obtained by composing the uniform distribution on $C_x^n(\bA)$ by $\gamma$. By theorem \ref{bundec} $C_y^n(N)$ is a $C^n_0(\mathcal{D}_k(A_k))$-bundle over $C^n_{\pi_{k-1}(y)}(N_{k-1})$. Thus by induction it is enough to show that $\nu$ is invariant under the natural action of $C_0^n(\mathcal{D}_k(A_k))$ on $C_y^n(N)$.
This invariance can be proved by showing for a function system $U$ which linearly spans an $L^1$-dense set in $L^\infty(C_y^n(N))$ that for every $r\in C_0^n(\mathcal{D}_k(A_k))$ and $u\in U$ the equation $\mathbb{E}_\nu(u)=\mathbb{E}_\nu(u^r)$ holds. (The shift $u^r$ of $u$ is defined as the function satisfying $u^r(y)=u(y+r)$.)
We define $U$ as the collection of all functions $[F]^\times$ 
on $C_y^n(N)$ where $\{\chi_v\in\hat{A_k}\}_{v\in K_n}$ is a system of characters and $f_v\in W{\chi_v,N)}$ is a continuous function for every $v\in K_n$.
The set $U$ is closed under multiplication and contains a separating system of continuous functions. It follows from the Stone-Weierstrass theorem that every function in $L^\infty(C_y^n(N))$ can be approximated by some finite linear combination of elements from $U$.
We have for $r\in C_0^n(\mathcal{D}_k(A_k))$ that   
\begin{equation}\label{chmult}
([F]^\times)^r=[F]^\times\prod_{v\in K_n}\chi_v(r_v).
\end{equation}
where $r_v=\psi_v(r)$ is the component of $r$ at $v$.
There are two cases.
In the first case the function $\beta: v\rightarrow \chi_v$ is not in  $\hom(K_n,\mathcal{D}_{n-k-1}(\hat{A_k}))$. In this case by the character preserving property of $\gamma$ and by lemma \ref{charconv} we get that $\mathbb{E}_\nu([F]^\times)=\mathbb{E}_\mu([F\circ\gamma]^\times)=0$ and so by (\ref{chmult}) we have $\mathbb{E}_\nu([F]^\times)=0=\mathbb{E}_\nu(([F]^\times)^r)$.
In the second case $\beta\in\hom(K_n,\mathcal{D}_{n-k-1}(\hat{A}_k))$. Then we have by lemma \ref{dualker2} that $\prod_{v\in K_n}\chi_v(r_v)=1$.
\bigskip

\noindent$(2)\Rightarrow(3):$~~ Notice that a random element in $C_x^{n+1}(\bA)$ (resp. $C_{\gamma(x)}^{n+1}(N)$) restricted to an $n$ dimensional face of $\{0,1\}^{n+1}$ not containing $0^{n+1}$ is $C^n(\bA)$ (resp. $C^n(N)$) with the uniform distribution. 

\bigskip

\noindent$(3)\Rightarrow(4):$~~ Let $f:N\rightarrow\mathbb{C}$ be anarbitrary measurable function and let $f_1=\mathbb{E}(f|\pi_i)~,~f_2=f-f_1$. We have by lemma \ref{homok} that $f_1\circ\gamma$ is measurable in $\mathcal{F}_i$. By lemma \ref{gownilproj} we have that $\|f_2\|_{U_{i+1}}=\|\mathbb{E}(f_2|\pi_i)\|_{U_{i+1}}=0$. By the measure preserving property of $\gamma$ we get that $$0=\|f_2\|_{U_{i+1}}^{2^{i+1}}=\mathbb{E}((f_2)^\times)=\mathbb{E}((f_2\circ\gamma)^\times)=\|f_2\circ\gamma\|_{U_{i+1}}^{2^{i+1}}.$$ It follows from theorem \ref{propfk} that $f_1\circ\gamma=\mathbb{E}(f\circ\gamma|\mathcal{F}_i)$.

\bigskip

\noindent$(4)\Rightarrow(1):$~~ Let $f:N\rightarrow\mathbb{C}$ be a function in $W(\chi,N)$ of absolute value $1$.  
We show that if $\chi$ is non-trivial then $f\circ\gamma$ is non-trivial in $\hat{\bA}_k$. This is equivalent with saying that $f\circ\gamma$ is not measurable in $\mathcal{F}_{k-1}(\bA)$.
Assume by contradiction that $g=f\circ\gamma$ is measurable in $\mathcal{F}_{k-1}$. We have by factor consistency that there is a function $h:N_{k-1}\rightarrow\mathbb{C}$ such that $f\circ\gamma=h\circ\pi_{k-1}\circ\gamma$ almost everywhere on $\bA$. Our induction hypothesis implies that the factor $\pi_{k-1}\circ\gamma$ is rooted measure preserving. This means by lemma \ref{nilspcharderiv} that for every fix $x\in\bA$ we have $$[f\circ\gamma]_{k+1}(x)=\mathbb{E}([f\circ\gamma]_{k+1}^\times(x))=\mathbb{E}([f]_{k+1}^\times(\gamma(x)))=[f]_{k+1}(\gamma(x)).$$
By $([f\circ\gamma]_{k+1},\overline{f})=\|f\circ\gamma\|^{2^{k+1}}_{U_{k+1}}\neq 0$ we have that $[f\circ\gamma]_{k+1}$ is not identically $0$.  
Furthermore by the induction hypothesis we have that
$$[h]_{k+1}(\pi_{k-1}(\gamma(x)))=[h\circ\pi_{k-1}\circ\gamma]_{k+1}(x)=[f\circ\gamma]_{k+1}(x).$$
It follwos that $[h\circ\pi_{k-1}]_{k+1}=[f]_{k+1}$ holds on $N$. This is a contradiction since by lemma \ref{nilspcharderiv} we have that $[f]_{k+1}\in W(\overline{\chi},N)$ and so it does not factor through $\pi_{k-1}$.

\medskip

\begin{lemma}\label{strongcirc} Let $\gamma:\bA\rightarrow N$ be a strong $k$-step nilspace factor and let $\phi:N\rightarrow M$ be a fibre surjective morphism. Then $\phi\circ\gamma$ is a strong nilspace factor.
\end{lemma}

\begin{proof} Using the fact that $\phi$ induces measure preserving maps from $C^n(N)$ to $C^n(M)$ it follows that $\phi\circ\gamma$ is measure preserving. Then theorem \ref{charpres} shows the claim.
\end{proof}

\section{The main theorem}

We are ready to state the main theorem of this paper which is our crucial tool to describe higher order Fourier analysis.
The result is a decomposition theorem of an arbitrary bounded measurable function on $\bA$ into a $k$-th ordered structured part and a $k$-th order random part where randomness is measured by the Gowers norm $U_{k+1}$.

\begin{theorem}[Main theorem]\label{main} Let $f:\bA\rightarrow\mathbb{C}$ be an arbitrary bounded measurable function and $k\in\mathbb{N}$ be a natural number. Then there is a strong $k$-step nilspace factor $\gamma:\bA\rightarrow N$ and a Borel measurable function $h:N\rightarrow\mathbb{C}$ such that $\|f-h\circ\gamma\|_{U_{k+1}}=0$.
\end{theorem}

\begin{remark} Note that in the theorem \ref{main} the decomposition $f=f_s+f_r$ with $f_s=h\circ\gamma$ and $\|f_r\|_{U_{k+1}}=0$ is unique since $f_s=\mathbb{E}(f|\mathcal{F}_k)$. This follows from the fact that $f_s$ is measurable in $\mathcal{F}_k$ and $f_r$ is orthogonal to $L^2(\mathcal{F}_k)$.
\end{remark}

An equivalent formulation of the main theorem is the following version of it.

\begin{theorem}[Second version of the main theorem]\label{main2} A function $f:\bA\rightarrow\mathbb{C}$ is measurable in $\mathcal{F}_k$ if and only if $f=h\circ\gamma$ for some strong $k$-step nilspace factor $\gamma:\bA\rightarrow N$ and Borel measurable function $h:N\rightarrow\mathbb{C}$.
\end{theorem}

It is easy to see that in this statement we don't need to assume that $f$ is bounded. Various strengthenings of the main theorem can also be formulated. These statements have essentially the same proof.

\begin{theorem}[Third version of the main theorem]\label{main3} For every separable $\sigma$-algebra $\mathcal{G}\subset\mathcal{F}_k$ there is a separable $k$-step nilspace factor $\gamma:\bA\rightarrow N$ such that $\mathcal{G}$ is contained in the $\sigma$-algebra generated by $\gamma$.
\end{theorem}

\begin{remark}[Affine invariance]  In theorem \ref{main3} one can also assume that the topological factor $\gamma$ is invariant under a countable set of prescribed invertible affine transformations of the form $\alpha:\bA\rightarrow\bA$ defined by $\alpha(x)=nx+a,~n\in\mathbb{N},a\in\bA$. This helps in connecting our results with classical ergodic theory. 
\end{remark}


\bigskip

\subsection{Outline of the proof}\label{sketch}

In this chapter we list the major components of the proof of theorem \ref{main3}.

\bigskip

\noindent{\bf Step 1. (Finding the $\sigma$-algabra)}~Before finding $\gamma$ we construct a weaker object namely the $\sigma$-algebra generated by $\gamma$. The idea is to mimic the properties of $\mathcal{F}_k$ by a separable sub $\sigma$-algabra containing $\mathcal{G}$. 

\begin{definition} Let $\mathcal{B}\subset\mathcal{A}$ be a $\sigma$-algebra. We say that $\mathcal{B}$ is a {\bf nil $\sigma$-algebra} (of order $k$) if $\mathcal{B}=[\mathcal{B}]_{k+1}$ and for every $2\leq i\leq k+1$ we have that on $C^i(\bA)$ 
\begin{equation}\label{rhsnil}
[\mathcal{B}]_i\circ\psi_0=\mathcal{A}\circ\psi_0\wedge[[\mathcal{B}]_i]_i^\times
\end{equation}
\end{definition}

With the help of the next lemma we extend $\mathcal{G}$ into a separable nil $\sigma$-algebra $\mathcal{B}$.

\begin{lemma}\label{embednil} Every separable sigma algebra in $\mathcal{F}_k$ is contained in a separable nil $\sigma$-algebra of order $k$.
\end{lemma}

In the rest of the proof it will be enough to show that each separable nil $\sigma$-algebras is generated by a strong nilspace factor. 

\bigskip

\noindent{\bf Step 2.~(Topologization)}~In this step we construct a topological factor from a separable nil $\sigma$-algebra $\mathcal{B}\subset\mathcal{F}_k$. Let $\gamma:\bA\rightarrow N$ denote the topological factor generated by all convolutions $[F]$ where $\{f_v\}_{v\in K_{k+1}}$ is a function system in $L^\infty(\mathcal{B})$. 
We say that a function $f$ is continuous (resp. a set $S\subseteq\bA$ is open) in $\gamma$ if $f=h\circ\gamma$ (resp. $S=\gamma^{-1}(S')$) for some continuous function $h$ (resp. open set $S'$) on $N$.
Note that $N$ has an inherited cubic structure $\{C^n(N)\}_{n=0}^\infty$ which arises by composing cubes in $\bA$ with $\gamma$.
It will be a useful point of view that $\gamma$ is the topological factor generated by the single map $x\rightarrow \mathcal{B}\circ\Psi^{k+1}_x$ in the coupling topology.
To prove theorem \ref{main3} it is enough to show the following proposition.

\begin{proposition}[topologization]\label{topologization}If $\mathcal{B}$ is a separable nil $\sigma$-algebra of order $k$ then the corresponding topological factor $\gamma:\bA\rightarrow N$ (generated by $k+1$-th order convolutions) is a $k$-step strong nilspace factor. 
\end{proposition}

An important observation is that the first two nilspace axioms automatically hold in $N$.
The proof of proposition \ref{topologization} deals with the checking of the third nilspace axiom. 
(It will be clear that $\gamma$ is factor consistent and thus it gives a strong nilspace factor.)

We fix a natural number $k$ and by induction we assume that proposition \ref{topologization} is true for $k-1$.
Note that the statement is trivial for $k=0$. Let us introduce the notation  $\mathcal{B}_i:=[\mathcal{B}]_{i+1}=\mathcal{B}\cap\mathcal{F}_i$ for $1\leq i\leq k+1$.
It will be crucial that, by induction, the topological factor $\gamma_{k-1}:\bA\rightarrow N_{k-1}$ corresponding to $\mathcal{B}_{k-1}$ is a strong $k-1$ step nilspace factor. Since $\gamma_{k-1}$ generates a courser topology than $\gamma$ we have a natural projection $\pi_{k-1}:N\rightarrow N_{k-1}$.

\bigskip

\noindent{\bf Step 3.~(Local properties of $\mathcal{B}$)}~ We prove the following measure theoretic analogy of the unique gluing aximom in $k$-step nilspaces. 

\begin{lemma}\label{compdep} Let $\mathcal{B}$ be a nil-$\sigma$-algebra of order $k$ and $x\in\bA$. Then the coupling $\mathcal{B}\circ\Psi^{k+1}_x$ is completely dependent. 
\end{lemma}

\bigskip

\noindent{\bf Step 4.~(Convolutions of open sets)}~For a set $S\subset \bA$ let $\cl(S)$ denote the closure of $S$ in the topolgy generated by $\gamma$. We prove the next topological statement which is a preparation for the proof of the unique gluing axiom for $N$. 

\begin{lemma}\label{uclos} Let $\varrho:K_{k+1}\rightarrow\bA$ be some function. Then there is at most one element $z
\in N$ with the following property. For every system of $\gamma$-open neighborhoods $U(v)$ of $\varrho(v)$ where $v$ runs through $K_{k+1}$ we have that $\gamma^{-1}(z)\subset\cl({\rm supp}([F]))$ where $F=\{1_{U(v)}\}_{v\in K_{k+1}}$. 
\end{lemma}

Note that the only application of lemma \ref{compdep} is in the proof of lemma \ref{uclos}.
These two lemmas together form the most technical part of the proof.

\bigskip

\noindent{\bf Step 5.~(Support of measure)}~To prove the unique gluing axiom for $\gamma$ we have to check that if $\varrho:K_{k+1}\rightarrow\bA$ is a morphism then in lemma \ref{uclos} the set $\cl({\rm supp}([F]))$ is not empty.
This will follow by analyzing the support of the measure in the topology generated by $\gamma$ on $C^n(\bA)$ and $C_x^n(\bA)$. 

\begin{definition}[Positive cubes]\label{defposcube} Let $n\in\mathbb{N},x\in\bA$ arbitrary. We say that $c\in C^n(\bA)$ (resp $c\in C^n_x(\bA)$) is positive if it is in the support of the uniform measure on $C^n(\bA)$ (resp. $C^n_x(\bA)$) with respect to the topology generated by $\gamma$. 
\end{definition}

Note that it is not clear from the above definition that if $c\in C^n_x(\bA)$ is positive then $c$ is also positive in the space $C^n(\bA)$. To avoid confusion we will always emphasize the space in which $c$ is positive.
The main result in this part of the proof is that every cube is positive.

\subsection{Nil $\sigma$-algabras}

\begin{lemma}\label{weaknilprop} Let $\mathcal{B}\subset\mathcal{A}$ be a $\sigma$-algebra such that $[\mathcal{B}]_n\subseteq\mathcal{B}$. Then for every $1\leq i\leq j\leq n$ we have the following statements.
\begin{enumerate}
\item For every $f\in L^\infty(\mathcal{B})$ we have $\mathbb{E}(f|[\mathcal{B}]_i)=\mathbb{E}(f|\mathcal{F}_{i-1})$,
\item $[\mathcal{B}]_i=\mathcal{B}\cap\mathcal{F}_{i-1}$~,
\item $[[\mathcal{B}]_i]_j=[\mathcal{B}]_j$.
\end{enumerate}
\end{lemma}

\begin{proof} Let $f'=\mathbb{E}(f|\mathcal{F}_{i-1})$ and Let $g=f-\mathbb{E}(f'|[\mathcal{B}]_i)$. Using the fact that $[\mathcal{B}]_i\subseteq\mathcal{F}_{i-1}$ we obtain that if $h\in L^\infty([\mathcal{B}]_i)$ then $$(g,h)=(\mathbb{E}(g|\mathcal{F}_{i-1}),h)=(f'-\mathbb{E}(f'|[\mathcal{B}]_i),h)=0.$$ We have by $[g]_i\in L^\infty([\mathcal{B}]_i)$ that $\|g\|_{U_i}^{2^i}=(g,\overline{[g]})=0$. It follows that $0=\mathbb{E}(g|\mathcal{F}_{i-1})$ and thus $f'=\mathbb{E}(f'|[\mathcal{B}]_i)$ implying that $f'\in L^\infty([\mathcal{B}]_i)$. Using again that $[\mathcal{B}]_i\subseteq\mathcal{F}_{i-1}$ the proof of the first statement is complete. The second statement follows immediately from the first one.
To see the third statement observe that $[[\mathcal{B}]_i]_i\subseteq [\mathcal{B}]_i$ and thus the second statement applied for $[\mathcal{B}]_i$ we obtain that $[[\mathcal{B}]_i]_j=[\mathcal{B}]_i\cap\mathcal{F}_{j-1}=(\mathcal{B}\cap\mathcal{F}_{i-1})\cap\mathcal{F}_{j-1}
=\mathcal{B}\cap\mathcal{F}_{j-1}=[\mathcal{B}]_j$. 
\end{proof}

\medskip

An immediate corollary of lemma \ref{weaknilprop} is the following,

\begin{corollary}\label{nilsubnil} If $\mathcal{B}\subset\mathcal{A}$ is a nil $\sigma$-algebra of order $k$ and $1\leq i\leq k+1$ then $[\mathcal{B}]_{i+1}=\mathcal{B}\cap\mathcal{F}_i$ is a nil $\sigma$-algebra of order $i$.
\end{corollary}

\medskip

\noindent{\it Proof of lemma \ref{embednil}}~Starting with a separable $\sigma$-algebra $\mathcal{G}\subset\mathcal{F}_k$ we construct an increasing sequence of separable $\sigma$-algebras $\mathcal{G}=\mathcal{G}_0\subset\mathcal{G}_1\subset\mathcal{G}_2\subset\dots$ in $\mathcal{F}_k$ in the following way. Assume that $\mathcal{G}_n$ is already constructed. 
Theorem \ref{propfk} together with lemma \ref{sepsiggen} imply that for every $2\leq i\leq k+1$ there is a separable $\sigma$-algebra $\mathcal{D}_i\subset\mathcal{F}_{i-1}$ such that $[\mathcal{G}_n]_i\circ\psi_0\subset [[\mathcal{D}_i]_i]_i^\times$. We define  $\mathcal{G}_{n+1}=(\vee_{i=2}^{k+1}\mathcal{D}_i)\vee\mathcal{G}_n\vee[\mathcal{G}_n]_{k+1}$. Let $\mathcal{B}=\vee_{n=1}^\infty\mathcal{G}_n$. We claim that $\mathcal{B}$ is a nil $\sigma$-algebra. 

Since $\{\mathcal{G}_i\}_{i=0}^\infty$ is a chain we have that for every set $S$ in $\mathcal{B}$ and $\epsilon>0$ there is an index $j$ and set $S'\in\mathcal{G}_j$ such that $\bm(S\triangle S')\leq\epsilon$. Furthermore for every function in $f\in L^\infty_u(\mathcal{B})$ and $\epsilon>0$ there is an index $j$ and function $f'\in L_u^\infty(\mathcal{G}_j)$ such that $\|f-f'\|_2\leq\epsilon$. We say that $S'$ (resp. $f'$) is a finite index $\epsilon$-approximation of $S$ (resp. $f$).

Let $F=\{f_v\}_{v\in K_{k+1}}$ be a function system in $L^\infty_u(\mathcal{B})$. Then the convolution $[F]$ can be approximated arbitrarily well by a convolution of finite index $\epsilon$-approximations of the function system. On the other hand such convolutions are contained in some memeber of the chain $\{\mathcal{G}_n\}_{n=1}^\infty$. It follows that  $[\mathcal{B}]_{k+1}\subseteq\mathcal{B}$. Since $\mathcal{B}\subset\mathcal{F}_k$ we have by lemma \ref{weaknilprop} that $[\mathcal{B}]_{k+1}=\mathcal{B}\cap\mathcal{F}_k=\mathcal{B}$.
Let $\mathcal{C}_i$ be the unique sigma algebra on $\bA$ with $\mathcal{C}_i\circ\psi_0=\mathcal{A}\circ\psi_0\cap[[\mathcal{B}]_i]_i^\times$.
Similarly to the case of convolutions, by considering finite index approximations, we have for $2\leq i\leq k+1$ that $[\mathcal{B}]_i
\subseteq\mathcal{C}_i$. On the other hand we have by lemma \ref{convcontlift} that $\mathcal{C}_i\subset[\mathcal{B}]_i$.

\subsection{Local properties of $\mathcal{B}$}

In this chapter we prove lemma \ref{compdep}. By induction we assume that the statement is true for $k-1$. We start by proving the following statement using the induction hypothesis.

\medskip

\noindent{\bf Claim 1.}~~{\it Assume that $S_1,S_2\subset\{0,1\}^n$ are simplicial sets such that $S_1$ has dimension at most $k$. Let $w\in S_2\setminus S_1$. On $C^n_x(\bA)$ let
$$\mathcal{D}_1=\bigvee_{v\in (S_1\cup S_2)\setminus\{0,w\}}\mathcal{B}\circ\psi_v~~~~~~{\it and}~~~~~~\mathcal{D}_2=\bigvee_{v\in S_2\setminus\{0,w\}}\mathcal{B}\circ\psi_v.$$
Assume that on $C^n_x(\bA)$ that the $\sigma$-algebra $\mathcal{B}\circ\psi_w$ is contained in $\mathcal{D}_1$ then it is also contained in $\mathcal{D}_2$.}

\medskip

We prove the statement by induction on the size of $S_1\setminus S_2$. If $S_1\subset S_2$ then there is nothing to prove. Assume that $S_1\nsubseteq S_2$ and $u$ is a maximal element of $S_1$ which is not contained in $S_2$. Then $d=d(u)=h(u)\leq k$.
On the space $C^n_x(\bA)$ let $$\mathcal{D}_3=\bigvee_ {(S_1\cup S_2)\setminus\{0,u\}}\mathcal{B}\circ\psi_v~~~~~~{\rm and}~~~~~~\mathcal{D}_4=\bigvee_ {(S_1\cup S_2)\setminus\{0,u,w\}}\mathcal{B}\circ\psi_v.$$
Our induction hypothesis of the lemma on $k$ guarantees that $\mathcal{B}_{d-1}\circ\psi_u$ is contained in $\bigvee_{0\neq v<u}\mathcal{B}_{d-1}\circ\psi_v$ which is contained in $\mathcal{D}_4$.

Assume that $f\in L^\infty(\mathcal{B})$ has the property that $\mathbb{E}(f|\mathcal{B}_{d-1})=0$. Then by lemma \ref{weaknilprop} we have that $\mathbb{E}(f|\mathcal{F}_{d-1})=0$ and thus $\|f\|_{U_d}=0$. It follows by lemma \ref{simpzero} that $f\circ\psi_u$ is orthogonal to the $\sigma$-algebra $\mathcal{D}_3$.
This means that $\mathcal{B}\circ\psi_u$ and $\mathcal{D}_3$ are conditionally independent. Furthermore $\mathcal{B}\circ\psi_0\wedge\mathcal{D}_3=\mathcal{B}_{d-1}\circ\psi_u\subset\mathcal{D}_4$.
We get by using lemma \ref{siggen2} that $$\mathcal{B}\circ\psi_w\subset\mathcal{D}_1\wedge\mathcal{D}_3=(\mathcal{B}\circ\psi_u\vee\mathcal{D}_4)\wedge\mathcal{G}_3=(\mathcal{B}\circ\psi_u \wedge\mathcal{D}_3)\vee\mathcal{D}_4=\mathcal{D}_4.$$ 
Since $S_1\setminus\{u\}$ is a simplicial set, by induction we have that $\mathcal{B}\circ\psi_w$ is contained in $\mathcal{D}_2$.

\bigskip

Now we switch to the proof of the lemma. Let $z\in K_{k+1}$ be arbitrary. Our goal is to show that on $C^{k+1}_x(\bA)$ the $\sigma$-algebra $\mathcal{B}\circ\psi_z$ is generated by the system $\{\mathcal{B}\circ\psi_v\}_{z\neq v\in K_{k+1}}$. Let $T=\{0,1\}^{[k+1]\times[2]}$. 
We will need the following special subsets of $T$. 
$$S_1=\{v~|~v_{i,1}=0~{\rm for}~1\leq i\leq k+1\} ~~,~~S_2=\{v~|~v_{i,2}=0~{\rm for}~1\leq i\leq k+1\},$$
$$S_3=\{v~|~v_{k+1,2}=0~,~h(v)\leq k~,~v_{i,1}v_{i,2}=0~{\rm for}~1\leq i\leq k\},$$
$$S_4=\{v~|~v_{k+1,2}=0~,~v_{i,1}v_{i,2}=0~{\rm for}~1\leq i\leq k\}.$$
It is clear that all the sets $S_1,S_2,S_3,S_4$ are simplicial. 
Without loss of generality we can assume that the last coordinate of $z$ is $1$.
Let $w$ be the vector such that $w_{i,1}=0$ and $w_{i,2}=z_i$ for $1\leq i\leq k+1$.
On $\hom_{0\mapsto x}(T,\bA)$ let $\mathcal{G}_1=\bigvee_{v\in S_1\setminus\{w,0\}}\mathcal{B}\circ\psi_v$ and for $2\leq i\leq 4$ let
$\mathcal{G}_i=\bigvee_{v\in S_i\setminus\{0\}}\mathcal{B}\circ\psi_v$.
The statement of the lemma is equivalent with the fact that $\mathcal{B}\circ\psi_w$ is contained in $\mathcal{G}$. 

\medskip

\noindent{\bf Claim 2.}~$\mathcal{B}\circ\psi_w\subseteq \mathcal{G}_1\vee\mathcal{G}_4$

\medskip

Let $\phi:\{0,1\}^{k+1}\rightarrow T$ be defined such that $\phi(v)_{i,1}=0~,~\phi(v)_{i,2}=v_i$ if $1\leq i\leq k$ and $\phi(v)_{k+1,1}=1-v_{k+1}~,~\phi(v)_{k+1,2}=v_{k+1}$. We have that $\phi(z)=w$. Let $K=\{0,1\}^{k+1}\setminus\{z\}$. If $v\in K$ then $\phi(v)\in S_1\cup S_4$ and thus $\mathcal{B}\circ\psi_{\phi(v)}\in\mathcal{G}_1\vee\mathcal{G}_4$.
Since by lemma \ref{cupis1} and lemma \ref{cupis} the coupling $\{\psi_{\phi(v)}\}_{v\in\{0,1\}^{k+1}}$ is the same as $\Psi^{k+1}$ by (\ref{rhsnil}) applied for $i=k+1$ we get that $\mathcal{B}\circ\psi_w$ is generated by $\bigvee_{v\in K}\mathcal{B}\circ\psi_{\phi(v)}$. This shows the claim. 

\medskip

\noindent{\bf Calim 3.}~ $\mathcal{G}_2\vee\mathcal{G}_3=\mathcal{G}_4$.

\medskip

The containment $\subseteq$ is trivial. For a vector $v$ in $T$ let $h^*(v)=\sum_{i=1}^{k+1}v_{i,2}$.
We prove by induction on $h^*(v)$ that if $v\in S_4$ then $\mathcal{B}\circ\psi_v$ is in $\mathcal{G}_2\vee\mathcal{G}_3$. If $h^*(v)=0$ then $v\in S_2$ and the statement is trivial.
Assume that the statement holds for every $v\in S_4$ with $h^*(v)\leq n-1$ (where $n\geq 1$)  and that $h^*(b)=n$ for some $b\in S_4$.
If $h(b)\leq k$ then $b\in S_3$ and the statement is trivial. We can assume that $h(b)=k+1$.
This means that $b_{i,1}+b_{i,2}=1$ for every $1\leq i\leq k+1$. Let $\phi:\{0,1\}^{k+1}\rightarrow S_4$ be defined such that $\phi(v)_{i,1}=v_i~,~\phi(v)_{i,2}=0$ if $b_{i,2}=0$ and $\phi(v)_{i,1}=1-v_i~,~\phi(v)_{i,2}=v_i$ if $b_{i,2}=1$. It is clear that $\phi$ is an injective cube morphism and that the image of $\phi$ does not contain $0$. Let $K=\{0,1\}^{k+1}\setminus\{1\}$. For every $v\in K$ we have that either $h^*(\phi(v))<n$ or $\phi(v)\in S_3$. Using the induction hypothesis, in both cases $\mathcal{B}\circ\psi_{\phi(v)}$ is in $\mathcal{G}_2\vee\mathcal{G}_3$. Since by lemma \ref{cupis1} and lemma \ref{cupis} the coupling $\{\psi_{\phi(v)}\}_{v\in\{0,1\}^{k+1}}$ is the same as $\Psi^{k+1}$ we obtain by (\ref{rhsnil}) and $b=\phi(1)$ the $\mathcal{B}\circ\psi_b$ is generated by $\bigvee_{ v\in K}\mathcal{B}\circ\psi_{\phi(v)}$ which is conatined in $\mathcal{G}_2\vee\mathcal{G}_3$. This finishes the proof of the claim.

\medskip

By claim 2. and claim 3. we obtain that $\mathcal{B}\circ\psi_w\subseteq\mathcal{G}_1\vee\mathcal{G}_2\vee\mathcal{G}_3$. Now claim 1. shows that $\mathcal{G}_3$ can be omitted and thus $\mathcal{B}\circ\psi_w\subseteq\mathcal{G}_1\vee\mathcal{G}_2$. On the other hand since $\mathcal{G}_2$ is independent from $\mathcal{G}_5=\mathcal{B}\circ\psi_w\vee\mathcal{G}_1$ we obtain from lemma \ref{siggen2} that $$\mathcal{B}\circ\psi_w\subset (\mathcal{G}_2\vee\mathcal{G}_1)\wedge\mathcal{G}_5=
(\mathcal{G}_2\wedge\mathcal{G}_5)\vee\mathcal{G}_1=\mathcal{G}_1.$$

\subsection{Convolutions of open sets}\label{conopen}

In this chapter we prove lemma \ref{uclos}.
Assume that there exists such an element $z$ and that $z=\gamma(x)$ for some $x\in\bA$.
The statement of the lemma is equivalent with saying that the coupling $\mathcal{B}\circ\Psi_x^{k+1}$ is uniquely determined by the couplings $\mathcal{B}\circ\Psi_{\varrho(v)}^{k+1}$ where $v$ runs through $K_{k+1}$. 
Our strategy is to put all these couplings into one big coupling $\Upsilon$ as sub-couplings and then we do the calculations in $\Upsilon$. The first part of the proof deals with the construction of $\Upsilon$.

\medskip

\noindent{\bf Construction of $\Upsilon$:}~~For every $v\in K_{k+1}$ let us choose a decreasing sequence $\{U_i(v)\}_{i=1}^\infty$ of $\gamma$ neighborhoods of $\varrho(v)$ which forms a neighborhood basis.
Let $F_i=\{1_{U_i(v)}\}_{v\in K_{k+1}}$ and $\{x_i\}_{i=1}^\infty$ be a sequence such that $\lim_{i\to\infty}\gamma(x_i)=z$ and $x_i\in{\rm supp}([F_i])$.
Let $T=\{0,1\}^{[k+1]\times[2]}$ and let $\tilde{T}$ be the set of vectors $v$ in $T$ with $v_{i,1}v_{i,2}=0$ for $1\leq i\leq k+1$.  Let $S=\{v~|~v_{i,1}=0~{\rm for}~1\leq i\leq k+1\}$. Let $t\in T$ be the vector with $t_{i,1}=0,t_{i,2}=1$ for $1\leq i\leq k+1$. Let $\tau:\{0,1\}^{k+1}\rightarrow S$ be the map such that $\tau(v)_{i,1}=0$ and $\tau(v)_{i,2}=1-v_i$. In particular $\tau(0)=t$.
Let $Q_i$ be the probability space of cubes $c\in\hom_{t\mapsto x_i}(T,\bA)$ conditioned on the event that $c(\tau(v))\in U_i(v)$ holds for every $v\in K_{k+1}$. The fact that $[F_i](x_i)\neq 0$ guarantees that we condition on a positive probability event. We will use $Q_i$ to define a coupling $\Upsilon_i$ on copies of $\mathcal{B}$ indexed by $T\setminus S$. 

We will see that the system $\{\psi_v\}_{v\in T\setminus S}$ restricted to $Q_i$ is a coupling of copies of $(\bA,\mathcal{A},\mu)$ and $\Upsilon_i$ will be defined as the factor coupling according to $\mathcal{B}\subset\mathcal{A}$. 
In order to show this we establish $\Upsilon_i$ as a convex combination of couplings.
Let $Q_i'$ be the probability space of cubes $c\in\hom(S,\bA)$ with $c(t)=x_i$ conditioned on the event that $c(\tau(v))\in U_i(v)$. Let $r:Q_i\rightarrow Q_i'$ be the restriction map to $S$. It is clear that $r$ is measure preserving and the preimage $r^{-1}(c)$ for $c\in Q_i'$ is the set $\hom_c(T,\bA)$. Since the restriction of $\{\psi_v\}_{v\in T\setminus S}$ to each set $\hom_c(T,\bA)$ is a coupling of copies of $(\bA,\mathcal{A},\mu)$ the restriction of $\{\psi_v\}_{v\in T\setminus S}$ to $Q_i$ is the convex combination of these couplings with distribution given by $Q_i'$.

The final step is to obtain $\Upsilon$ as the limit of some convergent subsequence from $\{\Upsilon_i\}_{i=1}^\infty$ in the coupling topology. 


\medskip

We will need the following notation. For every $w\in\{0,1\}^{k+1}$ let $\phi_w:\{0,1\}^{k+1}\rightarrow \tilde{T}$ such that $\phi_w(v)_{i,1}=v_i$ and $\phi_w(v)_{i,2}=(1-w_i)(1-v_i)$ holds for $i\in [k+1]$. It is clear that for every fixed $w$ the map $\phi_w$ is a cube morphism. For $1\leq i\leq k+1$ we denote by $T_i\subset T$ the set of vectors $v\in T$ whose coordinate sum $h(v)$ is at most $i$ and let $\tilde{T}_i=T_i\cap\tilde{T}$. We can think of $T_i$ as the $i$-dimensional frame of $T$.

\medskip

\noindent{\bf Claim 1.}~~{\it The coupling $\Upsilon$ has the following two properties.
\begin{enumerate}
\item The sub coupling of $\Upsilon$ induced by  $\phi_w:K_{k+1}\rightarrow T\setminus S$ is equal to $\mathcal{B}\circ\Psi_{\varrho(w)}^{k+1}$ for every $w\in K_{k+1}$ and is equal to $\mathcal{B}\circ\Psi_x^{k+1}$ if $w=0$.
\item Let $\Upsilon^i$ denote the sub-coupling of $\Upsilon$ induced by $T_i\setminus S\rightarrow T\setminus S$. Then $\Upsilon_i$ is independent over its $\mathcal{B}_{i-1}$ factor for $1\leq i\leq k+1$.
\end{enumerate}
}

\medskip

The first property follows from lemma \ref{cupis1}, lemma \ref{cupis} and the way $\Upsilon_i$ is obtained as a convex combination of couplings. 

To see the second property this let $\Upsilon^i_j$ denote the subcoupling of $\Upsilon_j$ induced by the map $T_i\setminus S\rightarrow T\setminus S$. It is obviously enough to show that $\Upsilon^i_j$ is independent over its $\mathcal{B}_{i-1}$ factor.
Let $F=\{f_v\}_{v\in T_i\setminus S}$ be a bounded $\mathcal{B}$ measurable function system on $\bA$ and assume that 
$\mathbb{E}(f_{v'}|\mathcal{B}_{i-1})=0$ for some $v'\in T_i\setminus S$. By lemma \ref{weaknilprop} we obtain that $\mathbb{E}(f_{v'}|\mathcal{F}_{i-1})=0$ and so $\|f_{v'}\|_{U_i}=0$. We have to show that $\xi(\Upsilon^i_j,F)=0$.
Let $G=\{g_v\}_{v\in T\setminus\{t\}}$ be the extended function system defined as follows. If $v\in T_i\setminus S$ then $g_v=f_v$, if $v\in T\setminus (T_i\cup S)$ then $g_v=1$ and if $v\in {S\setminus\{t\}}$ then $g_v$ is the characteristic function of $U_i(\tau^{-1}(v))$. By definition we have that $\xi(\Upsilon^i_j,F)=[G^\con](x_i)$ where the convolution is taken at $t$. Using that $S\cup T_i$ is simplicial and that $d(v')\leq i$ we get by lemma \ref{simpzero} that $[G^\con](x_i)=0$.

\medskip

\noindent{\bf Claim 2.}~~{\it There is at most one self coupling $\Theta$ of $\mathcal{B}$ with index set $\tilde{T}\setminus S$ with the following two properties.
\begin{enumerate}
\item The sub coupling of $\Theta$ induced by  $\phi_w:K_{k+1}\rightarrow \tilde{T}\setminus S$ is equal to $\mathcal{B}\circ\Psi_{\varrho(w)}^{k+1}$ whenever $w\neq 0$,
\item Let $\Theta^i$ denote the sub-coupling of $\Theta$ induced by $\tilde{T}_i\setminus S\rightarrow \tilde{T}\setminus S$. $\Theta_i$ is independent over its $\mathcal{B}_{i-1}$ factor for $1\leq i\leq k+1$.
\end{enumerate}
}

Assume that there is such a coupling $\Theta=\{\theta_v\}_{v\in\tilde{T}\setminus S}$. We prove by induction on $i$ that by the conditions on the claim the coupling $\Theta_i$ is uniquely determined. For $i=1$ the coupling $\Theta_i$ is independent and so it is a unique object. Assume that the uniqueness is verified for $i-1$ and $i>1$. The second condition implies that it is enough to prove the uniqueness of the $\mathcal{B}_{i-1}$ factor of $\Theta_i$. We use a second induction to show this.

For a vector $v$ in $T$ let $h^*(v)=\sum_{i=1}^{k+1}v_{i,2}$ and let $\tilde{T}_{i,n}=\tilde{T}_{i-1}\cup\{z|h^*(z)\leq n~,~z\in\tilde{T}_i\}$.
We prove by induction on $n=h^*(v)$ that the $\mathcal{B}_{i-1}$ factor of the sub-coupling on $\tilde{T}_{i,n}$ is uniquely determined by the $\mathcal{B}_{i-1}$ factor of the sub-coupling on $\tilde{T}_{i,n-1}$. If $n=0$ then from $v\in \tilde{T}_i\setminus \tilde{T}_{i-1}$ we have that $h(v)=i$ and every vector $v'< v$ is in $\tilde{T}_{i-1}$. Since by lemma \ref{compdep} the sub-coupling of $\mathcal{B}_{i-1}$ on $\{v'|v'\leq v,v'\neq 0\}$ is completely dependent and is isomorphic to $\mathcal{B}_{i-1}\circ\Psi_{\varrho(1)}^i$ the statement is clear.  

Assume by induction that the statement holds for $n-1$ and $h^*(v)=n$.  Let $v',w\in\{0,1\}^{k+1}$ be the vectors with $v'_i=v_{i,1}+v_{i,2}$ and $w_i=1-v_{i,2}$. Let $Q=\{z|z\neq 0~,~z\leq v'\}\subset\{0,1\}^{k+1}$ and let $\alpha:Q\rightarrow\tilde{T}$ be the restriction on $\phi_w$ to $Q$. It is clear that every element in the image of $\alpha$ which is not equal to $v$ is either in $\tilde{T}_{i-1}$ or has $h^*$ value at most $n-1$. By the first property of $\Theta$ we have that the restriction of $\Theta$ to the image of $\alpha$ is $\Psi_{\varrho(w)}^i$. Using lemma \ref{compdep} and our induction step we obtain the induction statement for $v$.

\medskip

Let $\tilde{\Upsilon}$ be the sub-coupling of $\Upsilon$ on $\tilde{T}$. We get from our two claims that $\tilde{\Upsilon}$ is uniquely determined by the function $\gamma\circ\varrho$.  Since $\mathcal{B}\circ\Psi_x^{k+1}$ is  sub-coupling of $\tilde{\Upsilon}$ we get that $\gamma(x)$ is uniquely determined and so the proof of the lemma is complete.

\medskip

\subsection{Support of measure}

In this chapter we prove statements related to definition \ref{defposcube}.
Notice that $c\in C^n_x(\bA)$ if and only for every system of open sets $\{U(v)\}_{v\in K_n}$ with $c(v)\in U(v)$ we have that if $F=\{1_{U(v)}\}_{v\in K_n}$ then $[F](x)>0$. General properties of supports of measures on compact spaces imply the next lemma. 

\begin{lemma}\label{ascubepos} Let $x\in\bA~,~n\in\mathbb{N}$ be arbitrary. Then almost every $c\in C^n(\bA)$ ($c\in C^n_x(\bA)$) is positive. Furthermore positive cubes in $C^n(\bA)$ (resp. $C_x^n(\bA)$) form a closed set in the topology generted by $\gamma$.
\end{lemma}

\begin{lemma}\label{opcontconv} Let $x\in\bA$. Then for every $\gamma$-open set $U$ containing $x$ there is a system of $\gamma$-open sets $\{U(v)\}_{v\in K_{k+1}}$ such that if $F=\{1_{U(v)}\}_{v\in K_{k+1}}$ then $x\in{\rm supp}([F])\subseteq U$.
\end{lemma}

\begin{proof} Let $c\in C^{k+1}_x$ be positive. For every $i\in\mathbb{N}$ let $\{U_i(v)\}_{v\in K_{k+1}}$ be a system of open sets such that $\{U_i(v)\}_{i=1}^\infty$ is a descending neighborhood basis for $c(v)$. Let $F_i=\{1_{U_i(v)}\}_{v\in K_{k+1}}$. Assume by contradiction that $(\bA\setminus U)\cap{\rm supp}([F_i])\neq\emptyset$. Then since $B_i=(\bA\setminus U)\cap\cl({\rm supp}([F_i]))$ is a descending chain of $\gamma$-closed sets we have that there is an element $y\in\cap_{i=1}^\infty B_i$. Lemma \ref{uclos} implies that $\gamma(x)=\gamma(y)$ which contradicts the fact that $x$ and $y$ are separated by the $\gamma$-open set $U$. 
\end{proof}

\begin{definition} Let $x\in\bA$. We say that $f:\bA\rightarrow\mathbb{R}$ is a convolution neighborhood of $x$ if $f(x)\neq 0$ and $f=[F]$ where $F=\{f_v\}_{v\in K_{k+1}}$ is a function system consisting of $0-1$ valued $\mathcal{B}$-measurable functions. 
\end{definition}

Note that the values of a convolution neighborhood $f$ of $x$ are non-negative and by lemma \ref{opcontconv} for every open neighborhood $U$ of $x$ there is a convolution neighborhood $f$ of $x$ such that ${\rm supp}(f)\subseteq U$.

\begin{lemma}\label{convnproj} Let $f$ be a convolution neighborhood of $x\in\bA$. Then there is a $\gamma_{k-1}$ continuous non-negative function $g$ such that $g=\mathbb{E}(f|\mathcal{F}_{k-1})$ (almost everywhere) and $g(x)>0$.
\end{lemma}

\begin{proof} Assume that $f=[F]$ with a $0-1$ valued $\mathcal{B}$ measurable function system $F=\{f_v\}_{v\in K_{k+1}}$. Let $G=\{g_v\}_{v\in K_{k+1}}$ where $g_v=\mathbb{E}(f_v|\mathcal{F}_{k-1})$ and $g=[G]$. Note that by lemma \ref{weaknilprop} each $g_v$ is measurable in $\mathcal{B}_{k-1}$. Then by lemma \ref{vetites1} we have that $g=\mathbb{E}(f|\mathcal{F}_{k-1})$. Furthermore, since $g_v$ is almost surely positive on $f^{-1}_v(1)$ we have that $[G](x)>0$. Finally, using our induction hypothesis that $\gamma_{k-1}$ is a strong nilspace factor we get that $g=[G']\circ\gamma$ where $G'=\{g'_v\}_{v\in K_{k+1}}$ is a function system on $N_{k-1}$ such that $g_v=g'_v\circ\gamma_{k-1}$. We obtain that $g$ is $\gamma_{k-1}$-continuous. 
\end{proof}

\begin{lemma}\label{openpos} Every non empty $\gamma$-open set has positive measure.
\end{lemma}

\begin{proof} We prove the statement by induction on $k$. For $k=0$ it is clear and we assume that it is true for $\gamma_{k-1}$. Let $U$ be a non-empty $\gamma$-open set, $x\in U$ and $f$ be a convolution neighborhood of $x$ such that ${\rm supp}(f)\subseteq U$. It is enough to show that $\mathbb{E}(f)>0$. On the other hand the function $g$ satisfying the conditions of lemma \ref{convnproj} has the property that $\mathbb{E}(g)=\mathbb{E}(f)$. Furthermore since $g$ is not identically $0$ and non-negative we have by our induction hypothesis that $\mathbb{E}(g)>0$.
\end{proof}

\begin{lemma}\label{pospreshef} Let $c\in C^n(\bA)$ be a positive cube and let $\phi:\{0,1\}^m\rightarrow\{0,1\}^n$ be an injective cube morphism. Then $c_2=c\circ\phi\in C^m(\bA)$ is also a positive cube. 
\end{lemma}

\begin{proof} Let $\{U(v)\}_{v\in\{0,1\}^m}$ be a system of $\gamma$-open sets with $c_2(v)\in U(v)$. Let $\{W(v)\}_{v\in\{0,1\}^n}$ be defined such that $W(v)=U(\phi^{-1}(v))$ if $v\in\im(\phi)$ and $W(v)=\bA$ otherwise. It is clear by lemma \ref{cupis1} and lemma \ref{cupis} that $\cap_{v\in\{0,1\}^m}\psi_v^{-1}(U(v))$ has the same measure as $\cap_{v\in\{0,1\}^n}\psi_v^{-1}(W(v))$ and thus by the positivity of $c$ we obtain the positivity of $c_2$.
\end{proof}

\subsection{Every cube is positive}

The main result of this chapter is the following.

\begin{proposition}\label{cubepos} Every cube in $\bA$ is positive with respect to $\gamma$.
\end{proposition}

In this chapter we assume by induction on $k$ that proposition \ref{cubepos} is true for $\gamma_{k-1}$.
We will need the following lemmas.

\begin{lemma}\label{uniqclos1} Let $c_1,c_2\in C^{k+1}(\bA)$ be two positive cubes such that $\gamma\circ c_1$ agrees with $\gamma\circ c_2$ on $K_{k+1}$. Then $\gamma\circ c_1=\gamma\circ c_2$.
\end{lemma}

\begin{proof} Assume by contradiction that $\gamma(c_1(0))\neq \gamma(c_2(0))$. Let $U$ be a $\gamma$-open set containing $c_1(0)$ such that $c_2(0)\notin\cl(U)$. Let furthermore $\{U(v)\}_{v\in K_{k+1}}$ be a system of open sets with $c_1(v)\in U(v)$ for every $v\in K_{k+1}$ such that ${\rm supp}([F])\subseteq U$ holds for the function system $F=\{1_{U(v)}\}_{v\in K_{k+1}}$. The existence of such a system of open sets is guaranteed by lemma \ref{opcontconv}. Notice that by the condition of the lemma, $c_2(v)\in U(v)$ holds for every $v\in K_{k+1}$. Let us define $U(0)$ as the complement of $\cl(U)$. Then we have that the measure of $\cap_{v\in\{0,1\}^{k+1}}\psi_v^{-1}(U(v))$ is equal to the integral of $[F]$ on $U(0)$ and thus it is $0$. This contradicts the positivity of $c_2$.
\end{proof}

\begin{lemma}\label{simpglpos} let $S\subset\{0,1\}^n$ be a simplicial set of hight at most $k$. Let $q:S\rightarrow \bA$ be a map such that the restriction of $q$ to any maximal face composed with $\gamma_{k-1}$ is a cube in $N_{k-1}$. Then there is a positive cube $c\in C^n(\bA)$ such that $\gamma\circ c$ restricted to $S$ is equal to $\gamma\circ q$. 
\end{lemma}

\begin{proof} Using the fact that positive cubes are forming a closed set, it is enough to show the following. Let $\{U(v)\}_{v\in S}$ be an arbitrary system of $\gamma$-open sets such that $q(v)\in U(v)$ holds for every $v\in S$. Then there is a positive cube $c_2\in C^n(\bA)$ such that $c_2(v)\in U(v)$ for every $v\in S$.
 To see this we choose a function system $F=\{f_v\}_{v\in S}$ such that $f_v$ is a convolution neighborhood of $q(v)$ and ${\rm supp}(f_v)\in U(v)$ for every $v\in S$. Let $F'=\{f_v\}_{v\in\{0,1\}^n}$ be the function system where for $v\in \{0,1\}^n\setminus S$ the function $f_v$ is identically $1$. Let $G=\{g_v\}_{v\in\{0,1\}^n}$ be the function such that $g_v$ is the continuous projection of $f_v$ guaranteed by lemma \ref{convnproj}. Note that if $v\in\{0,1\}^n\setminus S$ then $g_v$ is also constant $1$.
Let $c_3:\{0,1\}^n\rightarrow\bA$ be a cube in $C^n(\bA)$ such that $\gamma_{k-1}\circ c_3=\gamma_{k-1}\circ q$ on $S$. The existence of $c_3$ follows from the assumption that $\gamma_{k-1}$ defines a nilspace factor: It is an easy consequence of the nilspace axioms (see \cite{NP}) that morphisms of simplicial sets in $\{0,1\}^n$ into nilspaces can always be extended to the full cube.  
We have that $c_3(v)\in{\rm supp}(g_v)$ holds for every $v\in\{0,1\}^m$. Since $c_3$ is positive in $\gamma_{k-1}$ we have that $(G)>0$ holds. On the other hand by corollary \ref{simpzerocor} we have that $(F')=(G)$ and thus the measure of $T=\cap_{v\in S}\psi_v^{-1}(U(v))$ on $C^n(\bA)$ is positive. By lemma \ref{ascubepos} we have that $T$ contains a positive cube. 
\end{proof}

\medskip

We are ready to prove proposition \ref{cubepos}. Let $c\in C^{k+1}(\bA)$ be an arbitrary cube.
Let $T,S$ be defined as in chapter \ref{conopen}. Let $\phi:\{0,1\}^{k+1}\rightarrow S$ be defined by $\phi(v)_{i,1}=0$ and $\phi(v)_{i,2}=v_i$ and for $w\in\{0,1\}^{k+1}$ let $\phi_w:\{0,1\}^{k+1}\rightarrow T$ be defined by $\phi_w(v)_{i,1}=v_i$ and $\phi_w(v)_{i,2}=w_i(1-v_i)$. Let $\hat{\phi}_w:\Hom_{c\circ\phi^{-1}}(T,A)\rightarrow C_{c(w)}^{k+1}(\bA)$ be the map given by composing elements from $\Hom_{c\circ\phi^{-1}}(T,A)$ by $\phi_w$.  By lemma \ref{cupis1} we obtain that $\hat{\phi}_w$ is measure preserving. This fact combined with lemma \ref{ascubepos} implies that for almost every element in $c_2\in\Hom_{c\circ\phi^{-1}}(T,A)$ we have that for every $w\in\{0,1\}^{k+1}$ the cube $\hat{\phi}_w(c_2)$ is positive in $\Psi_{c(w)}^{k+1}$. Let $c_2$ be a fixed cube with this property. 
Let $T_2\subset T$ be the subset of elements $v$ such that  $v_{i,1}v_{i,2}=0$ holds for every $1\leq i\leq k+1$ and let $T_3\subset T_2$ be the set of vectors in $T_2$ in which the coordinate sum is at most $k$. Both $T_2$ and $T_3$ are simplicial sets. 
We define the bijection $\tau:T_2\rightarrow T_2$ by $\tau(v)_{i,1}=1-v_{i,1}-v_{i,2}$ and $\tau(v)_{i,2}=v_{i,2}$. Let $c_3=c_2\circ\tau$. The maximal faces of $T_2$ are the images of the maps $\tau^{-1}\circ\psi_w$. It follows that $c_3$ restricted to every maximal face is a positive cube.
The restriction of $c_3$ to $T_3$ satisfies the conditions of lemma \ref{simpglpos} and thus there is a positive cube $c_4\in\Hom(T,\bA)$ such that the restriction of $c_4$ to $T_3$ is equal to the restriction of $c_3$ to $T_3$.
The restriction of $c_4$ to every maximal face of $T_2$ is positive and thus by lemma \ref{uniqclos1} we have that $c_4$ restricted to $T_2$ is equal to $c_3$. Then $c=c_3\circ\phi_1$ is positive by lemma \ref{pospreshef}

\medskip 

\subsection{Verifying the nilspace structure}

In this chapter we finish the proof of proposition \ref{topologization}.
It is clear that $N$ satisfies the first two axioms. We focus on the last axiom. Let $K=\{0,1\}^n\setminus\{1^n\}$. Let $q:K\rightarrow\bA$ be a map such that the restriction of $q$ to every maximal face of $K$ composed with $\gamma$ is a cube in $N$.
In order to verify the third nilpsace axiom we have to show that there is a cube $c\in C^n(\bA)$ such that $\gamma\circ c=\gamma\circ q$ on $K$.  
Let $S$ be the set of vectors in $\{0,1\}^n$ of hight at most $k$.  It is clear that the restriction of $q$ to $S$ satisfies the conditions of lemma \ref{simpglpos}. Using the lemma we get that there is a positive cube $c\in C^n(\bA)$ such that $\gamma\circ c$ is equal to $\gamma\circ q$ on the set $S$. 
We claim that $\gamma\circ c$ is equal to $\gamma\circ q$ on $K$.
Let $F$ be an arbitrary maximal face in $K$. The condition on $q$ guarantees that there is a cube $c_2:F\rightarrow\bA$ such that $\gamma\circ c_2=\gamma\circ q$ on $F$. In particular $\gamma\circ c_2=\gamma\circ c$ holds on $F\cap S$.
The claim is equivalent with $\gamma\circ c_2=\gamma\circ c$ on $F$. 
We prove this by contradiction. Let $v\in F$ be an element of minimal hight for which $\gamma\circ c_2\neq\gamma\circ c$. Then there is a face $F'\subseteq F$ of dimension $k+1$ whose maximal element is $v$. Since $\gamma\circ c_2=\gamma\circ c$ holds on $F'\setminus\{v\}$ lemma \ref{uniqclos1} together with proposition \ref{cubepos} shows the contradiction. 

It is clear from lemma \ref{weaknilprop} that $\gamma$ is factor consistent. By theorem \ref{charpres} we obtain that $\gamma$ is a strong nilspace factor.

\medskip

\section{Higher order dual groups}\label{chap:higherdual}

The goal of this part of the paper is to analyze the structure of $\hat{\bA}_k$. We start with the prof of theorem \ref{hofdecomp}.

\medskip

\noindent{\it Proof of theorem \ref{hofdecomp}}.~~Let $f:\bA\rightarrow\mathbb{C}$ be an arbitrary bounded function measurable in $\mathcal{F}_k$. Then according to theorem \ref{main} there is a strong $k$-step nilspace factor $\gamma:\bA\rightarrow N$ and a Borel measurable function $h:N\rightarrow\mathbb{C}$ such that $f$ is equal to $h\circ\gamma$. Using lemma \ref{nilfourdec} and lemma \ref{nilspchar} we can decompose $h$ as $\sum_{\chi\in\hat{A_k}}g_\chi h_\chi$ (converging in $L^2$) where $h_\chi\in W(\chi,N)$~,~$|h_\chi|=1$ and $g_\chi$ is measurable in the $k-1$ step factor of $N$. We obtain that $f=\sum_{\chi\in\hat{A_k}}(g_\chi\circ\gamma)( h_\chi\circ\gamma)$. The terms $g_\chi\circ\gamma$ are measurable in $\mathcal{F}_{k-1}$ and the terms $h_\chi\circ\gamma$ are $k$-th order characters by lemma \ref{homok}. It follows that $L^2(\mathcal{F}_k)$ is spanned by the modules in $\hat{A}_k$ and thus by lemma \ref{charort} the proof is complete.

\medskip

\begin{lemma}\label{highconv} Let $f,g$ be two functions in $L^\infty(\mathcal{F}_k)$. For $a\in\hat{\bA}_k$ let us denote the component of $f$ and $g$ in $a$ by $f_a$ and $g_a$. Then the component of $fg$ in $c\in\hat{\bA}_k$ is equal to
$$\sum_{ab=c}f_ag_b$$
where the above sum has only countable many non zero term and the sum is convergent in $L^2$.
\end{lemma}

\begin{proof} First of all we observe that if $g$ is contained in a single rank one module then
the $k$-th order decomposition of $fg$ is $\sum_{a\in\hat{\bA}_k} f_ag$ since it converges in $L^2$ and the terms $f_ag$ are from distinct rank one modules.
From this observation we also get the statement if $g$ has finitely many non zero components.
If $g$ has infinitely many components then for an arbitrary $\epsilon$ we can approximate $g$ with precision $\epsilon$ in $L^2$ by a sub-sum $g_\epsilon$ of its components . Then $fg=fg_\epsilon+f(g-g_\epsilon)$. Here the $\|f(g-g_\epsilon)\|_2\leq \|f\|_\infty\epsilon$. So as $\epsilon$ goes to $0$ the $L^2$ error we make also goes to $0$.
\end{proof}

\begin{definition} Let $f$ be a function in $L^2(\mathcal{F}_k)$. We say that the $k$-th dual-support $S_k(f)\subseteq\hat{\bA}_k$ of $f$ is the set of rank one modules that are not orthogonal to $f$.
It is clear that $S_k(f)$ is a countable set. 
\end{definition}

The next two statement follows from lemma \ref{highconv}.

\begin{lemma}\label{prodsup} If $f,g\in L^\infty(\mathcal{F}_k)$ then $S_k(fg)\subseteq S_k(f)S_k(g)$ and $S_k(f+g)\subseteq S_k(f)\cup S_k(g)$. Furthermore if $t\in\bA$ is fixed then the function $f_t(x)=f(x+t)$ satisfies $S_k(f_t)=S_k(f)$.
\end{lemma}

For a subgroup $T\leq\hat{\bA}_k$ we denote by $\mathfrak{H}_k(T)$ the collection of sets $U$ that are measurable in $\mathcal{F}_k$ and $S(1_U)\subset T$. It follows from lemma \ref{prodsup} that $\mathfrak{H}_k(T)$ is a shift invariant $\sigma$-algebra such that $\mathcal{F}_{k-1}\subseteq\mathfrak{H}_k(T)\subseteq\mathcal{F}_k$. If $\mathcal{B}\subset\mathcal{F}_k$ is a separable $\sigma$-algebra and $\{U_i\}_{i=1}^\infty$ is a generating system of $\mathcal{B}$ then for the countable group $T$ generated by $\{S_k(1_{U_i})\}_{i=1}^\infty$ we have that $\mathcal{B}\subseteq\mathfrak{H}_k(T)$.

\begin{lemma}\label{secder} If $\phi$ is a $k$-th order character then there is a countable subgroup $T<\hat{\bA}_{k-1}$ such that $\Delta_{t_1,t_2}\phi$ is measurable in $\mathfrak{H}_{k-1}(T)$ for every pair $t_1,t_2\in\bA$. 
\end{lemma}

\begin{proof} Lemma \ref{charsep} and lemma \ref{sepsiggen} implies that there is a separable $\sigma$-algebra $\mathcal{B}\subset\mathcal{F}_{k-1}$ such that $\phi\in [\mathcal{B},k]^*$. Let $T\subset\hat{\bA}_{k-1}$ be a countable subgroup such that $\mathcal{B}\subseteq\mathfrak{H}_{k-1}(T)$. 
Then $\phi\in[\mathfrak{H}_{k-1}(T),k]^*$. Lemma \ref{pureprop} implies that $\Delta_{t_1,t_2}\phi\in[\mathfrak{H}_{k-1}(T),k-2]^*$ for every $t_1,t_2\in\bA$.
Using lemma \ref{charmes} we obtain that $\Delta_{t_1,t_2}\phi$ is measurable in $\mathfrak{H}_{k-1}(T)\vee\mathcal{F}_{k-2}=\mathfrak{H}_{k-1}(T)$.
\end{proof}

\begin{lemma}\label{nozer} Let $f,g$ be $L^\infty(\mathcal{A})$ functions such that non of $\mathbb{E}(f|\mathcal{F}_{k-1})$ and $\mathbb{E}(g|\mathcal{F}_{k-1})$ is the $0$ function.  Then  $\mathbb{E}_t\Bigl(\|f(x)\overline{g(x+t)}\|_{U_k}^{2^k}\Bigr)>0$.
\end{lemma}

\begin{proof} The support of both $f_1=\mathbb{E}(f|\mathcal{F}_{k-1})$ and $g_1=\mathbb{E}(g|\mathcal{F}_{k-1})$ has positive measure. This means that for a positive measure set of $t$'s the supports of $f_1$ and $g_1(x+t)$ intersect each other in a positive measure set. (By Fubini's theorem, the expected value of the measure of the intersection is the product of the measures of the supports.) Since $U_k$ is a norm on $L^\infty(\mathcal{F}_{k-1})$ we get that in (\ref{twofun}) the right hand side is not $0$. By lemma \ref{twofunct} the proof is complete.
\end{proof}

\begin{proposition}\label{fixtype} Let $\phi$ be a $k$-th order character. Then there is a countable subgroup $T\leq\hat{\bA}_{k-1}$ such that $S_{k-1}(\Delta_t\phi)S_{k-1}^{-1}(\Delta_t\phi)\subseteq T$ for every fixed $t\in \bA$.
\end{proposition}

\begin{proof}  Let $T$ be the subgroup of $\hat{\bA}_{k-1}$ guaranteed by lemma \ref{secder}. We have that $S_{k-1}(\Delta_{t_1,t_2}\phi)\subseteq T$ for every $t_1,t_2$. Let $t_1\in\bA$ be an arbitrary fixed element and let $\Delta_{t_1}\phi=f_1+f_2+\dots$ be the unique $k-1$-th order Fourier decomposition of $\Delta_{t_1}\phi$ into non zero functions. Assume that $\lambda_i\in\hat{\bA}_{k-1}$ is the module containing $f_i$ for every $i$. We have to show that $\lambda_i\lambda_j^{-1}\in T$ for every pair of indices $i,j$.
Let us choose a $k-1$-th order character $\phi_i$ from every module $\lambda_i$ and let $g_i$ denote $(\Delta_{t_1}\phi)\overline{\phi_i}$.
We have that $\mathbb{E}(g_i|\mathcal{F}_{k-2})$ is not $0$.
This means by lemma \ref{nozer} and theorem \ref{propfk} that for a positive measure of $t_2$'s $\mathbb{E}(g_i(x)\overline{g_j(x+t_2)}|\mathcal{F}_{k-2})$ is not the $0$ function. On the other hand $g_i(x)\overline{g_j(x+t_2)}=(\Delta_{t_1,t_2}\phi(x))\overline{\phi_i(x)}\phi_j(x+t_2)$.
Here $\overline{\phi_i(x)}\phi_j(x+t_2)$ is an element from the module $\lambda_j\lambda_i^{-1}$.
If $\mathbb{E}(g_i(x)\overline{g_j(x+t_2)}|\mathcal{F}_{k-2})$ is not $0$ for some $t_2$ then the $\lambda_i\lambda_j^{-1}$ component of $\Delta_{t_1,t_2}\phi$ is not zero. It shows that $\lambda_i\lambda_j^{-1}\in T$.
\end{proof}

\begin{lemma}\label{charhom} For every $k$-th order character $\phi$ there is a countable subgroup $T\subset\hat{\bA}_{k-1}$ and a homomorphism $h:\bA\rightarrow\hat{\bA}_{k-1}/T$ such that $S_{k-1}(\Delta_t\phi)$ is contained in the coset $h(t)$.
\end{lemma}

\begin{proof} Proposition \ref{fixtype} implies that there is a countable subgroup $T\subset\hat{\bA}_{k-1}$ such that $S_{k-1}(\Delta_t\phi)$ is contained in a coset of $T$ for every element $t\in\bA$. We denote this coset by $h(t)$. We have to show that $h$ is a homomorphism. This follwos from 
$\Delta_{t_1+t_2}\phi(x)=\Delta_{t_2}\phi(x+t_1)\Delta_{t_1}\phi(x)$
together with lemma \ref{prodsup}.
\end{proof}

\begin{lemma}\label{countriv} Let $k\geq 2$ and $\phi$ be a $k$-th order character such that there is a countable subgroup $T\subseteq \hat{\bA}_{k-1}$ with the property that $\Delta_t\phi$ is measurable in $\mathfrak{H}_{k-1}(T)$ for every $t\in\bA$. Then $\phi$ is measurable in $\mathcal{F}_{k-1}$ (or in other words $\phi$ represents the trivial module).
\end{lemma}

\begin{proof} For $a\in T$ let $Q_a=\{t|t\in\bA,a\in S_{k-1}(\Delta_t\phi)\}$. First we claim that $Q_a$ is measurable. Indeed, if $\phi'$ is a fixed character representing $a$ then $Q_a$ is the set of $t$'s for which the measurable function $t\mapsto\|\overline{\phi'}\Delta_t\phi\|_{U_{k-1}}$ is not zero. Using $\sigma$-additivity and $\cup_{a\in T}Q_a=\bA$ we obtain that there is a fixed $a\in T$ such that $Q_a$ has positive measure.  Assume that $\phi'$ represents $a$ and let $f(x)=\overline{\phi'(x)}\phi(x)$. Let $H=\{h_v\}_{v\in\{0,1\}^k}$ be the function system with $h_{(w,0)}=f$ and $h_{(w,1)}=\phi$ for $w\in\{0,1\}^{k-1}$. Then we have by (\ref{dimred}) that  $$(H)=\mathbb{E}_t(\|f(x)\overline{\phi(x+t)}\|^{2^{k-1}}_{U_{k-1}})=\mathbb{E}_t(\|\overline{\phi'}\Delta_t\phi\|^{2^{k-1}}_{U_{k-1}})>0.$$ By (\ref{GCS}) we obtain that $\|\phi\|_{U_k}>0$ and so lemma \ref{erosort} implies that $\phi$ represent the trivial module.
\end{proof}

\subsection{Various Hom-sets}

In this chapter we use multiplicative notation for Abelian groups. For two Abelian groups $A_1$ and $A_2$ we denote by $\hom(A_1,A_2)$ the set of all homomorphism from $A_1$ to $A_2$. The set $\hom(A_1,A_2)$ is an Abelian group with respect to the point wise multiplication. Let $\aleph_0(A_2)$ denote the set of countable subgroups in $A_2$. The groups $\hom(A_1,A_2/T)$ where $T\in\aleph_0(A_2)$ are forming a direct system with the natural homomorphisms $\hom(A_1,A_2/T_1)\rightarrow\hom(A_1,A_2/T_2)$ defined when $T_1\subseteq T_2$.

\begin{definition} $\hom^*(A_1,A_2)$ is the direct limit of the direct system $$\{\hom(A_1,A_2/T)\}_{T\in\aleph_0(A_2)}$$
with the homomorphisms induced by embeddings on $\aleph_0(A_2)$.
\end{definition}

We describe the elements of $\hom^*(A_1,A_2)$.
Let $H(A_1,A_2)$ be the disjoint union of all the sets $\hom(A_1,A_2/T)$ where $T$ runs through the countable subgroups of $A_2$. If $h_1\in\hom(A_1,A_2/T_1)$ and $h_2\in\hom(A_1,A_2/T_2)$ are two elements in $H(A_1,A_2)$ then we say that $h_1$ and $h_2$ are equivalent if there is a countable subgroup $T_3$ of $A_2$ containing both $T_1$ and $T_2$ such that $h_1$ composed with $A_2/T_1\rightarrow A_2/T_3$ is the same as $h_2$ composed with $A_2/T_2\rightarrow A_2/T_3$.
The equivalence classes in $H(A_1,A_2)$ are forming an Abelian group that is the same as $\hom^*(A_1,A_2)$.

For two abelian groups let $\hom^c(A_1,A_2)$ denote the set of homomorphisms whose image is countable. We denote by $\hom^0(A_1,A_2)$ the factor $\hom(A_1,A_2)/\hom^c(A_1,A_2)$. If $T$ is a countable subgroup of $A_2$ then there is a natural embedding of $\hom^0(A_1,A_2/T)$ into $\hom^*(A_1,A_2)$ in the following way.
The set $\hom(A_1,A_2/T)$ is a subset of $H(A_1,A_2)$. It is easy to see that $\phi_1,\phi_2\in\hom(A_1,A_2/T)$ are equivalent if and only if they are contained in the same coset of $\hom^c(A_1,A_2)$. From the definitions it follows that

\begin{equation}\label{homstarcup} \hom^*(A_1,A_2)=\bigcup_{T\in\aleph_0(A_2)} \hom^0(A_1,A_2/T).
\end{equation}

From (\ref{homstarcup}) we obtain the next lemma.

\begin{lemma}\label{exp} If $A_1$ is of exponent $n$ then so is $\hom^*(A_1,A_2)$.
\end{lemma}

Note that an abelian group is said to be of exponent $n$ if the $n$-th power of every element is the identity.
It is easy to see that If $p$ is a prime number and $A_2$ is of exponent $p$ then $\hom^*(A_1,A_2)=\hom^0(A_1,A_2)$.

\begin{definition} We say that an Abelian group is essentially torsion free if there are at most countably many finite order elements in it.
\end{definition}

\begin{lemma}\label{estors} If $A$ is essentially torsion free then $A/T$ is essentially torsion free whenever $T\in\aleph_0(A)$.
\end{lemma}

\begin{proof} Assume by contradiction that there are uncountably many finite order elements in $A/T$. Then there is a natural number $n$ and element $t\in T$ such that the set
$S=\{x~|~x\in A,~x^n=t\}$ is uncountable. Then for a fixed element $y\in S$ the set $Sy^{-1}$ is an uncountable set of finite order elements in $A$ which is a contradiction.
\end{proof}

\begin{lemma}\label{nullfree} If $A_2$ is essentially torsion free then $\hom^0(A_1,A_2)$ is torsion free.
\end{lemma}

\begin{proof} Assume by contradiction that there is an element $\tau\in\hom(A_1,A_2)$ and $n\in\mathbb{N}$ such that $\tau(A_1)$ is uncountable but $\tau^n(A_1)$ is countable.
Similarly to the proof of lemma \ref{estors} this means that there is a fixed element $t\in\tau^n(A_1)$ whose pre image under the map $x\mapsto x^n$ is uncountable which is a contradiction.
\end{proof}

\begin{lemma}\label{estfree} If $A_2$ is essentially torsion free then $\hom^*(A_1,A_2)$ is torsion free.
\end{lemma}

\begin{proof} By (\ref{homstarcup}) it is enough to prove that for every $T\in\aleph_0(A_2)$ the group $\hom^0(A_1,A_2/T)$ is torsion free. Lemma \ref{estors} implies that $A_2/T$ is essentially torsion free. Lemma \ref{nullfree} finishes the proof.
\end{proof}

\subsection{On the structure of the higher order dual groups}

We return to the structure of $\hat{\bA}_k$.
Let us start with $\hat{\bA}_1$.

\begin{lemma}\label{firstdual} The group $\hat{\bA}_1$ is isomorphic to $\prod_\omega\hat{A_i}$.
\end{lemma}

\begin{proof}
We have that $\bA$ is the ultra product of a sequence $\{A_i\}_{i=1}^\infty$ of compact abelian groups.
Let $H\subseteq\hat{\bA}_1$ denote the set of those characters that are ultra limits of characters on $\{A_i\}_{i=1}^\infty$.
Let $\lambda_i:A_i\rightarrow\mathbb{C}$ and $\chi_i:A_i\rightarrow\mathbb{C}$ be two sequences of linear characters. Let furthermore $\lambda$ be the ultra limit of $\{\lambda_i\}_{i=1}^\infty$ and $\chi$ be the ultra limit of $\{\chi_i\}_{i=1}^\infty$.
If $\lambda_i$ differs from $\chi_i$ on an index set which is in the ultra filter then they are orthogonal at this index set and so they are orthogonal in the limit. This implies that $\lambda=\chi$ if and only if the sequences $\{\lambda_i\}_{i=1}^\infty$ and $\{\mu_i\}_{i=1}^\infty$ agree on a set from the ultra filter. In other words $H$ is isomorphic to the ultra product of the dual groups of $\prod_\omega\hat{A_i}$. 
We show that $H=\hat{\bA}_1$. Assume by contradiction that $H$ is strictly smaller than $\bA_1$. Then there is a character $\phi\in\hat{\bA}_1$ which is orthogonal to every character in $H$.
We have that $\phi=\lim_\omega\phi_i$ where $\phi_i:A_i\rightarrow\mathbb{C}$ is a measurable function of absolute value $1$. For $i\in\mathbb{N}$ let $a_i$ denote the $L^\infty$ norm of the Fourier transform of $\phi_i$. 
We have that $\lim_\omega a_i=0$. Equation (\ref{u2norm}) implies that $\|\phi_i\|_{U_2}\leq \sqrt{a_i}$ and thus $\|\phi\|_{U_2}=\lim_\omega\|\phi_i\|_{U_2}=0$ which is a contradiction by lemma \ref{kisk}.
\end{proof}

\medskip

By lemma \ref{charhom} every $k$-th order character $\phi$ induces a homomorphism from $\bA$ to $\hat{\bA}_{k-1}/T$ for some countable subgroup. This homomorphism represents an element in $\hom^*(\bA,\hat{\bA}_{k-1})$. We denote this element by $q_k(\phi)$. If $q_k(\phi_1)=q_k(\phi_2)$ then Lemma \ref{countriv} shows that if $k\geq 2$ then $\phi_1$ and $\phi_2$ belong to the same rank one module.
This implies the following theorem.

\begin{theorem}\label{dualemb} If $k\geq 2$ then $q_k:\hat{\bA}_k\rightarrow\hom^*(\bA,\hat{\bA}_{k-1})$ is an injective homomorphism.
\end{theorem}

Note that if $k=1$ then $\hat{\bA}_1$ is embedded into $\hom(\bA,\hat{\bA}_0)$ where $\hat{\bA}_0$ is defined as the complex unit circle with multiplication.
The next theorem follows immediately from theorem \ref{dualemb}

\begin{theorem}[Structure of the dual groups]\label{dualstruct} $\hat{\bA}_k$ is isomorphic to a subgroup in
$$\hom^*(\bA,\hom^*(\bA,\dots,\hom^*(\bA,\hat{\bA}_1))\dots)$$
where the number of $\hom^*$-s is $k-1$,
\end{theorem}

\begin{proof} The proof follows directly from Lemma \ref{dualemb} and the fact that $\hom^*(A_1,A_2)\subseteq \hom^*(A_1,A_3)$ whenever $A_2\subseteq A_3$.
\end{proof}

Theorem \ref{dualstruct} has the next two useful consequences.

\begin{lemma}\label{duexp} Let $e$ be a natural number and assume that the groups $\{A_i\}_{i=1}^\infty$ have exponent $e$. Then $\hat{\bA}_k$ has exponent $e$ for every $k\geq 1$.
\end{lemma}

\begin{proof} We have by lemma \ref{firstdual} that $\hat{\bA}_1$ has exponent $e$. Then lemma \ref{exp} and theorem \ref{dualstruct} finish the proof.
\end{proof}

\begin{lemma}\label{charzero} Let  $\{A_i\}_{i=1}^\infty$ be a sequence of groups such that for every natural number $n>1$  there are only finitely many indices $i$ such that $\hat{A_i}$ has an element of order $n$. Then $\hat{\bA}_k$ is torsion free for every $k\geq 1$. 
\end{lemma}

\begin{proof} We have by lemma \ref{firstdual} that $\hat{\bA}_1$ is torsion free. Then lemma \ref{estfree} and theorem \ref{dualstruct} finish the proof.
\end{proof}

\section{Consequences of the main theorem}

\subsection{Regularization, inverse theorem and special families of groups} 

\bigskip

We prove theorem \ref{reglem}, theorem \ref{invthem}, theorem \ref{restreg} and theorem \ref{restinv}.

\noindent{\it Proof of theorem \ref{reglem}}~We proceed by contradiction. Let us fix $k$ and $F$. Assume that the statement fails for some $\epsilon>0$. 
This means that there is a sequence of measurable functions $\{f_i\}_{i=1}^\infty$ on the compact abelian groups $\{A_i\}_{i=1}^\infty$ with $|f_i|\leq 1$ such that $f_i$ does not satisfy the statement with $\epsilon$ and $n=i$. Let $\omega$ be a fixed non-principal ultra filter and $\bA=\prod_\omega A_i$.
We denote by $f$ the ultra limit of $\{f_i\}_{i=1}^\infty$.
By theorem \ref{main} we have that $f=f_s+f_r$ where $\|f_r\|_{U_{k+1}}=0$ and $f_s=\tilde{\gamma}\circ g$ for some strong nilspace factor $\tilde{\gamma}:\bA\rightarrow N$ and measurable function $g:N\rightarrow\mathbb{C}$.
Now we use theorem \ref{inverselimit} which says that $N$ is an inverse limit of finite dimensional nilspaces $\{N_i\}_{i=1}^\infty$ in such a way that the projections from $N$ to $N_i$ are all fibre surjective.
Let $\mathcal{N}_i$ denote the $\sigma$-algebra generated by the projection to $N_i$. Then we have that 
$g=\lim_{i\to\infty}\mathbb{E}(g|\mathcal{N}_i)$ in $L^1$.
It follows that there is an index $j$ such that
$g_j=\mathbb{E}(g|\mathcal{N}_j)$ satisfies $\|g-g_j\|_1\leq\epsilon/3$.
Let $\gamma:\bA\rightarrow N_j$ be the composition of $\tilde{\gamma}$ with the fibre surjective map $N\rightarrow N_j$. 
We have by lemma \ref{strongcirc} that $\gamma$ is a strong nilspace factor of $\bA$. 
Furthermore there is a Lipschitz function $h:N_j\rightarrow\mathbb{C}$ with $|h|\leq 1$ and Lipschitz constant $c$ such that $\|g_j-h\|_1\leq\epsilon/3$.

Using that $\gamma$ is a strong nilspace factor we have that $q=\gamma\circ h$ satisfies that $\|f_s-q\|\leq 2\epsilon/3$.
Let $f_e=f_s-q$. 
The function $\gamma$ is a continuous function so there is a sequence of continuous functions $\{\gamma_i':A_i\rightarrow N_j\}_{i=1}^\infty$ such that $\lim_{\omega}\gamma'_i=\gamma$. It is easy to see that $\gamma'_i$ is an approximate morphism with error tending to $0$.
It follows from \cite{NP} that it can be corrected to a morphism $\gamma_i$ (if $i$ is sufficiently big) such that the maximum point wise distance of $\gamma_i$ and $\gamma_i'$ goes to $0$. As a consequence we have that $\lim_\omega\gamma_i=\gamma$.

Let $f^i_s=\gamma_i\circ h$, and let $f^i_r$ be a sequence of measurable functions with $\lim_\omega f^i_r=f_r$.
We set $f^i_e=f_i-f^i_s-f^i_r$. It is clear that $\lim_\omega f^i_s=q$ and $\lim_\omega f^i_e=f_e$.
We also have that $\lim_\omega\|f^i_r\|_{U_{k+1}}=\|f_r\|_{U_{k+1}}=0$, $\lim_\omega (f^i_r,f^i_s)=(f_r,q)=0$ and $\lim_\omega(f^i_r,f^i_e)=(f_r,f_e)=0$.
Let $m$ be the maximum of the complexity of $N_i$ and $c$.
There is an index set $S$ in $\omega$ such that 
\begin{enumerate}
\item $\|f^i_r\|_{U_{k+1}}\leq F(\epsilon,m)$,
\item $\|f^i_e\|_1\leq \epsilon$,
\item $|(f^i_r,f^i_s)|~,~|(f^i_r,f^i_e)|\leq F(\epsilon,m)$,
\item $\gamma_i$ is at most $F(\epsilon,m)$ balanced,
\end{enumerate}
hold simultaneously on $S$.
Note that $\gamma$ itself is $0$ balanced.
This is a contradiction.

\bigskip

\noindent{\it Proof of theorem \ref{restreg}}~In the proof of theorem \ref{reglem} the nilspace $N_j$ that we construct is a character preserving factor of $\bA$. This means that the $i$-th structure group of $N_j$ is embedded into $\hat{\bA}_i$. This shows that $N_j$ is a $\mathfrak{A}$-nilspace. 

\bigskip

\noindent{\it Proof of theorem \ref{invthem} and theorem \ref{restinv}}~It is clear that if we apply theorem \ref{reglem} with $\epsilon_2>0$ and function $F(a,b)=a/b$ then in the decomposition $f=f_s+f_e+f_r$ the scalar product $(f,f_s)$ is arbitrarily close to $(f_s,f_s)$ and $\|f_s\|_{U_{k+1}}$ is arbitrarily close to $\|f\|_{U_{k+1}}$ if $\epsilon_2$ is small enough (depending only on $\epsilon$). This means by corollary \ref{l2becs} that $(f_s,f_s)\geq2\epsilon^{2^k}/3$ holds if $\epsilon_2$ is small and also $(f,f_s)\geq\epsilon^{2^k}/2$ holds simultaneously.

The inverse theorem is special families follows in the same way from theorem \ref{restreg}.

\subsection{Limit objects for convergent function sequences}

We start the chapter with an important observation.

\begin{lemma} Let $f\in L^\infty(\bA)$ and $M\in\mathcal{M}_i$ be a simple moment of degree $i$. Then $M(f)=M(\mathbb{E}(f|\mathcal{F}_i))$.
\end{lemma}

\begin{proof}\label{momentproj} Let $F=\{f_v\}_{v\in K_{i+1}}$ be a function system such that each $f_v$ is one of $1_{\bA},f,\overline{f}$ and let $F'=\{\mathbb{E}(f_v|\mathcal{F}_i)\}_{v\in K_{i+1}}$. Observe that for each $M\in\mathcal{M}_i$ there is a fuction system of this form such that $[F](0)=M(f)$. Then by the multilinearity of convolutions and lemma \ref{cornineq} we have that $[F](0)=[F'](0)=M(\mathbb{E}(f|\mathcal{F}_i))$.
\end{proof}

We continue with a few technical notions. We denote the $\sigma$-algebra $\vee_{i=1}^\infty\mathcal{F}_i$ by $\mathcal{F}$. We say that a $\sigma$-algebra $\mathcal{B}$ is a nil $\sigma$-algebra of infinite order if $\mathcal{B}=\vee_{i=1}^\infty[\mathcal{B}]_i$ and (\ref{rhsnil}) holds for every $i\in\mathbb{N}$.
Lemma \ref{weaknilprop} shows that in this case $[\mathcal{B}]_i=\mathcal{B}\cap\mathcal{F}_i$ for every $i\in\mathbb{N}$ and that $[\mathcal{B}]_i$ is a nil $\sigma$-algebra of order $i-1$ for every $i\in\mathbb{N}$.
The proof of lemma \ref{embednil} with minor modifications shows that every separable sub $\sigma$-algebra in $\mathcal{F}$ is contained in a separable nil $\sigma$-algebra of infinite order. 
First we prove theorem \ref{simplim} and corollary \ref{simplimcor}.

\bigskip

Let $\{f_i:A_i\rightarrow\mathbb{C}\}$ be a sequence of functions with $|f_i|\leq r$.
Let $\bA$ be the ultra product of $\{A_i\}_{i=1}^\infty$ and $f$ be the ultra limit of $\{f_i\}_{i=1}^\infty$.
We have that $M(f)=\lim_\omega M(f_i)=\lim M(f_i)$ for every moment $M$.

Let $g$ denote the projection of $f$ to the $\sigma$-algebra $\mathcal{F}$ and let $g_i=\mathbb{E}(f|\mathcal{F}_i)$. 
Using that $\mathbb{E}(f|\mathcal{F}_i)=\mathbb{E}(g|\mathcal{F}_i)=g_i$ holds for every $i\in\mathbb{N}$ we get by lemma \ref{momentproj} that $M(g)=M(f)=M(g_i)$ holds for every moment $M\in\mathcal{M}_i$.

Let $\mathcal{B}$ be a separable nil $\sigma$-algebra of infinite order such that $g$ is measurable in $\mathcal{B}$ and let $\mathcal{B}_i=[\mathcal{B}]_{i+1}=\mathcal{B}\cap\mathcal{F}_i$.
Using theorem \ref{main} we can create a sequence $\gamma_i:\bA\rightarrow N_i$ of nilspace factors generating the $\sigma$-algebra $\mathcal{B}_i$ in a way that these factors form an inverse system.
Let $\gamma:\bA\rightarrow N$ be the inverse limit of these factors. 

Since all the functions $g_i$ are measurable in the $\sigma$-algebra generated by $\gamma$ the function $g$ is also measurable in it. This means that there is a function $h:N\rightarrow\mathbb{C}$ such that $\gamma\circ h=g$.
Using the rooted measure preserving property of the factors $\gamma_i$ we have $M(g)=M(h)$ for every simple moment $M$.
This completes the proof of theorem \ref{simplim}. For corollary \ref{simplimcor} let $h_i=\mathbb{E}(h|N_i)$. It is clear that $h_i\circ\gamma=g_i$ and so $M(h_i)=M(g_i)=M(f)$ holds for every $M\in\mathcal{M}_i$.

\bigskip

The proof of theorem \ref{genlim} goes in a very similar way. The only difference is that we project all the functions $f^a\overline{f^b}$ with $a,b\in\mathbb{N}$ to $\mathcal{F}$ (resp. $\mathcal{F}_i$). The resulting function system $g^{a,b}$ (resp. $g^{a,b}_i$) at almost every point $x$ describes the complex moments of a probability distribution on the complex disc of radius $r$. Then we chose the separable nil $\sigma$-algebra $\mathcal{B}\subset\mathcal{F}$ so that each $g^{a,b}$ is measurable in $\mathcal{B}$. The rest of the proof is essentially the same.

\vskip 0.2in

\noindent
Bal\'azs Szegedy
\noindent
University of Toronto, Department of Mathematics,
\noindent
St George St. 40, Toronto, ON, M5R 2E4, Canada

\end{document}